\documentclass[11pt,oneside]{amsart}

\pdfoutput=1
\usepackage[centering]{geometry}                		
\usepackage{graphicx,amssymb,bm,tikz-cd}

\usepackage[shortlabels]{enumitem}

\usepackage[hyperfootnotes=false,psdextra]{hyperref}
\hypersetup{hypertexnames=false,colorlinks=true,citecolor=couleur_cite,linkcolor=couleur_link,urlcolor=couleur_url,
pdfstartview=FitH, pdfauthor=Valentin Blomer and Gergely Harcos and Peter Maga and Djordje Milicevic, pdftitle=Beyond the spherical sup-norm problem}

\usepackage{xcolor}
\definecolor{couleur_cite}{rgb}{0.05,.4,0.05}
\definecolor{couleur_link}{rgb}{0.05,0.05,0.4}
\definecolor{couleur_url}{rgb}{0.5,0,0}

\numberwithin{equation}{section}

\theoremstyle{definition}
\newtheorem{remark}{Remark}

\theoremstyle{plain}
\newtheorem{theorem}{Theorem}
\newtheorem{lemma}{Lemma}

\newcommand\diag{\mathop{\mathrm{diag}}}
\newcommand\dist{\mathop{\mathrm{dist}}}
\newcommand\rk{\mathop{\mathrm{rk}}}
\newcommand\sgn{\mathop{\mathrm{sgn}}}
\newcommand\tr{\mathop{\mathrm{tr}}}
\newcommand\vol{\mathop{\mathrm{vol}}}
\newcommand\End{\mathop{\mathrm{End}}}
\newcommand\Hom{\mathop{\mathrm{Hom}}}
\newcommand\meas{\mathop{\mathrm{meas}}}

\newcommand\cusp{\mathrm{cusp}}
\newcommand\Eis{\mathrm{Eis}}
\newcommand\fin{\mathrm{fin}}
\newcommand\dd{\mathrm{d}}
\newcommand\id{\mathrm{id}}
\newcommand\np{\mathrm{np}}
\newcommand\pp{\mathrm{p}}

\newcommand\GL{\mathrm{GL}}
\newcommand\PGL{\mathrm{PGL}}
\newcommand\SL{\mathrm{SL}}
\newcommand\OO{\mathrm{O}}
\newcommand\SO{\mathrm{SO}}
\newcommand\U{\mathrm{U}}
\newcommand\SU{\mathrm{SU}}
\newcommand\MM{\mathrm{M}}

\newcommand\NN{\mathbb{N}}
\newcommand\ZZ{\mathbb{Z}}
\newcommand\QQ{\mathbb{Q}}
\newcommand\RR{\mathbb{R}}
\newcommand\CC{\mathbb{C}}

\newcommand\mcD{\mathcal{D}}
\newcommand\mcI{\mathcal{I}}
\newcommand\mcH{\mathcal{H}}
\newcommand\mcL{\mathcal{L}}
\newcommand\mcM{\mathcal{M}}
\newcommand\mcN{\mathcal{N}}
\newcommand\mcS{\mathcal{S}}

\newcommand\mfa{\mathfrak{a}}
\newcommand\mfB{\mathfrak{B}}
\newcommand\mfg{\mathfrak{g}}

\newcommand\mtD{\mathtt{D}}
\newcommand\mtP{\mathtt{P}}

\newcommand\mbv{\mathbf{v}}
\newcommand\mbt{\mathbf{t}}

\newcommand\ov{\overline}
\newcommand\eps{\varepsilon}
\newcommand\one{\!\!\!\!\pmod{1}}
\newcommand\mupl{\mu_{\mathrm{Pl}}}
\newcommand\Gtemp{\widehat{G}_\mathrm{temp}}
\newcommand\Vcan{\mathcal{V}}

\renewcommand\leq{\leqslant}
\renewcommand\geq{\geqslant}

\DeclareFontFamily{U}{mathx}{\hyphenchar\font45}
\DeclareFontShape{U}{mathx}{m}{n}{
      <5> <6> <7> <8> <9> <10>
      <10.95> <12> <14.4> <17.28> <20.74> <24.88>
      mathx10
      }{}
\DeclareSymbolFont{mathx}{U}{mathx}{m}{n}
\DeclareFontSubstitution{U}{mathx}{m}{n}
\DeclareMathAccent{\widecheck}{0}{mathx}{"71}
\DeclareMathAccent{\wideparen}{0}{mathx}{"75}

\title{Beyond the spherical sup-norm problem}

\author{Valentin Blomer}
\address{Mathematisches Institut, Endenicher Allee 60, D-53115 Bonn, Germany}
\email{blomer@math.uni-bonn.de}

\author{Gergely Harcos, P\'eter Maga}
\address{Alfr\'ed R\'enyi Institute of Mathematics, POB 127, Budapest H-1364, Hungary}\email{gharcos@renyi.hu, magapeter@gmail.com}
\address{MTA R\'enyi Int\'ezet Lend\"ulet Automorphic Research Group}\email{gharcos@renyi.hu, magapeter@gmail.com}

\author{Djordje Mili\'cevi\'c}
\address{Bryn Mawr College, Department of Mathematics, 101 North Merion Avenue, Bryn Mawr, PA 19010, USA}
\email{dmilicevic@brynmawr.edu}

\thanks{This research was supported in part by the DFG-SNF Lead Agency Program grant BL 915/2-2 (V.B.), Germany's Excellence Strategy grant EXC-2047/1 - 390685813 (V.B.), National Science Foundation Grant DMS-1903301 (D.M.), European Research Council grant CoG-648017 (G.H.), the MTA R\'enyi Int\'ezet Lend\"ulet Automorphic Research Group (G.H. \& P.M.), NKFIH (National Research, Development and Innovation Office) grants K~119528 (G.H. \& P.M.), KKP~133819 (P.M.), FK~135218 (P.M.), and the Premium Postdoctoral Fellowship of the Hungarian Academy of Sciences (P.M.).}

\keywords{sup-norm problem, automorphic form, amplification, pre-trace formula, arithmetic hyperbolic 3-manifold, spherical transform, Paley--Wiener theorem, spherical function}
\subjclass[2020]{Primary 11F72; Secondary 11F41, 11F70, 11J25, 22E30, 43A90.}

\AtBeginDocument{%
   \def\MR#1{}
}

\begin{document}

\begin{abstract}
We open a new perspective on the sup-norm problem and propose a version for non-spherical Maa{\ss} forms when the maximal compact $K$ is non-abelian and the dimension of the $K$-type gets large.  We solve this problem for an arithmetic quotient of $G=\SL_2(\CC)$ with $K=\SU_2(\CC)$. Our results cover the case of vector-valued Maa{\ss} forms as well as all the individual scalar-valued Maa{\ss} forms of the Wigner basis, reaching  sub-Weyl exponents in some cases. On the way, we develop analytic theory of independent interest, including uniform strong localization estimates for generalized spherical functions of high $K$-type and a Paley--Wiener theorem for the corresponding spherical transform acting on the space of rapidly decreasing functions. The new analytic properties of the generalized spherical functions lead to novel counting problems of matrices close to various manifolds that we solve optimally.
\end{abstract}

\maketitle

\section{Introduction}

\subsection{The spherical sup-norm problem}
The sup-norm problem on arithmetic Riemannian manifolds is a question at the interface of harmonic analysis and  number theory that intrinsically combines techniques from both areas. Let $X = \Gamma \backslash G/K$ be a locally symmetric space of finite volume, where $\Gamma$ is an arithmetic subgroup. Arithmetically and analytically, the most interesting functions in $L^2(X)$ are joint eigenfunctions $\phi$ of all invariant differential operators and the Hecke operators: these are precisely the functions that arise from (spherical) automorphic forms.   The sup-norm problem asks for a quantitative comparison of the $L^2$-norm ${\|\phi\|}_2$ and the sup-norm ${\|\phi\|}_{\infty}$, most classically in terms of the Laplace eigenvalue $\lambda_{\phi}$, but depending on the application  also in terms of the volume of $X$ or other relevant quantities. Upper bounds for the sup-norm in terms of the Laplace eigenvalue are a measure for the equidistribution of the mass of  high energy eigenfunctions which  sheds light on the question to what extent these eigenstates can localize (``scarring''). Besides the quantum mechanical interpretation, the sup-norm problem in its various incarnations has connections to the multiplicity problem, zero sets and nodal lines of automorphic functions, and bounds for Faltings' delta function, to name just a few. See \cite{Sarnak2004Morawetz, Rudnick2005, GhoshReznikovSarnak2013, JorgensonKramer2004}.

If $X$ is  compact, the most general upper bound is due to Sarnak \cite{Sarnak2004Morawetz}:
\begin{equation}\label{sa}
{\|\phi\|}_{\infty} \ll_X  \lambda_{\phi}^{(\dim X - \rk X)/4}{\|\phi\|}_2,
\end{equation}
a bound which does not use the Hecke property and is in fact sharp (for general $X$) under these weaker assumptions. Sarnak derives this bound from asymptotics of spherical functions. A slightly different but ultimately related argument proceeds via a pre-trace inequality that bounds ${\|\phi\|}_{\infty}^2$ by a sum of an automorphic kernel over $\gamma \in \Gamma$. If the test function is an appropriate Paley--Wiener function, only the identity contributes to this sum, and one obtains as a (``trivial'') upper bound for ${\|\phi\|}_{\infty}$ the square-root of the spectral density as given in terms of the Harish-Chandra $\textbf{c}$-function. If the Langlands parameters of $\phi$ are in generic position, this coincides with \eqref{sa}.

To go beyond \eqref{sa}, one uses a test function that localizes   not only the archimedean Langlands parameters, but in addition the parameters at a large number of finite places (where ``large'' means a function tending to infinity as a small and carefully chosen power of  $\lambda_{\phi}$). This is called the amplification technique and leads, after estimating the automorphic kernel, to a problem in the geometry of numbers: count the elements of $G$ which appear in Hecke correspondences and lie in regions of $G$ according to the size of the kernel (such as counting rescaled integer matrices lying close to $K$). It has been implemented successfully in a variety of cases, see e.g.\ \cite{IwaniecSarnak1995, HarcosTemplier2013, BlomerPohl2016, BlomerMaga2016, Marshall2014a, Templier2015, Saha2017a, BlomerHarcosMagaMilicevic2020} and the references therein.

\subsection{Automorphic forms with \texorpdfstring{$K$}{K}-types}
In this paper we open a new perspective on the sup-norm problem and propose a version of higher complexity. The sup-norm problem makes perfect sense not only on the level of symmetric spaces, but also on the level of groups, and a priori there is no reason why one should restrict to spherical, i.e.\ right $K$-invariant automorphic forms. Let $\tau$ be an irreducible unitary representation of $K$ on some finite-dimensional complex vector space $V^{\tau}$, and consider the homogeneous vector bundle over $G/K$ defined by $\tau$. A cross-section may then be identified with a vector-valued function $f:G\to V^{\tau}$ which transforms on the right by $K$ with respect to $\tau$:
\[ f(gk) = \tau(k^{-1})f(g),\qquad g\in G,\quad k \in K. \]
It is now an interesting question to bound the sup-norm of $f$ or, more delicately, its components as the \emph{dimension}   of $V^{\tau}$ gets large. Such a situation cannot be realized in the classical case $G = \SL_2(\RR)$, since $K = \SO_2(\RR)$ is abelian, hence each $V^{\tau}$ is one-dimensional. In this paper, we offer a detailed investigation of the first nontrivial case $G = \SL_2(\CC)$. For concreteness, we choose the congruence lattice $\Gamma = \SL_2(\ZZ[i])$, although our results extend to more general arithmetic quotients of $G$ using the techniques in \cite{BlomerHarcosMagaMilicevic2020}.

Nontrivial irreducible unitary representations of $G$ are principal series representations parametrized by certain pairs $(\nu, p) \in \mfa^{\ast}_{\CC} \times \frac{1}{2}\ZZ$, where as usual $\mfa$ is the Lie algebra of the subgroup of positive diagonal matrices; see \S\ref{SL2C-subsec}.
(By a small abuse of notation we will later interpret $\nu$ simply as a complex number.) Each representation space $V$ of $G$ decomposes as a Hilbert space direct sum
\begin{equation}\label{decomp}
V = \bigoplus_{\substack{\ell\geq|p|\\\ell\equiv p\one}}  V^{\ell}
= \bigoplus_{\substack{\ell\geq|p|\\\ell\equiv p\one}} \ \bigoplus_{\substack{|q| \leq \ell\\q\equiv\ell\one}}V^{\ell,q},
\end{equation}
where $V^{\ell,q}$ is one-dimensional. Here and later, $\ell\in\frac{1}{2}\ZZ_{\geq 0}$ parametrizes the $K$-type, i.e.\ the $(2\ell+1)$-dimensional representation $\tau_{\ell}$ of $K$, and the diagonal matrix $\diag(e^{i\varrho}, e^{-i\varrho})\in K$ acts on $V^{\ell,q}$ by $e^{2qi\varrho}$. (The upper index $\ell$ in $V^\ell$ should not be mistaken for an $\ell$-th power.) 

Representations occurring in $L^2(\Gamma\backslash G)$ consist of even functions on $G$ and have $p\in\ZZ$. A representation  contains a spherical vector if and only if $p = 0$. In particular, the forms with $p\neq 0$ are untouched by any of the spherical sup-norm literature. For $p \neq 0$, no complementary series exists, so $\nu \in i\mfa^{\ast}$.

\subsection{Main results I: vector-valued forms}\label{main-results-intro-sec}
As explained above, we are interested in ``big'' $K$-types which occur for all representation parameters $|p| \leq \ell$, but arguably the most interesting case is when the $K$-type is ``new'' and no lower $K$-types appear in the same automorphic representation space. Hence from now on we restrict to $p = \ell$. The sup-norm problem for large $\nu$ was studied in detail in \cite{BlomerHarcosMagaMilicevic2020}, so here we keep $\nu$ in a fixed compact subset $I\subset i\RR$ and let $\ell$ vary.  The spectral density is a constant multiple of $p^2 - \nu^2$. In particular, for a  given $K$-type $\tau_{\ell}$, there are $\OO_I(\ell^2)$ cuspidal automorphic representations $V\subset L^2(\Gamma\backslash G)$ with spectral parameter $\nu\in I$ and $p = \ell$  (see \cite{dever2020ambient}), and in the light of the trace formula this bound is expected to be sharp. In each of these we consider the $(2\ell+1)$-dimensional subspace $V^{\ell}$. Let us choose an orthonormal basis $\{ \phi_q : |q| \leq \ell\}$ of $V^{\ell}$, with $\phi_q\in V^{\ell,q}$ as in \eqref{decomp}. The function $G\to\CC^{2\ell+1}$ given by
\begin{equation}\label{eq:vectorvalued}
g\mapsto\left(\phi_{-\ell}(g), \dotsc, \phi_\ell(g)\right)^{\top}
\end{equation}
is a vector-valued automorphic form for the group $\Gamma$ with spectral parameter $\nu$ and $K$-type $\tau_{\ell}$.
The Hermitian norm of this function,
\[\Phi(g):=\Bigl(\sum_{|q|\leq \ell} |\phi_q(g)|^2\Bigr)^{1/2},\qquad g\in G,\]
is independent of the choice of the orthonormal basis, and it satisfies ${\|\Phi\|}_2=(2\ell+1)^{1/2}$. Let us fix a compact subset $\Omega\subset G$. Our remarks on spectral density and dimension suggest that
\begin{equation}\label{3/2-exp}
{\| \Phi|_{\Omega} \|}_{\infty} :=\Bigl\| \sum_{|q| \leq \ell} |\phi_q|_{\Omega}|^2\Bigr\|^{1/2}_{\infty}\ll_{I,\Omega} \ell^{3/2}
\end{equation}
should be regarded as the ``trivial'' bound; this is made precise in Remark~\ref{non-arith} below. Our first result is a power-saving improvement.

\begin{theorem}\label{thm1} Let $\ell\geq 1$ be an integer, $I\subset i\RR$ and $\Omega\subset G$ be compact sets.
Let $V\subset L^2(\Gamma\backslash G)$ be a cuspidal automorphic representation with minimal $K$-type $\tau_{\ell}$ and spectral parameter $\nu_V\in I$. Then for any $\eps>0$ we have
\[ {\| \Phi|_{\Omega} \|}_{\infty} \ll_{\eps,I,\Omega} \ell^{4/3+\eps}. \]
\end{theorem}

We will explain some ideas of the proof in a moment, but we remark already at this point that the exponent is  the best possible, given that we sacrifice cancellation of the terms on the geometric side of the pre-trace formula and given our
current knowledge on the construction of the most efficient amplifier. In other words, under these conditions we solve the arising matrix counting problem optimally.
Since we trivially have ${\|\Phi\|}_{\infty}\gg\ell^{1/2}$, the above bound is one-sixth of the way   from the trivial down to the best possible exponent (absent the possibility of some escape of mass into a cusp). This matches (after a renormalization) the original and still the best available subconvexity exponent $5/24$ of Iwaniec--Sarnak~\cite{IwaniecSarnak1995} for the sup-norms of spherical Maa{\ss} forms of large Laplace eigenvalue on arithmetic hyperbolic surfaces.

\subsection{Main results II: individual vectors}
It is a much more subtle endeavor to investigate the sup-norm of the individual basis elements $\phi_q$. Here one must contend with the inherent high multiplicity, a known serious barrier in the sup-norm problem. Indeed, a straightforward construction \cite{Sarnak2004Morawetz} shows that some scalar-valued $L^2$-normalized form $\phi\in V^{\ell}$ (essentially the projection of the vector-valued form \eqref{eq:vectorvalued} in the modulus-maximizing direction) has sup-norm on $\Omega$ as large as ${\|\Phi|_{\Omega}\|}_{\infty}$ in Theorem~\ref{thm1}. However, our natural basis $\{\phi_q : |q| \leq \ell\}$ of $V^{\ell}$ is distinguished by consisting of eigenfunctions under the action of the group $\{\diag(e^{i\theta}, e^{-i\theta}) : \theta \in \RR\}$ of diagonal matrices in $K$. This is the classical basis with respect to which the representation $\tau_\ell$ is given by the Wigner $D$-matrix. By a similar heuristic reasoning as for \eqref{3/2-exp}, one might expect that the baseline bound should be ${\| \phi_q |_{\Omega} \|}_{\infty} \ll_{I,\Omega}\ell$. Indeed, we prove this bound in considerable generality up to a factor of $\ell^{\eps}$ (cf.\ Remark~\ref{non-arith} below), noting that it is not ``trivial'' in any sense other than that it does not require arithmeticity. Moreover, in the situation of Theorem~\ref{thm1}, we are in fact able to break this barrier uniformly for all $q$, as shown by the next theorem.
 
\begin{theorem}\label{thm2} Under the assumptions of Theorem~\ref{thm1}, we have
\[\max_{|q|\leq\ell}\, {\| \phi_q |_{\Omega} \|}_{\infty}  \ll_{\eps,I,\Omega} \ell^{26/27+\eps}.\]
\end{theorem}

For special values of $q$ we can improve on the exponent considerably. The central vector $\phi_{0}$ is distinguished as the ``archimedean newvector'' \cite{Popa2008} in the sense that its Whittaker function determines the archimedean $L$-factor of the underlying representation. Another interesting situation is the  extreme case of the vector $\phi_{\pm\ell}$.

\begin{theorem}\label{thm3} Keep the assumptions of Theorem~\ref{thm1}. 
\begin{enumerate}[(a)]
\item\label{thm3-a}
For $q = 0$ we have
\[ {\| \phi_{0} |_{\Omega} \|}_{\infty} \ll_{\eps,I,\Omega} \ell^{7/8+\eps}.\]
\item\label{thm3-b}
Suppose that $V$ lifts to an automorphic representation for $\PGL_2(\ZZ[i])\backslash\PGL_2(\CC)$. For $q = \pm \ell$ we have
\[{\| \phi_{\pm \ell} |_{\Omega} \|}_{\infty} \ll_{\eps,I,\Omega} \ell^{1/2+\eps}. \]
\end{enumerate}
\end{theorem}

The strong numerical saving in the case $q = \pm \ell$, going far beyond the Weyl exponent, is quite remarkable, in particular in view of the seemingly weaker saving in Theorem~\ref{thm1} which might be regarded as an easier case. We will discuss this in \S \ref{KS-intro}. The assumption that $V$ is associated to a representation of $\PGL_2$ rather than $\SL_2$ is only for technical simplicity and not essential to the  method, cf.\ \S \ref{RSSection}. This assumption holds if and only if the elements of $V$ are fixed by the Hecke operator $T_i$ (which is an involution on $L^2(\Gamma\backslash G)$).

\begin{remark} In the case of the spherical sup-norm problem, Sarnak~\cite{Sarnak2004Morawetz} put forward the purity conjecture that the accumulation points of the set
\[ \left\{\frac{\log {\|\psi\|}_{\infty}}{\log \lambda_{\psi}}:\text{$\psi$ is a joint eigenfunction}\right\} \]
lie in $\frac{1}{4}\ZZ$. It would be very interesting to see if an analogous conjecture may be expected in the $K$-aspect, and even if there may be examples exhibiting different layers of power growth as in \cite{Milicevic2011,Blomer2020,BrumleyMarshall2020}. In particular, the savings in Theorem~\ref{thm3} produce already a considerable ``exponent gap''.
\end{remark}

\begin{remark}\label{non-arith}
We record that our essentially best possible estimates on the spherical trace function in \S\ref{gen-sph-fun-intro-sec}, which are of purely analytic nature, coupled with the formalism of the pre-trace inequality, yield what might be considered ``trivial'' geometric estimates: for any co-finite Kleinian subgroup $\Gamma\leq G$, without any arithmeticity assumption, we have
\[ {\|\Phi|_{\Omega}\|}_{\infty}\ll_{I,\Omega,\Gamma}\ell^{3/2}\qquad\text{and}\qquad
\max_{|q|\leq\ell}\,{\|\phi_q|_{\Omega}\|}_{\infty}\ll_{\eps,I,\Omega,\Gamma}\ell^{1+\eps} \]
for any $L^2$-normalized vector-valued Maa{\ss} eigenform $(\phi_{-\ell},\dots,\phi_{\ell})^{\top}$ with spectral parameter $\nu\in I$ and $K$-type $\tau_{\ell}$ (with $\phi_q\in V^{\ell,q}$ as before).
\end{remark}

Our Theorems~\ref{thm1}--\ref{thm3} above, and the non-spherical sup-norm problem in general, come with several novelties of representation theoretic, analytic and arithmetic nature that we discuss briefly in the following subsections.

\subsection{Generalized spherical functions}\label{gen-sph-fun-intro-sec}
The classical pre-trace formula features on the geometric side the Harish-Chandra transform $\widecheck{h}$ of the test function $h$ on the spectral side. This transform is a bi-$K$-invariant function obtained by integrating $h$ against the elementary spherical functions (which themselves are bi-$K$-invariant, and hence in the case of $G = \SL_2(\CC)$ simply a function of one real variable). In typical applications there is no cancellation in this integral, so an asymptotic analysis of spherical functions is the first key step (see \cite{BlomerPohl2016} for a general result in this direction). Our set-up requires a generalized version for homogeneous vector bundles over $G/K$. For $G = \SL_2(\CC)$, the corresponding \emph{spherical trace function} equals (see \S \ref{section:sphericaltransform} for details)
\begin{equation}\label{spherical-def}
\varphi_{\nu,\ell}^{\ell}(g) =
(2\ell+1)\int_K \psi_{\ell}(\kappa(k^{-1} g k))\,e^{(\nu-1)\rho(H(gk))}\,\dd k,
\end{equation}
where $\dd k$ is the probability Haar measure on $K$, $\rho$ is the unique positive root, $\kappa$ (resp.\ $H$) is the $KAN$ Iwasawa projection onto $K$ (resp.\ $\mfa$), and
\begin{equation}\label{chi-ell}
\psi_{\ell}\left( \begin{pmatrix} \alpha  &\beta \\ - \bar{\beta} & \bar{\alpha} \end{pmatrix} \right) := \bar{\alpha}^{2\ell}, \qquad \left( \begin{matrix} \alpha  &\beta \\ - \bar{\beta} & \bar{\alpha} \end{matrix} \right) \in K.
\end{equation}
The trivial bound is $|\varphi_{\nu,\ell}^{\ell}(g)|\leq 2\ell+1$, which is sharp for $g = \pm\id$, and the key question is how quickly $\varphi_{\nu,\ell}^{\ell}(g)$ decays, uniformly in $\ell$, as $g\in G$ moves away from $\pm\id$. We observe that $\varphi_{\nu,\ell}^{\ell}(g)$ is  invariant under conjugation by $K$, hence it suffices to investigate it for upper triangular matrices $g\in G$. We shall use the Frobenius norm
$\|g\|:= \sqrt{\tr(gg^*)}$, and we note that for $g\in G$ this is always at least $\sqrt{2}$. The following bound is new and most likely sharp for fixed $\nu\in i\RR$ (up to factors $\ell^{\eps}$ and powers of $\| g \|$, which we did not try to optimize).

\begin{theorem}\label{thm4} Let $\ell\geq 1$ be an integer, and let $g = \left(\begin{smallmatrix} z & u\\ & z^{-1}\end{smallmatrix}\right) \in G$ be upper triangular. Then for any $\nu\in i\RR$, $k \in K$, $\eps>0$, we have
\[ \varphi_{\nu,\ell}^{\ell}(k^{-1}gk) \ll_{\eps} 
\min\bigg(\ell, \frac{\ell^\eps\|g\|^6}{|z^2 - 1|^2}, \frac{\ell^{1/2+\eps}\|g\|^3}{|u|}\bigg). \]
\end{theorem}

The proof shows that the factors $\ell^{\eps}$ can be replaced with a suitable power of $\log 2\ell$. The same remark applies to Theorems~\ref{thm6} and \ref{thm5} below.
 
The spherical trace function $\varphi_{\nu,\ell}^{\ell}$ can be used to analyze the vector-valued function
\eqref{eq:vectorvalued}. It is, unfortunately, unable to identify the individual components $\phi_q$, and there does not seem to exist a general theory of spherical functions covering such cases. As the components are eigenfunctions of the action of the diagonal elements, we can single out $\phi_q$ by considering
\begin{equation}\label{spherical-averaged}
\varphi_{\nu,\ell}^{\ell,q}(g):=\frac1{2\pi}\int_0^{2\pi}\varphi_{\nu,\ell}^{\ell}\left(g\diag(e^{i\varrho},e^{-i\varrho})\right)\,e^{-2qi\varrho}\,\dd \varrho.
\end{equation}
The function $\varphi_{\nu,\ell}^{\ell,q}$ is an interesting object that does not seem to have been considered before. It is not conjugation invariant anymore, so it needs to be analyzed on the entire $6$-dimensional group $G = \SL_2(\CC)$, and little preliminary reduction is possible. When restricted to $K$, it is not hard to see that $\varphi_{\nu,\ell}^{\ell,q}(k)$, for $k = k[u,v,w]\in K$ written in terms of Euler angles (cf.\ \eqref{decomp-K}), is essentially a Jacobi polynomial in $\cos 2v$. We refer to \S \ref{thm5a-proof-sec} for a more detailed discussion. In particular,
$\varphi_{\nu,\ell}^{\ell,q}(\pm\id)=1$. Therefore, at least heuristically, a safe baseline bound should be
\begin{equation}\label{trivial-q}
\varphi_{\nu,\ell}^{\ell,q}(g) \ll_\eps \ell^\eps.
\end{equation}
Unlike in the bi-$K$-invariant case, where the trivial bound is just an application of the triangle inequality and hence is indeed trivial, the expected baseline bound \eqref{trivial-q} turns out to be hard to prove. It requires very strong cancellation in the $\varrho$-integral, along with the decay properties of $\varphi_{\nu,\ell}^{\ell}$.  Taking \eqref{trivial-q} for granted, we wish to investigate in what directions and with what speed we can identify decay as we move away from $\pm\id\in G$. Interestingly, this is extremely sensitive to the value of $q$.

Let $\mcD\subset G$ be the set of diagonal matrices, $\mcS$ the normalizer of $A$ in $K$ (which consists of the diagonal and the skew-diagonal matrices lying in $K$), and
\begin{equation}\label{adbc}
\mcN:=\left\{\begin{pmatrix}a & b\\ c & d \end{pmatrix}\in G: |a| = |d|, \ |b| = |c|\right\}.
\end{equation}
It is clear that  $\mcS\subset K\subset\mcN\subset G$. For $g \in G$ and non-empty $\mcH \subset G$, we shall write $\dist(g,\mcH)$ for their distance $\inf_{h\in\mcH}\|g-h\|$. For later reference, we note that $\|g-h\|=\|g^{-1}-h^{-1}\|$, hence also
\begin{equation}\label{distinvariance}
\dist(g,\mcH)=\dist(g^{-1},\mcH^{-1}).
\end{equation}
As an alternative to $\dist(g,\mcN)$, we shall also use
\begin{equation}\label{Dgdef}
D(g):=\left||a|^2-|d|^2\right|+\left||b|^2-|c|^2\right|.
\end{equation}
For orientation, we remark the elementary inequality
\[ \dist(g,\mcN)^2\leq D(g)\leq 2\|g\|\dist(g,\mcN). \]

In the following theorem, we show that  $\varphi_{\nu,\ell}^{\ell,q}(g)$ decays away from $K$ and $\mcD$ in generic ranges, for all $|q|\leq\ell$, and with considerable uniformity.

\begin{theorem}\label{thm6} Let $\ell,q\in\ZZ$ be such that $\ell\geq\max(1,|q|)$. Let $\nu\in i\RR$ and $g \in G$.
Then for any $\eps>0$ and $\Lambda>0$, we have
\begin{equation}\label{thm6bound}
\varphi_{\nu,\ell}^{\ell,q}(g) \ll_{\eps,\Lambda}\ell^{\eps}\min\left(1,\frac{\| g \|}{\sqrt{\ell}\dist(g,K)^2\dist(g,\mcD)}\right) + \ell^{-\Lambda}.
\end{equation}
\end{theorem}

The proof of Theorem~\ref{thm6} uses a soft argument that provides some decay for all $|q|\leqslant\ell$, despite the substantial dependence of $\varphi_{\nu,\ell}^{\ell,q}$ on this parameter. In the special case $q\in\{-\ell,0,\ell\}$, we use more elaborate arguments for stronger bounds.

\begin{theorem}\label{thm5} Let $\ell\geq 1$ be an integer, $\nu\in i\RR$ and $g\in G$. Let $\eps>0$ and $\Lambda>0$ be two parameters.
\begin{enumerate}[(a)]
\item\label{thm5-a}
We have
\begin{equation}\label{thm5boundq=0}
\varphi_{\nu,\ell}^{\ell,0}(g) \ll_{\eps,\Lambda}
\ell^{\eps}\min\left(1, \frac{1}{\sqrt{\ell} \dist(g, \mcS)}\right)+\ell^{-\Lambda}.
\end{equation}
Moreover, $\varphi_{\nu,\ell}^{\ell,0}(g) \ll_{\Lambda} \ell^{-\Lambda}$ holds unless $D(g)\ll_\Lambda\|g\|^2(\log\ell)/\sqrt{\ell}$.
\item\label{thm5-b}
We have
\begin{equation}\label{thm5boundq=ell}
\varphi_{\nu,\ell}^{\ell,\pm \ell}(g) \ll_{\eps} \| g \|^{-2+\eps} \ell^{\eps}.
\end{equation}
Moreover, $\varphi_{\nu,\ell}^{\ell,\pm \ell}(g) \ll_{\Lambda} \ell^{-\Lambda}$ holds unless
$\dist(g, \mcD) \ll_\Lambda\| g \|\sqrt{\log\ell}/\sqrt{\ell}$.
\end{enumerate}
\end{theorem}

We expect that the bounds in Theorem~\ref{thm5} are essentially best possible, possibly up to powers of $\ell^{\eps}$ and $\| g \|$. The proof requires detailed analysis that could in principle be applied to all values of $q$ and would detect, for instance, further Airy-type bumps in certain regions and for certain choices of parameters.

\begin{remark}\label{remark3}
Less precise results but in a more general setting were obtained by Ramacher~\cite{Ramacher2018} using operator theoretical methods. Combined with an argument of Marshall~\cite{Marshall2014a}, these were applied by Ramacher--Wakatsuki~\cite{RamacherWakatsuki2017a} to the sup-norm problem with $K$-types. For compact arithmetic quotients of $\SL_2(\CC)$, and for $\phi\in V^{\ell}$ as before,
\cite[Th.~7.12]{RamacherWakatsuki2017a} yields ${\| \phi \|}_{\infty} \ll \ell^{5/2 - \delta}$ with an unspecified
constant $\delta>0$; this does not even recover the baseline bound.
\end{remark}

\subsection{Paley--Wiener theory}
For a reductive Lie group $G$, Paley--Wiener theory characterizes the image of $C_c^{\infty}(G)$ under the Harish-Chandra transform. For bi-$K$-invariant functions, this is a famous result of Gangolli~\cite{Gangolli1971}: the image consists of entire, Weyl group invariant functions satisfying certain growth conditions. For general $K$-finite functions, the picture is much more complicated: any linear relation that holds for the matrix coefficients of generalized principal series also needs to hold for the matrix coefficients of the operator-valued Fourier transform (and hence for the $\tau$-spherical transforms for $\tau\in\widehat{K}$). A complete list of these ``Arthur--Campoli relations'' requires a full knowledge of all the irreducible subquotients of the non-unitary principal series, which in general is not available. Arthur~\cite{Arthur1983} describes them as a sequence of successive residues of certain meromorphic functions; see also \cite{Campoli1979}. Needless to say, a good knowledge of available functions on the spectral side is crucial for the quantitative analysis of the pre-trace formula in the sup-norm problem.

For the case of $G = \SL_2(\CC)$, in a somewhat neglected paper, Wang~\cite{Wang1974} devised an elegant argument to establish a completely explicit Paley--Wiener theorem for the $\tau_{\ell}$-spherical transform acting on $C_c^{\infty}(G)$: in addition to the Weyl group symmetry, we have the additional symmetry $(\nu, p) \leftrightarrow (p, \nu)$ whenever $\nu \equiv p\pmod{1}$ and $|\nu|, |p| \leq \ell$; see Theorem~\ref{thm10} in \S \ref{section:sphericaltransform}. The additional symmetry is counter-intuitive at first (the pairs $(\nu,p)\neq(0,0)$ satisfying $\nu\equiv p\pmod{1}$ correspond to a discrete set of non-unitary representations), but it enters the picture as it fixes the eigenvalues $\nu^2+p^2$ and $\nu p$ of two generators of $Z(\mathcal{U}(\mfg))$, and hence the infinitesimal character. See \cite[Cor.~2]{Wang1974} and its proof. A more conceptual explanation, along the lines of irreducible subquotients, can be found after \eqref{eq:tau-ell-isotypical-decomposition-algebraic}. Wang's remarkable result is that these are \emph{all} relations.

The extra symmetry makes the application of the pre-trace formula more delicate. For instance, it appears impossible to single out an individual value of $p$ by a manageable test function on the spectral side. We circumvent this problem by employing a carefully chosen Gaussian \eqref{eq:def-gaussian-spectral-weight} that at least asymptotically singles out our preferred value $p=\ell$. The price to pay for this maneuver is that we lose compact support.  As a result of independent interest, we prove a new Paley--Wiener theorem for $K$-finite Schwartz class functions on $G = \SL_2(\CC)$. For the notation, see  \S \ref{section:sphericaltransform}.

\begin{theorem}\label{thm:pws} For $f\in\mcH(\tau_{\ell})$, the following two conditions are equivalent (with implied constants depending on $f$).
\begin{enumerate}[(a)]
\item\label{pws-a} The function $f(g)$ is smooth, and for any $m\in\ZZ_{\geq 0}$ and $A>0$ we have
\begin{equation}\label{eq:mAbound}
\frac{\partial^m}{\partial h^m}f(k_1 a_h k_2)\ll_{m,A} e^{-A|h|},\qquad h\in\RR,\quad k_1,k_2\in K.
\end{equation}
\item\label{pws-b} The function $\widehat{f}(\nu,p)$ extends holomorphically to $\CC\times\tfrac12\ZZ$ such that
\begin{equation}\label{eq:symmetry}\widehat{f}(\nu,p)=\widehat{f}(p,\nu),\qquad\nu\equiv p\one,\quad|\nu|,|p|\leq\ell,
\end{equation}
and for any $B,C>0$ we have
\begin{equation}\label{eq:BCbound}
\widehat{f}(\nu,p)\ll_{B,C} (1+|\nu|)^{-C},\qquad|\Re\nu|\leq B,\quad p\in\tfrac12\ZZ.
\end{equation}
\end{enumerate}
\end{theorem}

The Schwartz space offers a lot more flexibility in applications. A less precise result for more general groups is given in \cite[Th.~3]{DelormeFlensted-Jensen1991}, and we refer the reader to the introduction of that paper for additional discussion and motivation of Paley--Wiener type theorems for rapidly decaying functions.

\subsection{Beyond the pre-trace formula: a fourth moment}\label{KS-intro}
We still owe an explanation for the sub-Weyl exponent in Theorem~\ref{thm3}\ref{thm3-b}, where $q = \pm \ell$. The proof of this bound is different from the other results: it is inspired by a brilliant recent idea of Steiner and Khayutin--Steiner~\cite{St, KhayutinSteiner2020} in the \emph{weight} aspect for the groups $\SO_3(\RR)$ and $\SL_2(\RR)$. The starting point is the desire to choose the amplifier so long that it works as self-amplification. In this way, the amplifier can be made independent of the well-known but inefficient trick of using the Hecke relation $\lambda_p^2 - \lambda_{p^2} = 1$. A self-amplified second moment is in effect a fourth moment, and the key observation is that it can be realized as the diagonal term in a \emph{double} pre-trace formula. This only has a chance to work if the corresponding geometric side can be analyzed sufficiently accurately, and to this end, two extra features are necessary: a special behavior of spherical functions with rapid decay conditions (such as, for instance, the Bergman kernel for $\SL_2(\RR)$) and the possibility for a \emph{second moment} count on the geometric side, i.e.\ pairs of matrices, in a best possible way.

For the proof of Theorem~\ref{thm3}\ref{thm3-b}, we implement this idea for the first time in the context of \emph{principal series representations}. Our proof proceeds differently than both of \cite{St} and  \cite{KhayutinSteiner2020}. We avoid the theta correspondence and instead detect the diagonal term in the double pre-trace formula by an argument that is reminiscent of the Voronoi formula for  Rankin--Selberg $L$-functions over $\QQ[i]$, cf.\ \S \ref{sec28}. As we lose positivity, we have to use the full power of the pre-trace formula, unlike our other results where the softer pre-trace inequality suffices. The argument is analytically subtle, since we also lose the possibility to choose the test function in the pre-trace formula freely:  part of it is now  given to us by the gamma kernel in the Voronoi summation formula (one of several new features compared to \cite{St} and \cite{KhayutinSteiner2020}). At this point we need a very precise understanding of the Harish-Chandra transform in Theorem~\ref{thm:pws} with complete uniformity in the auxiliary complex parameters, and the reader may observe that in the end only the strong $g$-dependence in \eqref{thm5boundq=ell} saves the final bound.

\subsection{Matrix counting}
Having discussed some of the analytic and representation theoretic novelties, we finally comment briefly on the arithmetic part. In all previous instances of the sup-norm problem, the analysis of the geometric side of the pre-trace formula  amounts to counting matrices close to $K$, because the elementary spherical function is bi-$K$-invariant and decays away from $K$. Given the results on spherical trace functions in \S \ref{gen-sph-fun-intro-sec}, it is clear that from an arithmetic point of view the sup-norm problem with big $K$-types is conceptually very different from the spherical sup-norm problem.

The localization behavior of generalized spherical functions has distinct features as reflected by Theorems~\ref{thm4} and \ref{thm5}. The spherical trace function $\varphi_{\nu,\ell}^{\ell}$ concentrates close to the identity. The functions $\varphi_{\nu,\ell}^{\ell,\pm\ell}$ localize sharply around diagonal matrices (but not necessarily within $K$). For $\varphi_{\nu,\ell}^{\ell,0}$, there is localization on diagonal and skew-diagonal matrices within $K$, then there is a gradual transition to a second layer in a neighborhood of the 4-dimensional manifold $\mcN$ defined by \eqref{adbc}, and outside this neighborhood we see sharp decay. Theorem~\ref{thm6} is in some sense a combination of these two extreme cases. Correspondingly, the counting techniques in \S\S \ref{thm1-proof-sec}--\ref{sec-proof2} are still based on the geometry of numbers, but they differ conceptually and technically from the earlier treatment of the spherical sup-norm problem. In particular, as mentioned in \S \ref{KS-intro}, for the proof of Theorem~\ref{thm3}\ref{thm3-b}  we have to achieve a best possible double matrix count, cf.\ Lemma~\ref{lemma-ell-count}.

\subsection{Notation}
The group $G=\SL_2(\CC)$ and its arithmetic subgroup $\Gamma=\SL_2(\ZZ[i])$ are fixed throughout the paper. We use the $\eps$-convention in that $\eps>0$ denotes a number that may be different from line to line but may in each instance be taken to be as small as desired. As usual, we write $f\ll g$ or $f=\OO(g)$ to denote that $|f|\leqslant Cg$, where the implied constant $C>0$ may be different from line to line; it is absolute unless otherwise indicated by a subscript, except that we occasionally allow it to depend on the (fixed) quantities $I$ and $\Omega$ as well as on $\eps$. We also write $f\asymp g$ for $f\ll g\ll f$, and, when used as an asymptotic notation, $f\sim g$ for $\lim f/g=1$, where the direction of the limit is clear from the context.

\subsection{Acknowledgements}
This work began during D.M.'s term as Director's Mathematician in Residence at the Budapest Semesters of Mathematics program in the summer of 2018; D.M. would like to thank BSM, the Alfr\'ed R\'enyi Institute of Mathematics, as well as the Max Planck Institute for Mathematics for their hospitality and excellent working conditions.

\section{Preliminaries}

\subsection{Representations of \texorpdfstring{$\SU_2(\CC)$}{SU(2,C)}}\label{SU2-subsec}
In this subsection, we review the representation theory of the maximal compact subgroup
\[ K=\SU_2(\CC)=\left\{ k[\alpha,\beta]:=\begin{pmatrix}\alpha&\beta\\-\bar{\beta}&\bar{\alpha}\end{pmatrix}:|\alpha|^2+|\beta|^2=1\right\} \]
of $G=\SL_2(\CC)$.  We use \cite[\S2.1.1,2.2]{Lokvenec-Guleska2004} as a convenient reference.

For $u,v,w\in\RR$, we parametrize $K$ using essentially Euler angles $(2u,2v,2w)$ as follows:
\begin{equation}\label{decomp-K}
k[u,v,w]:=\begin{pmatrix}e^{iu}&\\&e^{-iu}\end{pmatrix}\begin{pmatrix}\cos v&i\sin v\\i\sin v&\cos v\end{pmatrix}\begin{pmatrix}e^{iw}&\\&e^{-iw}\end{pmatrix}.
\end{equation}
Generating an equivalence relation $\sim$ on $\RR^3$ by
\begin{equation}\label{angleequiv}
(u,v,w)\ \sim\ (u+2\pi,v,w),\ (u,v,w+2\pi),\ (u+\pi,v+\pi,w),\ (u+\pi/2,-v,w-\pi/2)
\end{equation}
we may parametrize $\SU_2(\CC)$ by $\RR^3/\!\sim$, or by a specific fundamental domain such as $[0,\pi)\times[0,\pi/2]\times[-\pi,\pi)$, in which each point in $\SU_2(\CC)$ has exactly one pre-image other than those with $v\in\frac{\pi}2\ZZ$. The probability Haar measure on $\SU_2(\CC)$ is given by
\begin{equation}\label{dk}
\dd k=(2\pi^2)^{-1}\sin 2v\,\dd u\,\dd v\,\dd w.
\end{equation}

The irreducible representations of $K=\SU_2(\CC)$ are classified as $(2\ell+1)$-dimensional representations $\tau_{\ell}$, for $\ell\in\frac12\ZZ_{\geq 0}$, described explicitly as the space $V_{2\ell}$ of polynomials of degree at most $2\ell$, with a basis given by $\{z^{\ell-q}:|q|\leq\ell,\,q\equiv\ell\pmod{1}\}$ and $\SU_2(\CC)$ action given by
\begin{equation}\label{matrix-coeff}
\tau_{\ell}(k[\alpha,\beta])z^{\ell-q}=(\alpha z-\bar{\beta})^{\ell-q}(\beta z+\bar{\alpha})^{\ell+q}=\sum_{\substack{|p|\leq\ell\\p\equiv\ell\one}}\Phi_{p,q}^{\ell}(k[\alpha,\beta])z^{\ell-p}.\end{equation}
A $K$-invariant scalar product on $V_{2\ell}$ is given by $(z^{\ell-q},z^{\ell-p})=(\ell-q)!(\ell+q)!\delta_{q=p}$, so that $\Phi_{p,q}^{\ell}$ are (unnormalized) matrix coefficients of $\tau_{\ell}$. Moreover,
\[\left\{\Phi_{p,q}^{\ell}\,:\,\text{$p,q,\ell\in\tfrac12\ZZ$ and $|p|,|q|\leq\ell$ and $p,q\equiv\ell\one$}\right\}\]
is an orthogonal basis of $L^2(K)$. In harmony with \cite[\S4.4.2]{Warner1972a}, we denote by $\xi_{\ell}$ the character of $\tau_{\ell}$, by $d_{\ell}=2\ell+1$ the dimension of $\tau_\ell$, and by $\chi_{\ell}=d_{\ell}\xi_{\ell}$ the normalized character of $\tau_\ell$. Finally, we denote by $\widehat{K}=\{\tau_{\ell}:\ell\in\frac12\ZZ_{\geq 0}\}$ the unitary dual of $K$.

\subsection{Representations of \texorpdfstring{$\SL_2(\CC)$}{SL(2,C)}}\label{SL2C-subsec}
For compatibility with the existing literature, we shall use the Iwasawa decomposition of $G=\SL_2(\CC)$ in two forms, $G=NAK$ and $G=KAN$, where $N$ (resp.\ $A$) is the subgroup of unipotent upper-triangular (resp.\ positive diagonal) matrices, and $K=\SU_2(\CC)$ is the standard maximal compact subgroup.

We fix a Haar measure on $G$ by setting
\[
\dd g=|\dd z|\frac{\dd r}{r^5}\dd k\quad\text{for } g=\begin{pmatrix}1&z\\&1\end{pmatrix}\begin{pmatrix}r&\\&r^{-1}\end{pmatrix}k, \quad z\in\CC,\,\,r>0,\,\,k\in K,
\]
where $|\dd z|=\dd x\,\dd y$ for $z=x+iy$, $x,y\in\RR$, and $\dd k$ is as in \eqref{dk}.

We write $\mfa\simeq\RR$ for the Lie algebra of $A$,  $\rho$ for the root on $\mfa$ mapping
$\big(\begin{smallmatrix}x&\\&-x\end{smallmatrix}\big)$ to $2x$,
$\exp:\mfa\to A$ for the exponential map, and $\kappa:G\to K$ and $H:G\to\mfa$ for
the projection and height maps defined by $g\in\kappa(g)\exp(H(g)) N$ for every $g\in G$.
Thus explicitly, for $g=\left(\begin{smallmatrix}a&b\\c&d\end{smallmatrix}\right)\in G$ we have
\begin{equation}\label{eq:kappaH}
\kappa(g)=\begin{pmatrix}a/\sqrt{|a|^2+|c|^2}&\ast\\c/\sqrt{|a|^2+|c|^2}&\ast\end{pmatrix},\quad \exp(H(g))=\begin{pmatrix}\sqrt{|a|^2+|c|^2}&\\&1/\sqrt{|a|^2+|c|^2}\end{pmatrix}.
\end{equation}
Finally, let $M\simeq S^1$ be the centralizer of $A$ in $K$, which consists of diagonal matrices in $K$.

Following \cite[Ch.~III]{GGV}, we introduce for every pair $(\nu,p)\in\CC\times\frac12\ZZ$
the (generalized) principal series representation $\pi_{\nu,p}$. Let us denote by $C^\infty(\CC)$ the set of functions $\CC\to\CC$ that are smooth when regarded as functions $\RR^2\to\CC$. The representation space $V_{\nu,p}$ consists of those functions $v\in C^\infty(\CC)$ for which the transformed functions
\begin{equation}\label{transformedfunctions}
\pi_{\nu,p}\left(\begin{pmatrix}a&b\\c&d\end{pmatrix}\right)v(z)
=|bz+d|^{2p+2\nu-2}(bz+d)^{-2p}v\left(\frac{az+c}{bz+d}\right),\qquad
\begin{pmatrix}a&b\\c&d\end{pmatrix}\in G,
\end{equation}
extend to elements of $C^\infty(\CC)$. The above display then actually defines the representation $\pi_{\nu,p}:G\to\GL(V_{\nu,p})$. The space $V_{\nu,p}$ is complete with respect to the countable family of seminorms
\[\sup\bigl\{\bigl|v^{(a,b)}(x+yi)\bigr|+\bigl|\widehat{v}^{(a,b)}(x+yi)\bigr|:x^2+y^2\leq c\bigr\},\qquad (a,b,c)\in\NN^3,\]
where we abbreviate $\widehat{v}:=\pi_{\nu,p}\left(\big(\begin{smallmatrix}&-1\\1&\end{smallmatrix}\big)\right)v$ for $v\in V_{\nu,p}$. The action of $G$ is continuous in the topology induced by these seminorms; thus, $\pi_{\nu,p}$ is a Fr\'echet space representation.

Using the action of $K=\SU_2(\CC)$ and its diagonal subgroup $\left\{\diag(e^{i\varrho},e^{-i\varrho}):\varrho\in\RR\right\}$, we can decompose the $K$-finite part of $V_{\nu,p}$ into an \emph{algebraic direct sum} of finite-dimensional subspaces and further into one-dimensional subspaces:
\begin{equation}\label{eq:tau-ell-isotypical-decomposition-algebraic}
V_{\nu,p}^{\text{$K$-finite}} = \bigoplus_{\substack{\ell\geq|p|\\\ell\equiv p\one}}V_{\nu,p}^{\ell}
= \bigoplus_{\substack{\ell\geq|p|\\\ell\equiv p\one}}
\ \bigoplus_{\substack{|q|\leq\ell\\q\equiv\ell\one}}V_{\nu,p}^{\ell,q}.
\end{equation}
Precisely, $V_{\nu,p}^\ell$ is a $(2\ell+1)$-dimensional subspace on which $\pi_{\nu,p}|_K$ acts by $\tau_\ell\in\widehat{K}$.

If $\nu\not\equiv p\pmod{1}$ or $|\nu|\leq|p|$, then $\pi_{\nu,p}\simeq\pi_{-\nu,-p}$ is irreducible, and these are all the equivalences among the representations $\pi_{\nu,p}$. If $\nu\equiv p\pmod{1}$ and $|\nu|>|p|$, then $\pi_{\nu,p}$ and $\pi_{-\nu,-p}$ are reducible. Assume $\nu>0$, say. Then the sum of $V_{\nu,p}^{\ell}$ with $|p|\leq\ell<\nu$ is a closed invariant subspace of $V_{\nu,p}$, and the representation induced on the quotient is irreducible. The closure of the sum of $V_{-\nu,-p}^{\ell}$ with $\ell\geq\nu$ is an invariant subspace of $V_{-\nu,-p}$, and the representation induced on it is irreducible. Both of these representations of $G$ are isomorphic to $\pi_{p,\nu}\simeq\pi_{-p,-\nu}$. This observation will become relevant in \eqref{eq:spherical-function-symmetry-2} below.

The space $V_{\nu,p}$ has a $G$-invariant Hermitian inner product if and only if $\nu\in i\RR$, or $p=0$ and $\nu\in(-1,0)\cup(0,1)$. In the first case, we say that $\pi_{\nu,p}$ belongs to the (tempered) unitary principal series. In the second case, we say that $\pi_{\nu,p}$ belongs to the (non-tempered) complementary series. In either case, the Fr\'echet space representation $\pi_{\nu,p}$ induces an irreducible unitary representation on the Hilbert space completion
$\widehat{V_{\nu,p}}$ that we shall still denote by $\pi_{\nu,p}$. The only equivalences among these unitary representations are $\pi_{\nu,p}\simeq\pi_{-\nu,-p}$. The equivalence classes, along with the trivial representation, form the unitary dual $\widehat{G}$ of $G$.

For $\pi\simeq\pi_{\nu,p}\in\widehat{G}$ we write
\[V_\pi:=\widehat{V_{\nu,p}},\qquad V_\pi^\ell:=V_{\nu,p}^\ell,\qquad V_\pi^{\ell,q}:=V_{\nu,p}^{\ell,q},\]
and then \eqref{eq:tau-ell-isotypical-decomposition-algebraic} is equivalent to the orthogonal Hilbert space decomposition
(cf.~\eqref{decomp}):
\[V_\pi = \bigoplus_{\substack{\ell\geq|p|\\\ell\equiv p\one}}V_\pi^{\ell}
= \bigoplus_{\substack{\ell\geq|p|\\\ell\equiv p\one}}
\ \bigoplus_{\substack{|q|\leq\ell\\q\equiv\ell\one}}V_\pi^{\ell,q}.\]
The projection $V_\pi\to V_{\pi}^{\ell}$ is realized by the operator
\begin{equation}\label{eq:projectionbychi}
\pi(\ov{\chi_\ell}):=\int_K \ov{\chi_\ell}(k)\pi(k)\,\dd k\in \End(V_\pi),
\end{equation}
where $\End(V_\pi)$ denotes the Hilbert space of Hilbert--Schmidt operators on $V_\pi$ endowed with the Hilbert--Schmidt norm. This leads to the ``block matrix decomposition''
\begin{equation}\label{eq:block-decomposition}
\End(V_\pi)=\bigoplus_{\substack{m,n\geq|p|\\ m,n\equiv p\one}}\Hom(V_{\pi}^{m},V_{\pi}^{n}),
\end{equation}
where the direct sum is meant in the Hilbert space sense. Hence, for $f\in C_c(G)$, the $(m,n)$-component of the Hilbert--Schmidt operator (cf.\ \cite[Th.~2]{GelfandNaimark})
\begin{equation}\label{eq:pif}
\pi(f):=\int_G f(g)\pi(g)\,\dd g\in\End(V_\pi)
\end{equation}
equals
\begin{equation}\label{eq:tau-projection}
\pi(\ov{\chi_n})\pi(f)\pi(\ov{\chi_m})=\pi(\ov{\chi_n}\star f\star\ov{\chi_m})\in\Hom(V_{\pi}^{m},V_{\pi}^{n}),
\end{equation}
where the convolutions are meant over $K$.

\subsection{Plancherel theorem}\label{subsec:plancherel}
In this subsection, we review the Plancherel theorem for $G=\SL_2(\CC)$ pioneered by Gelfand and Naimark, following the original sources \cite{GelfandNaimark,GelfandNaimark2} and their translations \cite{GelfandNaimarkTranslated,GelfandNaimark2Translated}. We note that the list of unitary representations given in \cite{GelfandNaimark2} is incomplete for higher rank groups (cf.\ \cite{Stein,Vogan,Tadic}), but this does not affect the results we are quoting. In addition, we warn the reader that the translations contain some misprints not present in the originals, e.g.\ in the crucial formulae \cite[(137)--(138)]{GelfandNaimarkTranslated}.

We identify once and for all (non-canonically) the \emph{tempered unitary dual} $\Gtemp$ with the set
\[\left\{\pi_{it,p}:(t,p)\in\left(\RR_{>0}\times\tfrac12\ZZ\right)\cup\left(\{0\}\times\tfrac12\ZZ_{\geq 0}\right)\right\},\]
with topology inherited from the standard topology on $\RR^2$. The \emph{Plancherel measure} on $\widehat{G}$ is supported on $\Gtemp$, and it is given explicitly as
\begin{equation}\label{eq:concrete-plancherel-measure}
\dd\mupl(\pi_{it,p}):=\frac{1}{\pi^2}(t^2+p^2)\,\dd t\,\dd p,
\end{equation}
with $\dd t$ the Lebesgue measure on $\RR_{\geq 0}$ and $\dd p$ the counting measure on $\tfrac12\ZZ$. For $\pi_{it,p}\in\Gtemp$, the underlying Hilbert space $\widehat{V_{it,p}}$ is independent of the parameters: it equals $\Vcan:=L^2(\CC)$. On this common representation space, \eqref{transformedfunctions} defines the unitary action $\pi_{it,p}:G\to\U(\Vcan)$ that agrees with \cite[(65)]{GelfandNaimark} for $(n,\rho)=(2p,2t)$. The \emph{operator-valued spherical transform} of $f\in C_c(G)$ is the map $\Gtemp\to\End(\Vcan)$ given by $\pi\mapsto\pi(f)$ as in \eqref{eq:pif}. The Plancherel theorem for $G$ concerns the extension of this transform to $L^2(G)$, and characterizes its image.

\begin{theorem}[Gelfand--Naimark]\label{thm:abstract-isomorphism}
The map given by \eqref{eq:pif} extends (uniquely) to an $L^2$-isometry
\[L^2(G)\longrightarrow L^2(\Gtemp\to\End(\Vcan)),\]
where the operator-valued $L^2$-space on the right-hand side is meant with respect to the Hilbert--Schmidt norm ${\|\cdot\|}_{\mathrm{HS}}$ on
$\End(\Vcan)$ and the Plancherel measure $\mupl$ on $\Gtemp$. In particular, for every $f\in L^2(G)$, the following Plancherel formula holds:
\begin{equation}\label{eq:abstract-plancherel}
\int_G |f(g)|^2\,\dd g = \int_{\Gtemp} {\|\pi(f)\|}_{\mathrm{HS}}^2\,\dd \mupl(\pi).
\end{equation}
\end{theorem}

\begin{proof}
The theorem follows from \cite[Th.~5]{GelfandNaimark}; we only need to check that our Plancherel measure corresponds to the one in \cite[(137)]{GelfandNaimark}. We do this in four steps.
\newline\emph{Step 1.}
We observe that the constant $(8\pi^4)^{-1}$ in \cite[(137)]{GelfandNaimark} should be $(16\pi^4)^{-1}$ due to a small oversight in the derivation of \cite[(130)]{GelfandNaimark} from \cite[(129)]{GelfandNaimark}. The oversight is that the change of variables
\[(w_1,w_2,\lambda)\mapsto(\zeta_1,\zeta_2,\zeta_3):=(w_2,w_1\bar\lambda+w_2/\bar\lambda,w_1)\]
coming from \cite[(123)]{GelfandNaimark} is not 1-to-1 but 2-to-1.
\newline\emph{Step 2.}
We rewrite the corrected right-hand side of \cite[(137)]{GelfandNaimark} as a sum over $p\in\frac{1}{2}\ZZ$ and an integral over $t>0$, keeping in mind that $(n,\rho)$ in \cite{GelfandNaimark} is $(2p,2t)$ in our notation.
\newline\noindent\emph{Step 3.}
We observe that the Haar measure $\dd\mu(g)$ used by Gelfand--Naimark is $2\pi^2\dd g$. Indeed, applying
\cite[(40)]{GelfandNaimark} to a right $K$-invariant test function $f\in C_c(G)$, we obtain by several changes of variables that
\begin{align*}\int_G f(g)\,\dd\mu(g)
&=\int_{\CC\times\CC^\times\times\CC}f\left(\begin{pmatrix}w^{-1}&z\\&w\end{pmatrix}
\begin{pmatrix}1&\\v&1\end{pmatrix}\right)|\dd v|\,|\dd w|\,|\dd z|\\[4pt]
&=\int_{\CC\times\CC^\times\times\CC}f\left(\begin{pmatrix}w^{-1}&z\\&w\end{pmatrix}
\begin{pmatrix}1/\sqrt{1+|v|^2}&\bar v/\sqrt{1+|v|^2}\\&\sqrt{1+|v|^2}\end{pmatrix}\right)|\dd v|\,|\dd w|\,|\dd z|\\[4pt]
&=\int_{\CC\times\CC^\times\times\CC}f\left(\begin{pmatrix}w^{-1}&z\\&w\end{pmatrix}\right)
\frac{|\dd v|\,|\dd w|\,|\dd z|}{(1+|v|^2)^2}\\[4pt]
&= \pi \int_{\CC^\times\times\CC}f\left(\begin{pmatrix}1&z\\&1\end{pmatrix}
\begin{pmatrix}w&\\&w^{-1}\end{pmatrix}\right)\frac{|\dd w|\,|\dd z|}{|w|^6} = 2\pi^2\int_G f(g)\,\dd g.
\end{align*}
\emph{Step 4.}
Putting everything together, the corrected version of \cite[(137)]{GelfandNaimark} yields
\[\int_G |f(g)|^2\,2\pi^2\dd g = \frac{1}{16\pi^4}\sum_p\int_0^\infty{\|2\pi^2\pi_{it,p}(f)\|}_{\mathrm{HS}}^2\,(4t^2+4p^2)\,2\dd t.\]
This formula is equivalent to \eqref{eq:abstract-plancherel}, hence we are done.
\end{proof}

\begin{remark}\label{remark:Plancherel} In the proof above, we claimed that the Plancherel measure in \cite[Th.~5]{GelfandNaimark} is off by a factor of $2$. For double checking this claim, we looked at \cite[Th.~11.2]{Knapp}, and we found (to our dismay) that the Plancherel measure there is off by a factor of $\pi$. For example, for the test function $f(g):=1/\tr(gg^*)^2$, the Fourier transform given by \cite[(11.14)]{Knapp} equals $F_f^T(t)=\pi/\tr(tt^*)$, hence in \cite[(11.17)]{Knapp} the left-hand side is $\pi^2$, while the right-hand side is $\pi$. For triple checking our claim, we verified that our Plancherel measure yields the correct inversion formula for the classical spherical transform (for bi-$K$-invariant functions), as in \cite[\S 3.3]{FHMM}.
\end{remark}

\begin{theorem}[Gelfand--Naimark]\label{thm:general-inversion-formula}
Let $f\in C_c^\infty(G)$. For every $\pi\in\Gtemp$, the operator $\pi(f)\in\End(\Vcan)$ is of trace class, and the following inversion formula holds:
\begin{equation}\label{eq:general-inversion-formula}
f(g)=\int_{\Gtemp} \tr(\pi(f)\pi(g^{-1}))\,\dd\mupl(\pi).
\end{equation}
\end{theorem}

\begin{proof}
The theorem follows from \cite[Th.~19]{GelfandNaimark2} applied to $n=2$ and $x=R(g)f$, or from \cite[Th.~11.2]{Knapp}, with appropriate correction of the Plancherel measure (cf.\ Remark~\ref{remark:Plancherel}).
\end{proof}

\begin{remark} By a celebrated result of Dixmier--Malliavin~\cite{DixmierMalliavin1978}, every $f\in C_c^\infty(G)$ can be written as a linear combination of convolutions $w\star w^\ast$, where $w\in C_c^\infty(G)$ and $w^\ast(g):=\ov{w(g^{-1})}$. 
Hence Theorem~\ref{thm:general-inversion-formula} also follows from Theorem~\ref{thm:abstract-isomorphism} and \cite[Th.~2]{GelfandNaimark}. In fact for this implication we only need that $w\in C_c(G)$, which is easier to achieve.
\end{remark}

\subsection{The \texorpdfstring{$\tau_{\ell}$}{tau-ell}-spherical transform}\label{section:sphericaltransform}
For a given $\ell\in\frac{1}{2}\ZZ_{\geq 0}$, it is interesting to see what Theorems~\ref{thm:abstract-isomorphism} and \ref{thm:general-inversion-formula} yield for test functions $f\in L^2(G)$ with the following property: for almost every $\pi\in\Gtemp$, the operator $\pi(f)$ acts by a scalar on $V_{\pi}^{\ell}$ and by zero on its orthocomplement $V_{\pi}^{\ell,\perp}$. In the light of \eqref{eq:block-decomposition}, \eqref{eq:tau-projection}, \eqref{eq:abstract-plancherel}, and Schur's lemma, these test functions form the Hilbert subspace $\mcH(\tau_{\ell})\subset L^2(G)$ defined by the conditions
\begin{itemize}
\item $f(g)=f(kgk^{-1})$ for almost every $g\in G$ and $k\in K$;
\item $f=\ov{\chi_\ell} \star f \star\ov{\chi_\ell}$.
\end{itemize}

Let $\Gtemp(\tau_{\ell})$ be the set of $\pi\in\Gtemp$ whose restriction to $K$ contains $\tau_{\ell}$. For $f\in\mcH(\tau_{\ell})$, the operator-valued function $\pi\mapsto\pi(f)$ is supported on $\Gtemp(\tau_{\ell})$, and there it is simply determined by the scalar-valued function $\pi\mapsto\tr(\pi(f))$ via
\begin{equation}\label{eq:opfromtr}
\pi(f)|_{V_{\pi}^{\ell}}= \frac{\tr(\pi(f))}{2\ell+1}\cdot \id_{V_{\pi}^{\ell}}
\qquad\text{and}\qquad\pi(f)|_{V_{\pi}^{\ell,\perp}}=0.
\end{equation}
In particular, for $\pi\in\Gtemp(\tau_{\ell})$ and $f\in\mcH(\tau_{\ell})$,
\begin{equation}\label{eq:HSfromtr}
{\|\pi(f)\|}_{\mathrm{HS}}^2 = \tr(\pi(f)\pi(f)^*)= \frac{|\tr(\pi(f))|^2}{2\ell+1}.
\end{equation}
For $(\nu,p)\in i\RR\times\frac12\ZZ$, the condition $\pi_{\nu,p}\in\Gtemp(\tau_{\ell})$ is equivalent to $|p|\leq\ell$ and $p\equiv\ell\pmod{1}$. Moreover, for $f\in L^1(G)\cap\mcH(\tau_{\ell})$, the trace of $\pi_{\nu,p}(f)$ can be expressed
in terms of the \emph{$\tau_\ell$-spherical trace function}
\begin{equation}\label{eq:def-spherical-function}
\begin{aligned}
\varphi_{\nu,p}^{\ell}(g):=&\tr(\pi_{\nu,p}(\ov{\chi_\ell})\pi_{\nu,p}(g)\pi_{\nu,p}(\ov{\chi_\ell}))\\
=&\tr(\pi_{\nu,p}(\ov{\chi_{\ell}})\pi_{\nu,p}(g))=\tr(\pi_{\nu,p}(g)\pi_{\nu,p}(\ov{\chi_{\ell}}))
\end{aligned}
\end{equation}
as (cf.\ \eqref{eq:pif} and \eqref{eq:tau-projection})
\begin{equation}\label{eq:trpif}
\widehat{f}(\nu,p):=\tr(\pi_{\nu,p}(f)) = \int_G f(g)\,\varphi_{\nu,p}^{\ell}(g)\,\dd g.
\end{equation}

The function $\varphi_{\nu,p}^{\ell}:G\to\CC$ vanishes unless $|p|\leq\ell$ and $p\equiv\ell\pmod{1}$, for else $\tau_{\ell}$ does not appear in $\pi_{\nu,p}$, and $\varphi_{\nu,p}^{\ell}(\id)=2\ell+1$ in this latter case. Moreover, we have the integral representation of Harish-Chandra \cite[Cor.~6.2.2.3]{Warner1972}:
\[\varphi_{\nu,p}^{\ell}(g)=\int_K\left(\chi_\ell\star\eta_p\right)(\kappa(k^{-1}gk))\,
e^{(\nu-1)\rho(H(gk))}\,\dd k.\]
Here, $\eta_p:M\simeq S^1\to\CC^{\times}$ is the unitary character $\eta_p(z)=z^{-2p}$, the convolution is over $M$, and $\kappa$, $\rho$, and $H$ are as in \S\ref{SL2C-subsec}. For computational purposes, we spell out the $\chi_\ell\star\eta_p$ term explicitly, cf.~\eqref{matrix-coeff}, \cite[(10) \& Lemma~3.2]{Wang1974}, \cite[Th.~29.18]{HewittRoss}:
\[\begin{aligned}
\left(\chi_\ell\star\eta_p\right)(k[\alpha,\beta])
&=(2\ell+1)\Phi_{p,p}^{\ell}(k[\alpha,\beta])\\
&=(2\ell+1)\sum_{r=0}^{\ell-|p|}(-1)^r\binom{\ell+p}{r}\binom{\ell-p}{r}
\alpha^{\ell-p-r}{\bar\alpha}^{\ell+p-r}|\beta|^{2r}.
\end{aligned}\]
We collect further useful properties of $\varphi_{\nu,p}^{\ell}:G\to\CC$ in the next lemma, where we write
\[a_h:= \diag(e^{h/2}, e^{-h/2}),\qquad h\in\RR.\]

\begin{lemma} The $\tau_\ell$-spherical trace function $\varphi_{\nu,p}^{\ell}(g)$ extends holomorphically to $\nu\in\CC$, and it satisfies the bound
\begin{equation}\label{eq:spherical-function-bound}
\bigl|\varphi_{\sigma+it,p}^{\ell}(k_1a_hk_2)\bigr|\leq(2\ell+1)\frac{\sinh(\sigma h)}{\sigma\sinh(h)},
\qquad \sigma,t,h\in\RR,\quad k_1,k_2\in K.
\end{equation}
(For $\sigma=0$ or $h=0$, the fraction on the right-hand side is understood as $1$.)
The extended function has the symmetries
\begin{equation}\label{eq:spherical-function-symmetry}
\varphi_{\nu,p}^{\ell}(g)=\ov{\varphi_{-\ov{\nu},p}^{\ell}(g)}=\varphi_{\nu,p}^{\ell}(g^{-1}),
\end{equation}
\begin{equation}\label{eq:spherical-function-symmetry-2}
\varphi_{\nu,p}^{\ell}(g)=\varphi_{p,\nu}^{\ell}(g),\qquad\nu\equiv p\one,\quad|\nu|,|p|\leq\ell.
\end{equation}
\end{lemma}

\begin{proof}
The holomorphic extension of $\varphi_{\nu,p}^{\ell}(g)$ and the bound \eqref{eq:spherical-function-bound} are
a straightforward generalization of \cite[Prop.~3.4]{Wang1974} and its proof. The identity $\ov{\varphi_{-\ov{\nu},p}^{\ell}(g)}=\varphi_{\nu,p}^{\ell}(g^{-1})$ follows from \eqref{eq:def-spherical-function} and $\pi(g)^*=\pi(g^{-1})$ for $\nu\in i\RR$, and then also for $\nu\in\CC$ by the uniqueness of analytic continuation. The identity
$\varphi_{\nu,p}^{\ell}(g)=\varphi_{\nu,p}^{\ell}(g^{-1})$ is \cite[Lemma~3.2]{Wang1974}, keeping in mind that $\pi_{\nu,p}\simeq\pi_{-\nu,-p}$ for $\nu\in i\RR$ and again invoking analytic continuation. Finally, the remarkable symmetry \eqref{eq:spherical-function-symmetry-2} follows from \cite[Cor.~2]{Wang1974}, or more conceptually from the discussion below \eqref{eq:tau-ell-isotypical-decomposition-algebraic}.
\end{proof}

As we shall see in Theorem~\ref{thm:tr-plancherel} below, the \emph{$\tau_\ell$-spherical transform} defined by \eqref{eq:trpif} is inverted by the following \emph{inverse $\tau_\ell$-spherical transform}. For $h\in L^1(\Gtemp(\tau_{\ell}))\cap L^2(\Gtemp(\tau_{\ell}))$ and $g\in G$, we define
\begin{equation}\label{eq:inverse-tauell-transform}
\widecheck{h}(g):=\frac{1}{(2\ell+1)\pi^2}\sum_{\substack{|p|\leq\ell\\p\equiv\ell\one}}
\int_{0}^{\infty} h(it,p)\,\varphi_{it,p}^{\ell}(g^{-1})\,(t^2+p^2)\,\dd t.
\end{equation}

\begin{theorem}\label{thm:tr-plancherel} The transforms defined by \eqref{eq:trpif} and \eqref{eq:inverse-tauell-transform} extend (uniquely) to a pair of Hilbert space isometries inverse to each other:
\[\mcH(\tau_{\ell})\longleftrightarrow L^2(\Gtemp(\tau_{\ell})).\]
In particular, for $f\in \mcH(\tau_{\ell})$, the following Plancherel formula holds:
\begin{equation}\label{eq:tr-plancherel}
\int_G |f(g)|^2\,\dd g = \frac{1}{(2\ell+1)\pi^2}
\sum_{\substack{|p|\leq\ell\\p\equiv\ell\one}}\int_{0}^{\infty}
|\widehat{f}(it,p)|^2\,(t^2+p^2)\,\dd t.
\end{equation}
\end{theorem}

\begin{proof} The fact that $\ \widehat{}\ $ extends to a Hilbert space isomorphism $\mcH(\tau_{\ell})\to L^2(\Gtemp(\tau_\ell))$ follows from Theorem~\ref{thm:abstract-isomorphism} and our discussion above. In particular, \eqref{eq:tr-plancherel} is a special case of \eqref{eq:abstract-plancherel} in the light of \eqref{eq:concrete-plancherel-measure}, \eqref{eq:HSfromtr}, \eqref{eq:trpif}. We are left with proving that $\ \widecheck{}\ $ is the inverse of $\ \widehat{}\ $, and for this it suffices to verify that $\ \widecheck{}\ $ applied after $\ \widehat{}\ $ is the identity on the dense subset $C_c^{\infty}(G)\cap\mcH(\tau_{\ell})$ of the Hilbert space $\mcH(\tau_{\ell})$. For $f\in C_c^{\infty}(G)\cap \mcH(\tau_{\ell})$, \eqref{eq:projectionbychi}, \eqref{eq:pif}, \eqref{eq:concrete-plancherel-measure}, \eqref{eq:general-inversion-formula}, \eqref{eq:opfromtr}, \eqref{eq:def-spherical-function}, \eqref{eq:trpif} yield
\begin{align*}f(g)
&=\int_{\Gtemp}\tr(\pi(f)\pi(g^{-1}))\,\dd\mupl=\frac{1}{2\ell+1}\int_{\Gtemp}\tr(\pi(f))\tr(\pi(\ov{\chi_\ell})\pi(g^{-1}))\,\dd\mupl\\
&=\frac{1}{(2\ell+1)\pi^2} \sum_{\substack{|p|\leq \ell \\ p\equiv \ell\one}} \int_{0}^{\infty} \widehat{f}(it,p)\, \varphi_{it,p}^{\ell}(g^{-1})\,(t^2+p^2)\,\dd t.
\end{align*} 
The proof is complete.
\end{proof}

Wang~\cite{Wang1974} proved an analogue of the Paley--Wiener theorem for the $\tau_\ell$-spherical transform, and in particular characterized the image of $\mcH(\tau_{\ell})\cap C_c^\infty(G)$ under the transform. The following is \cite[Prop.~4.5]{Wang1974} and should be compared to Theorem~\ref{thm:pws} in the introduction.

\begin{theorem}[Wang]\label{thm10}
Let $f\in\mcH(\tau_{\ell})$ be a test function, and let $R>0$. Then the following two conditions are equivalent.
\begin{enumerate}[(a)]
\item The function $f(g)$ is smooth, and
\[f(k_1 a_h k_2)=0,\qquad |h|>R,\quad k_1,k_2\in K.\]
\item The function $\widehat{f}(\nu,p)$ has a holomorphic extension to $\CC\times\tfrac12\ZZ$ such that
\[\widehat{f}(\nu,p)=\widehat{f}(p,\nu),\qquad\nu\equiv p\one,\quad|\nu|,|p|\leq\ell,\]
and for any $C>0$ we have
\[\widehat{f}(\nu,p)\ll_C (1+|\nu|)^{-C}e^{R|\Re\nu|},\qquad \nu\in\CC,\quad p\in\tfrac12\ZZ.\]
\end{enumerate}
\end{theorem}

We now prove a Schwartz class version of this result as stated in Theorem~\ref{thm:pws}.

\begin{proof}[Proof of Theorem~\ref{thm:pws}]
For harmony of notation with \cite{Wang1974}, in this proof we use $D_{p,q}^{\ell}(k)$ to denote the matrix coefficients of $\tau_{\ell}$ relative to the basis obtained by normalizing the orthogonal basis $\{z^{\ell-q}:|q|\leq\ell,q\equiv\ell\pmod 1\}$ in the space $V_{2\ell}$ of \S\ref{SU2-subsec}. Thus we explicitly have the renormalization
\[ D_{p,q}^{\ell}(k)=\left(\frac{(\ell-p)!(\ell+p)!}{(\ell-q)!(\ell+q)!}\right)^{1/2}\Phi_{p,q}^{\ell}(k). \]

Assume condition \ref{pws-a}. The holomorphic extension of $\widehat{f}(\nu,p)$ follows from \eqref{eq:spherical-function-bound} coupled with \eqref{eq:mAbound} for $m=0$, and then \eqref{eq:symmetry} is immediate from \eqref{eq:spherical-function-symmetry-2}. In order to derive \eqref{eq:BCbound}, we use an alternate representation of $\widehat{f}(\nu,p)$. We shall assume that $|p|\leq\ell$ and $p\equiv\ell\pmod{1}$, for else $\widehat{f}(\nu,p)=0$.
By the third line of the second display on \cite[p.~621]{Wang1974} and \cite[Lemma~3.2]{Wang1974}, we see that the (unique) holomorphic extension is also provided by
\begin{equation}\label{eq:sphericalviaAbel}
\widehat{f}(\nu,p)=\frac{2\ell+1}{2}\int_{-\infty}^\infty\breve{f}(h,p)\,e^{\nu h}\,\dd h,
\end{equation}
where
\begin{equation}\label{eq:Abelnew}
\breve{f}(h,p):=e^h\int_K\int_N f(ka_hn)\,D_{p,p}^{\ell}(k)\,\dd k\,\dd n.
\end{equation}
We claim that, for any $m\in\ZZ_{\geq 0}$ and $A>0$, we have
\begin{equation}\label{eq:Abelbound}
\frac{\partial^m}{\partial h^m}\breve{f}(h,p)\ll_{m,A} e^{-A|h|},\qquad h\in\RR,\quad |p|\leq\ell,\quad p\equiv\ell\one.
\end{equation}
For $|h|>1$ this follows by writing $a_hn=k_1a_{h'}k_2$ in \eqref{eq:Abelnew}, and then combining \eqref{eq:mAbound} with some calculus to keep track of the dependence of $h'\in\RR$ and $k_1,k_2\in K$ on $h\in\RR$. For $|h|\leq 1$ we proceed similarly for the part of the integral in \eqref{eq:Abelnew} that corresponds to $n=\big(\begin{smallmatrix}1&z\\& 1\end{smallmatrix}\big)$ with $|z|>1$, while we estimate the ($h$-derivatives of the) remaining integral directly by the smoothness of $f(g)$. With \eqref{eq:Abelbound} at hand, \eqref{eq:BCbound} follows from \eqref{eq:sphericalviaAbel} via integration by parts. We proved that \ref{pws-a} implies \ref{pws-b}.

Assume condition \ref{pws-b}. By Theorem~\ref{thm:tr-plancherel},
\[f(g)=\frac{1}{(2\ell+1)\pi^2} \sum_{\substack{|p|\leq \ell \\ p\equiv \ell\one}} \int_{0}^{\infty} \widehat{f}(it,p)\, \varphi_{it,p}^{\ell}(g^{-1})\,(t^2+p^2)\,\dd t.\]
Let us restrict, without loss of generality, to $g=k_1a_hk_2$ with $h>0$. Using the display below \cite[(29)]{Wang1974}\footnote{We note that in \cite[(29)]{Wang1974} the product $k_2k_1$ should be conjugated as $u_{\varphi_1}^{-1}k_2k_1u_{\varphi_1}$, and the integral over $0\leq\varphi_1\leq 2\pi$ with normalization factor $1/(2\pi)$ is missing. After this correction, the crucial next display follows as stated, by expanding the matrix coefficient $D_{-p,-p}^{\ell}$ (in our notation) via the entry-by-entry product of three matrices and executing the $\varphi_1$-integral.}, we infer
\[f(g)=\frac{1}{4\pi^2\sinh(h)}\sum_{\substack{|p|,|j|\leq \ell \\ p,j\equiv\ell\one}}
\int_{-h}^h \widetilde{f}(s,p) \,
D_{-p,j}^{\ell}(v_{\theta}^{-1})D_{j,j}^{\ell}(k_2k_1) D_{j,-p}^{\ell}(v_{\theta'})\,\dd s,\]
where $D_{-p,j}^{\ell}(v_{\theta}^{-1})$ and $D_{j,-p}^{\ell}(v_{\theta'})$ can be explicated using \cite[(5) \& (28)]{Wang1974}, and
\begin{equation}\label{eq:def-tilde-f-s-p}
\widetilde{f}(s,p):=\int_{-\infty}^\infty\widehat{f}(it,p)\,e^{-its}\,(t^2+p^2)\,\dd t,\qquad s\in\RR.
\end{equation}
By \eqref{eq:BCbound} and Cauchy's theorem, it follows for any $n\in\ZZ_{\geq 0}$ and $D>0$ that
\begin{equation}\label{eq:nDbound}
\frac{\partial^n}{\partial s^n}\widetilde{f}(s,p)\ll_{n,D} e^{-D|s|},\qquad s\in\RR.
\end{equation}
The smoothness of $f(g)$ is now straightforward, and this automatically verifies \eqref{eq:mAbound} for $|h|\leq 1$. From now on we can assume, without loss of generality, that $h>1$. From \eqref{eq:symmetry}, \eqref{eq:nDbound}, and the calculation around \cite[(38)--(41)]{Wang1974}, we see that
\[\sum_{\substack{|p|,|j|\leq \ell \\ p,j\equiv\ell\one}}
\int_{-\infty}^\infty \widetilde{f}(s,p) \,
D_{-p,j}^{\ell}(v_{\theta}^{-1})D_{j,j}^{\ell}(k_2k_1)D_{j,-p}^{\ell}(v_{\theta'})\,\dd s=0,\]
hence in fact
\[f(g)=\frac{-1}{4\pi^2\sinh(h)}\sum_{\substack{|p|,|j|\leq \ell \\ p,j\equiv\ell\one}}
\left(\int_{-\infty}^{-h}+\int_h^\infty\right)\widetilde{f}(s,p) \,
D_{-p,j}^{\ell}(v_{\theta}^{-1})D_{j,j}^{\ell}(k_2k_1)D_{j,-p}^{\ell}(v_{\theta'})\,\dd s.\]
From here it is straightforward to deduce \eqref{eq:mAbound} for $h>1$, using \eqref{eq:nDbound} and the remarks above it. We proved that \ref{pws-b} implies \ref{pws-a}.
\end{proof}

We shall denote by $\mcH(\tau_{\ell})_{\infty}$ the set of functions satisfying the equivalent conditions \ref{pws-a} and \ref{pws-b} of Theorem~\ref{thm:pws}. It is clear that $\mcH(\tau_{\ell})_{\infty}$ is a convolution subalgebra of $L^1(G)\cap L^2(G)$.

\begin{remark} In the previous display, we may estimate the product of the three matrix coefficients (recalling that each matrix $(D_{p,q}^{\ell}(k))_{p,q}$ is orthogonal) using the trivial bound $|D_{j,j}^{\ell}|\leqslant 1$ and the Cauchy--Schwarz inequality for the remaining two factors. Combining this with the observation $\widetilde{f}(s,p)=\widetilde{f}(-s,-p)$ yields the following refinement of \eqref{eq:mAbound} when $m=0$:
\begin{equation}\label{eq:inverse-spherical-transform-estimate}
\bigl|f(k_1a_hk_2)\bigr|\leq\sum_{\substack{|p|\leq \ell \\ p\equiv\ell\one}}
\int_h^\infty \bigl|\widetilde{f}(s,p)\bigr|\,\dd s,\qquad h>1,\quad k_1,k_2\in K. 
\end{equation}
\end{remark}

We end this subsection by stating a two-variable version of some of the previous definitions and results. Taking (topological) tensor products of Hilbert spaces, we can identify $\mcH(\tau_{\ell}) \hat{\otimes} \mcH(\tau_{\ell})$ with the space of functions $f \in L^2(G \times G)$ satisfying
\begin{itemize}
\item $f(g_1, g_2) = f(k_1g_1k_1^{-1}, k_2g_2k_2^{-1})$ for almost every
$g_1, g_2 \in G$ and $k_1, k_2 \in K$;
\item $f=(\ov{\chi_\ell},\ov{\chi_\ell})  \star f \star (\ov{\chi_\ell},\ov{\chi_\ell})$ almost everywhere.
\end{itemize}
This can be seen by projecting the isomorphism between $L^2(G)\hat{\otimes} L^2(G)$ and $L^2(G\times G)$ (see e.g.\ \cite[Cor.~4.11.9]{Simon2015}) to $\mcH(\tau_{\ell}) \hat{\otimes} \mcH(\tau_{\ell})$ and the (closed) subspace of functions in question. By Theorem~\ref{thm:tr-plancherel}, this space is isometrically isomorphic to $L^2(\Gtemp(\tau_{\ell})^2)$ via the obvious extension of the map \eqref{eq:trpif}:
\begin{equation}\label{doubletransform}
\widehat{f}(\nu_1, p_1, \nu_2, p_2) := \int_{G_1\times G_2} f(g_1, g_2) \,
\varphi_{\nu_1, p_1}^{\ell}(g_1)\varphi_{\nu_2, p_2}^{\ell}(g_2) \, \dd g_1\,\dd g_2.
\end{equation}
For $h\in L^1(\Gtemp(\tau_{\ell})^2)\cap L^2(\Gtemp(\tau_{\ell})^2)$, the 
inverse transform is given as in \eqref{eq:inverse-tauell-transform}:
\begin{equation}\label{inversedoubletransform}
\begin{split}
\widecheck{h}(g_1, g_2) := \frac{1}{(2\ell+ 1)^2\pi^4}&
\sum_{\substack{|p_1|, |p_2|\leq\ell\\p_1 \equiv p_2\equiv\ell\one}}
\int_{0}^{\infty}\int_{0}^{\infty}  h(it_1,p_1, it_2, p_2)\\
&\qquad\varphi_{it_1,p_1}^{\ell}(g_1^{-1})\varphi_{it_2,p_2}^{\ell}(g_2^{-1})\,(t_1^2+p_1^2)(t_2^2+p_2^2)\,\dd t_1\,\dd t_2.
\end{split}
\end{equation}
It is straightforward to adapt the above presented proof of Theorem~\ref{thm:pws} to obtain the following variant for $\mcH(\tau_{\ell}) \hat{\otimes} \mcH(\tau_{\ell})$:

\begin{theorem}\label{thm:pws2} For $f\in\mcH(\tau_{\ell}) \hat{\otimes} \mcH(\tau_{\ell})$, the following two conditions are equivalent (with implied constants depending on $f$).
\begin{enumerate}[(a)]
\item\label{pws2-a} The function $f(g_1,g_2)$ is smooth, and for any $m_1,m_2\in\ZZ_{\geq 0}$ and $A>0$ we have
\[\frac{\partial^{m_1+m_2}}{\partial h_1^{m_1}\partial h_2^{m_2}}f(k_1 a_{h_1} k_2,k_3 a_{h_2} k_4)\ll_{m_1,m_2,A} e^{-A(|h_1|+|h_2|)},
\quad h_1,h_2\in\RR,\,\, k_1,k_2,k_3,k_4\in K.\]
\item\label{pws2-b} The function $\widehat{f}(\nu_1,p_1,\nu_2,p_2)$ has a holomorphic extension to $\CC\times\tfrac12\ZZ\times\CC\times\tfrac12\ZZ$ such that
\begin{align*}
\widehat{f}(\nu_1,p_1,\nu_2,p_2) &= \widehat{f}(p_1,\nu_1,\nu_2,p_2),
\qquad \nu_1 \equiv p_1\pmod{1},\quad |\nu_1|, |p_1| \leq \ell,\\
\widehat{f}(\nu_1,p_1,\nu_2,p_2) &= \widehat{f}(\nu_1,p_1,p_2,\nu_2),
\qquad \nu_2 \equiv p_2\pmod{1},\quad |\nu_2|, |p_2| \leq \ell,
\end{align*}
and for any $B,C>0$ we have
\[\widehat{f}(\nu_1,p_1,\nu_2,p_2) \ll_{B, C} (1+|\nu_1|+|\nu_2|)^{-C},
\qquad |\Re\nu_1|,|\Re\nu_2|\leq B,\quad p_1,p_2\in\tfrac{1}{2}\ZZ.\]
\end{enumerate}
\end{theorem}

We shall denote by $\mcH(\tau_{\ell},\tau_{\ell})_{\infty}$ the set of functions satisfying the equivalent conditions \ref{pws2-a} and \ref{pws2-b} of Theorem~\ref{thm:pws2}; this is clearly a convolution subalgebra of $L^1(G\times G)\cap L^2(G\times G)$.

\subsection{Hecke operators}
The arithmetic quotient $\Gamma\backslash G$ comes equipped with a rich family of Hecke correspondences, which we now describe, referring to \cite{BlomerHarcosMilicevic2016} for further details and references. For every $n\in\ZZ[i]\setminus\{0\}$, consider the set
\[ \Gamma_n:=\left\{\begin{pmatrix} a&b\\c&d\end{pmatrix}\in\MM_2(\ZZ[i]):ad-bc=n\right\}. \]
In particular, $\Gamma_1=\Gamma$. Then we may define the Hecke operator $T_n$ acting on functions $\phi:\Gamma\backslash G\to\CC$ by
\begin{equation}\label{Hecke-def}
(T_n\phi)(g):=\frac1{|n|}\sum_{\gamma\in\Gamma\backslash\Gamma_n}\phi\left(\frac1{\sqrt{n}}\gamma g\right)=\frac1{4|n|}\sum_{ad=n}\sum_{b\bmod d}\phi\left(\frac1{\sqrt{n}}\begin{pmatrix}a&b\\0&d\end{pmatrix}g\right),
\end{equation}
where the result is independent of the choice of the square-root since $\pm\id\in\Gamma$. In particular, since $\Gamma_{-1}=\Gamma\cdot\big(\begin{smallmatrix}-1&\\&1\end{smallmatrix}\big)$ and $\frac{1}{i}\big(\begin{smallmatrix}-1&\\&1\end{smallmatrix}\big)=\big(\begin{smallmatrix}i&\\&-i\end{smallmatrix}\big)\in\Gamma$, we have $T_{-1}=T_1=\id$. We also observe that, as $\gamma$ ranges through a set of representatives of $\Gamma\backslash\Gamma_n$, $n\gamma^{-1}$ ranges through a set of representatives of $\Gamma_n/\Gamma$.

These Hecke operators are self-adjoint on $L^2(\Gamma\backslash G)$, commute with each other and the Laplace operator; thus they act by constants $\lambda_n(V)$ on each irreducible component $V\subset L^2(\Gamma\backslash G)$, with non-zero vectors in each $V$ being joint Hecke--Maa{\ss} eigenfunctions. They also satisfy the multiplicativity relation
\begin{equation}\label{hecke-mult}
T_mT_n=\sum_{(d)\mid(m,n)}T_{mn/d^2},\qquad m,n\in\ZZ[i]\setminus\{0\},
\end{equation}
where it is clear that the right-hand side does not depend on the choice of the generator $d$. Finally we have the Rankin--Selberg bound
\begin{equation}\label{RS-bound}
\sum_{|n|^2\leq x}|\lambda_n(V)|^2\ll_Vx.
\end{equation}

\subsection{Eisenstein series and spectral decomposition}\label{Eisenstein-subsec}
In this subsection, we review the construction and properties of the (not necessarily spherical) Eisenstein series on $\Gamma\backslash G$. The quotient $\Gamma\backslash G$ has a unique cusp at $\infty$. For $\ell\in\ZZ_{\geq 0}$, $p,q\in\ZZ$ with $2\mid p$ and $|p|,|q|\leq\ell$, and $\nu\in\CC$ with $\Re\nu>1$, we define the Eisenstein series of type $(\ell,q)$ at $\infty$ as in \cite[Def.~3.3.1]{Lokvenec-Guleska2004} by the absolutely and locally uniformly convergent series
\begin{equation}\label{Eisdef}
E_{\ell,q}(\nu,p)(g):=\sum_{\gamma\in\Gamma_{\infty}\backslash\Gamma}\phi_{\ell,q}(\nu,p)(\gamma g),
\end{equation}
where $\Gamma_{\infty}$ is the subgroup of upper-triangular matrices in $\Gamma$ (the stabilizer of $\infty$ in $\Gamma$), and
\begin{equation}\label{eq:phi-ell-q-nu-p}
\phi_{\ell,q}(\nu,p)\left(\begin{pmatrix}r&\ast\\&r^{-1}\end{pmatrix}k\right):=r^{2(1+\nu)}\Phi_{p,q}^{\ell}(k),\qquad r>0,\quad k\in K.
\end{equation}
These Eisenstein series possess a meromorphic continuation to $\nu\in\CC$, which is holomorphic along $i\RR$ \cite[\S5.1]{Lokvenec-Guleska2004}. An easy calculation with \eqref{Hecke-def} and \eqref{matrix-coeff} shows that they are also eigenfunctions of the Hecke operators $T_n$ with
\begin{equation}
\label{Eisenstein-Hecke-eigen}
T_nE_{\ell,q}(\nu,p)=\lambda_n(E(\nu,p))E_{\ell,q}(\nu,p),\quad \lambda_n(E(\nu,p)):=\frac{1}{4}\sum_{n=ad}\chi_{\nu,p}(a)\chi_{-\nu,-p}(d),
\end{equation}
where $\chi_{\nu,p}(z):=|z|^{\nu}(z/|z|)^{-p}$. In particular,
\begin{equation}\label{eq:hecke-i-eigenvalue-eisenstein}
\lambda_{in}(E(\nu,p))=(-1)^{p/2}\lambda_n(E(\nu,p)).
\end{equation}
While $E_{\ell,q}(\nu,p)$ for individual $\nu\in i\RR$ (barely) fail to lie in $L^2(\Gamma\backslash G)$, their averages against $C_c(i\RR)$ weights $f(\nu)$ comfortably do, and upon taking the Hilbert space closure of their span and orthocomplements one obtains the familiar orthogonal decomposition
\begin{equation}\label{eq:spectral-decomposition}
L^2(\Gamma\backslash G)=\CC\cdot 1\oplus L^2(\Gamma\backslash G)_{\cusp}\oplus L^2(\Gamma\backslash G)_{\Eis}.
\end{equation}

Let $H(\nu,p)$ be the linear span of all $\phi_{\ell,q}(\nu,p)$ with $|p|,|q|\leq\ell$. By \eqref{eq:phi-ell-q-nu-p}, the functions $f\in H(\nu,p)$ satisfy
\[f\left(\begin{pmatrix}z&\ast\\&z^{-1}\end{pmatrix}g\right)=|z|^2\chi_{\nu,p}(z^2)f(g),\qquad z\in\CC^\times,\quad g\in G,\]
and they are determined by their restriction to $K$. In fact $H(\nu,p)$ as a $(\mfg,K)$-module is isomorphic to the $K$-finite part of $V_{\nu,p}$ featured in \eqref{eq:tau-ell-isotypical-decomposition-algebraic}. That is, the appropriate completion of $H(\nu,p)$ serves as a model of the Fréchet/Hilbert space representation $\pi_{\nu,p}$, and we shall denote by 
$H^\infty(\nu,p)$ the dense subspace of smooth vectors in this completion.

Denoting by $C^K(\Gamma\backslash G)$ the space of $K$-finite smooth functions on $\Gamma\backslash G$, an automorphic representation of type $(\nu,p)$ for $\Gamma\backslash G$ may be realized as a unitary $(\mfg,K)$-module homomorphism $T:H(\nu,p)\to C^K(\Gamma\backslash G)$, with the corresponding $\pi_{\nu,p}$ an irreducible unitary representation on the Hilbert space $V_{\nu,p}$, cf.\ \cite[\S3.4 \& \S8]{Lokvenec-Guleska2004}. Such a $T$ may arise as $T_V$ for a cuspidal consituent $V\simeq V_{\nu,p}$ occurring discretely in $L^2(\Gamma\backslash G)_{\cusp}$, or from the Eisenstein series via
\[T_{E(\nu,p)}\phi_{\ell,q}(\nu,p):=E_{\ell,q}(\nu,p),\qquad |p|,|q|\leq\ell.\]
Indeed, by \eqref{Eisdef}, the last display defines a $(\mfg,K)$-module homomorphism for $\Re\nu>1$, hence by analytic continuation for all $\nu\in\CC$ where the relevant Eisenstein series have no pole. Following custom, we lighten the notation by denoting a generic automorphic representation of type $(\nu,p)$, whether of type $T_V$ or $T_{E(\nu,p)}$, as $V$, and its associated Hecke eigenvalues as $\lambda_n(V)$. Finally, we shall use that the above $(\mfg,K)$-module homomorphism extends uniquely to a $G$-module homomorphism $H^\infty(\nu,p)\to C^\infty(\Gamma\backslash G)$, and its image consists of functions of moderate growth.

Now \eqref{eq:spectral-decomposition} is explicated by the following two spectral identities. For $f$ in the space $C^{\infty}_0(\Gamma\backslash G)$ of smooth complex-valued functions on $\Gamma\backslash G$ with all rapidly decaying derivatives, we have
\begin{equation}\label{eq:spectral-decomposition-EGM}
\begin{split}
f= \frac{\langle f, 1\rangle }{\vol(\Gamma \backslash G)}  & + \sum_{\text{$V$ cuspidal}} \sum_{\substack{q,\ell\in \ZZ \\ |p_V|,|q|\leq \ell}} \frac{\langle f,T_V \phi_{\ell,q}(\nu_V,p_V) \rangle}{{\|\Phi_{p_V,q}^{\ell}\|}_K^2}  T_V \phi_{\ell,q}(\nu_V,p_V)\\ & + \frac{1}{\pi i} \int_{(0)} \sum_{p\in 2\ZZ} \sum_{\substack{q,\ell\in \ZZ \\ |p|,|q|\leq \ell}} \frac{\langle f,E_{\ell,q}(\nu,p) \rangle}{{\|\Phi_{p,q}
^{\ell}\|}_K^2} E_{\ell,q}(\nu,p) \,\dd\nu,
\end{split}
\end{equation}
with the obvious interpretation of $\langle f,E_{\ell,q}(\nu,p) \rangle$. For
$f_1,f_2\in C^{\infty}_0(\Gamma \backslash G)$, we have with the same interpretation
\begin{equation}\label{eq:spectral-decomposition-plancherel}
\begin{split}
\langle f_1,f_2 \rangle = \frac{ \langle f_1, 1\rangle \langle 1, f_2\rangle}{\vol(\Gamma \backslash G)} & + \sum_{\text{$V$ cuspidal}} \sum_{\substack{q,\ell\in \ZZ \\ |p_V|,|q|\leq \ell}} \frac{\langle f_1,T_V \phi_{\ell,q}(\nu_V,p_V) \rangle \langle T_V \phi_{\ell,q}(\nu_V,p_V), f_2 \rangle}{{\|\Phi_{p_V,q}^{\ell}\|}_K^2}  \\ & + \frac{1}{\pi i} \int_{(0)} \sum_{p\in 2\ZZ} \sum_{\substack{q,\ell\in \ZZ \\ |p|,|q|\leq \ell}} \frac{\langle f_1,E_{\ell,q}(\nu,p) \rangle \langle  E_{\ell,q}(\nu,p),f_2 \rangle}{{\|\Phi_{p,q}
^{\ell}\|}_K^2} \,\dd\nu.
\end{split}
\end{equation}
Compare with \cite[Ch.~6, Th.~3.4]{ElstrodtGrunewaldMennicke1998} and \cite[Th.~8.1]{Lokvenec-Guleska2004}.

We shorten the notation in two ways. First, for an automorphic representation $V$ (cuspidal or Eisenstein) of type $(\nu,p)$ occurring in $L^2(\Gamma\backslash G)$, we write
\[\phi_{\ell,q}^{V}:=\frac{T_V\phi_{\ell,q}(\nu,p)}{{\|\Phi_{p,q}^{\ell}\|}_K},\qquad |p|,|q|\leq\ell.\] 
In particular, when at least one of two such $V$ and $V'$ is cuspidal, $\langle\phi_{\ell,q}^V,\phi_{\ell',q'}^{V'}\rangle$ equals $\delta_{(\ell,q,V)=(\ell',q',V')}$. Second, while the decompositions in \eqref{eq:spectral-decomposition-EGM} and \eqref{eq:spectral-decomposition-plancherel} are over all automorphic representations $V$ (cuspidal or Eisenstein) occurring in $L^2(\Gamma\backslash G)$, keeping in mind the $\tau_{\ell}$-spherical transform of \S\ref{section:sphericaltransform}, it will be useful to introduce the shorthand notation $\int_{[\ell]}\,\dd V$ for the sum-integral over those $V$ of type $(\nu,p)$ such that $\pi_{\nu,p}\in\widehat{G}(\tau_{\ell})$ (that is, with $|p|\leq\ell$ as well as $p\in 2\ZZ$ for $V$ Eisenstein). Thus, for example, \eqref{eq:spectral-decomposition-EGM} may be rewritten in the more compact form
\begin{equation}
\label{compactform}
f=\frac{\langle f,1\rangle}{\vol(\Gamma \backslash G)}+\sum_{\ell\geq 0}\int_{[\ell]}\sum_{|q|\leq\ell}\langle f,\phi_{\ell,q}^{V}\rangle\phi_{\ell,q}^{V}\,\dd V.
\end{equation}

\subsection{Rankin--Selberg convolutions}\label{RSSection}
In this subsection, we review briefly the properties of Rankin--Selberg $L$-functions. We shall restrict to automorphic representations for $\Gamma\backslash G$ on which the Hecke operator $T_i$ acts trivially, so that they lift to automorphic representations for $\PGL_2(\ZZ[i])\backslash\PGL_2(\CC)$. This allows us to refer to the theory of $\GL_2$.

The Rankin--Selberg $L$-function of two automorphic representations $V_j$ of type
$(\nu_j,p_j)\in i\RR\times \ZZ$ for $\Gamma\backslash G$ is defined by the absolutely convergent series (cf.~\eqref{RS-bound})
\begin{equation}\label{RSdef}
L(s, V_1\times V_2) =  \frac{1}{4}\zeta_{\QQ(i)}(2s)\sum_{n \in \ZZ[i] \setminus \{0\}}
\frac{\lambda_n(V_1)\lambda_n(V_2)}{(|n|^2)^s}, \qquad \Re s > 1.
\end{equation}
This can be verified by matching the Euler factors on the two sides, using \cite[Th.~15.1]{Jacquet1972}, \cite[Prop.~3.5]{JacquetLanglands}, \cite[(3.1.3)]{Tate1977}, and \cite[Lemma~1.6.1]{Bump1997}.
In particular,
\[L(s,V\times E(\nu,p))=L(s-\tfrac{1}{2}\nu,V\otimes\chi_p)L(s+\tfrac{1}{2}\nu,V\otimes\chi_{-p})\]
for $V$ cuspidal and $(\nu,p)\in i\RR\times 4\ZZ$ according to \eqref{eq:hecke-i-eigenvalue-eisenstein}, as well as
\[ L\big(s,E(\nu_1,p_1)\times E(\nu_2,p_2)\big)=\prod_{\epsilon_1,\epsilon_2\in\{\pm 1\}}L\big(s+\tfrac{1}{2}(\epsilon_1\nu_1+\epsilon_2\nu_2),\chi_{-\epsilon_1p_1-\epsilon_2p_2}\big),\]
with $(\nu_j,p_j)\in i\RR\times 4\ZZ$ and $\chi_p(z):=(z/|z|)^{-p}$. All $L$-functions are meant over $\QQ(i)$.

The Rankin--Selberg $L$-function $L(s,V_1\times V_2)$ possesses a meromorphic continuation to the entire complex plane with the exception of finitely many possible poles along the line $\Re s=1$. It is in fact entire except as follows (cf.\ \cite[Th.~2.2]{GelbartJacquet}):
\begin{itemize}
\item If $V_1 = V_2\,(=V)$ is cuspidal of type $(\nu,p)$ (that is, $(\nu_1,p_1)=\pm(\nu_2,p_2)$), there is a simple pole at $s=1$ with (strictly) positive residue
\begin{equation}\label{RS-res}
\mathop{\mathrm{res}}_{s=1} L(s, V\times V) = \frac{\pi}{4}\cdot L(1,\mathrm{ad^2}V) \gg_{\eps} \big((1 + |p|)(1+ |\nu|)\big)^{-\eps},
\end{equation}
where the lower bound follows from \cite[Prop.~3.2]{Maga2013}.
\item If $V_1$ and $V_2$ are both Eisenstein series with $p_1 = \epsilon p_2$ for some $\epsilon\in\{\pm 1\}$, there are simple poles at $s = 1+\eta(\nu_1 - \epsilon\nu_2)/2$  for $\eta\in\{\pm 1\}$ with residue
\begin{equation}
\label{RS-res-Eis}
\mcL_{\eta} (V_1,V_2):=\frac{\pi}{4}\cdot\zeta_{\QQ(i)}(1+\eta(\nu_1-\epsilon\nu_2))L(1+\eta \nu_1,\chi_{-2\eta p_1})L(1-\eta\epsilon\nu_2,\chi_{2\eta\epsilon p_2}),
\end{equation}
unless $\nu_1=\pm\nu_2$ or $\nu_1=0$ or $\nu_2=0$, in which case, however, the definition still makes sense as a meromorphic function of $\nu_1$ and $\nu_2$.
\end{itemize}

Finally, the associated completed $L$-function satisfies the familiar functional equation
\begin{equation}\label{func-eq}
\Lambda(s, V_1 \times V_2) := 16^sL(s, V_1\times V_2)L_{\infty} (s, V_1\times V_2) = \Lambda(1-s, V_1 \times V_2),
\end{equation}
where the exponential factor $16^s$ coming from the discriminant of $\QQ(i)$ is included for convenience, and the factor at infinity is given by
\begin{align}
\notag L_{\infty}(s, V_1\times V_2) = \Gamma(s,\vec{\nu},\vec{p})
:&=\prod_{\epsilon_1,\epsilon_2\in\{\pm 1\}}
L_\infty(s,\chi_{\epsilon_1\nu_1,\epsilon_1p_1}\cdot\chi_{\epsilon_2\nu_2,\epsilon_2p_2})\\
\label{inffactor} &=\prod_{\epsilon_1,\epsilon_2\in\{\pm 1\}}
\Gamma_{\CC}\left(s+\textstyle\frac{1}{2}(\epsilon_1\nu_1+\epsilon_2\nu_2)+\textstyle\frac{1}{2}|\epsilon_1p_1+\epsilon_2p_2|\right).
\end{align}
Here we used the abbreviations
\[\Gamma_{\CC}(s):=2(2\pi)^{-s}\Gamma(s),\qquad\vec{\nu}:=(\nu_1,\nu_2),\qquad\vec{p}:=(p_1,p_2).\]
Indeed, \eqref{func-eq}--\eqref{inffactor} follow from \cite[Prop.~18.2]{Jacquet1972}, \cite[\S 3]{Tate1977}, \cite[Prop.~6 in \S VII-2]{Weil1974} and its proof, upon noting that $V_j$ is isomorphic to the principal series representation induced from the pair of characters $(\chi_{-\nu_j,-p_j},\chi_{\nu_j,p_j})$.

\begin{lemma}\label{eis-pos}
Let $f : i\RR \rightarrow \CC$ be a function decaying as $f(\nu)\ll (1+|\nu|)^{-3}$, and let $p \in \ZZ$. Then
\[\int_{(0)}\int_{(0)}\sum_{\eta \in \{\pm 1\}}
f(\nu_1)\ov{f(\nu_2)}\,\mcL_{\eta}((\nu_1,p),(\nu_2,p))\,\frac{\dd\nu_1}{\pi i}\,\frac{\dd\nu_2}{\pi i}\geq 0.\]
\end{lemma}

\begin{proof} First we note that the $\eta$-sum cancels the individual poles of $\mcL_{\eta}((\nu_1,  p), (\nu_2,  p))$ at $\nu_1 = \nu_2$. For $\eps > 0$ and $V_j = (\nu_j,p)$ with $j\in\{1,2\}$ define
\[\mcL_\eta(V_1, V_2, \eps):=\frac{\pi}{4}\cdot\zeta_{\QQ(i)}(1+\eps+\eta(\nu_1-\nu_2))
L(1+\eta\nu_1,\chi_{-2\eta p})L(1-\eta\nu_2,\chi_{2\eta p})\]
and
\[\mcI(\eps):=\int_{(0)}\int_{(0)}\sum_{\eta\in\{\pm 1\}}
f(\nu_1)\ov{f(\nu_2)}\,\mcL_{\eta}(V_1,V_2,\eps)\,\frac{\dd\nu_1}{\pi i}\,\frac{\dd\nu_2}{\pi i}.\]
This function is continuous at $\eps = 0$, so it suffices to show $\mcI(\eps) \geq 0$ for $\eps > 0$.
Inserting the definition and opening the Dedekind zeta function, we see that
\[\mcI(\eps) = \frac{\pi}{16} \sum_{\eta \in \{\pm 1\}}   \sum_{n \in \ZZ[i]\setminus\{0\}}
\frac{1}{|n|^{2+2\eps}}
\biggl|\int_{(0)} \frac{1}{|n|^{2\eta\nu}}  L(1 + \eta \nu, \chi_{-2\eta p}) f(\nu) \frac{\dd\nu}{\pi i} \biggr|^2 \geq 0\]
as desired.
\end{proof}

\subsection{Diagonal detection of Voronoi type}\label{sec28}
In this subsection, we prove a Voronoi-type formula that allows us to detect equality of two automorphic representations occurring in $L^2(\Gamma\backslash G)$ in terms of a certain weighted orthogonality relation between their Hecke eigenvalues.
We shall use that only tempered representations occur in $L^2(\Gamma\backslash G)$, e.g.\ by \cite[Ch.~7, Prop.~6.2]{ElstrodtGrunewaldMennicke1998}.

\begin{lemma}\label{lemma-vor}
Let $P \geq 1$ be a parameter. There exists a function
\[W_P:\RR_{>0}\times\CC^2\times\ZZ^2\to\CC,\]
given explicitly by \eqref{defW}, with the following properties.
\begin{enumerate}[(a)]
\item\label{vor-a} $W_{P}(x,\vec{\nu},\vec{p})$ is an entire function of $\vec{\nu}=(\nu_1,\nu_2)\in\CC^2$, and it is invariant under
\[ (\nu_j,p_j)\mapsto (-\nu_j,-p_j)\quad\text{as well as}\quad (\nu_j,p_j)\mapsto (p_j,\nu_j)\quad (\nu_j\in \ZZ). \]
\item\label{vor-b} Let us abbreviate $\tilde P:=\bigl(1+|p_1+p_2|\bigr)\bigl(1+|p_1-p_2|\bigr)$. Then for every
$A>|\Re\nu_1|+|\Re\nu_2|$ we have
\begin{equation}\label{W1}
W_{P}(x,\vec{\nu},\vec{p})\ll_{A,\Re\nu_1,\Re\nu_2}\bigl(1+(\tilde P/P)^{2A-2}\bigr)\bigl(1+|\nu_1|+|\nu_2|\bigr)^{4A}x^{-A}.
\end{equation}
\item\label{vor-c} For every two automorphic representations $V_j$ of type $(\nu_j,p_j)\in i\RR\times\ZZ$ for $\Gamma\backslash G$ we have
\begin{equation}\label{vor-formula}
\begin{split}
&\sum_{n\in\ZZ[i]\setminus\{0\}}W_P\left(\frac{|n|}{P},\vec{\nu},\vec{p}\right)\lambda_n(V_1)\lambda_n(V_2)\\
&=\begin{cases}
\frac{\pi}{4}L(1,\mathrm{ad^2}V_1)P^2,&V_1=V_2\text{ cuspidal};\\
\sum_{\eta\in\{\pm 1\}}\mcL_{\eta}(V_1,V_2)P^{2+\eta(\nu_1-\epsilon\nu_2)},&V_1,V_2\text{ Eisenstein, }p_1=\epsilon p_2,\,\,\epsilon\in\{\pm 1\};\\
0,&\text{otherwise,}
\end{cases}
\end{split}
\end{equation}
where $L(1,\mathrm{ad^2}V_1)$ and $\mcL_{\eta}(V_1,V_2)$ are as in \eqref{RS-res} and \eqref{RS-res-Eis}.
\end{enumerate}
\end{lemma}

\begin{proof}
Let $w:\RR_{>0}\to\CC$ be a smooth function supported inside $[1,2]$, and normalized so that its Mellin transform $\widehat{w}(s)=\int_0^{\infty}w(x)x^s\,\dd x/x$ satisfies $\widehat{w}(1)=1$. We define
\begin{equation}\label{defW}
W_P(x,\vec{\nu},\vec{p}):=\frac1{8\pi i}\int_{(2)}\zeta_{\QQ(i)}(2s)\left(\widehat{w}(s)-P^{2-4s}\frac{16^{2s-1}\Gamma(s,\vec{\nu},\vec{p})}{\Gamma(1-s,\vec{\nu},\vec{p})}\widehat{w}(1-s)\right)x^{-2s}\,\dd s,
\end{equation}
where $\Gamma(s,\vec{\nu},\vec{p})$ is as in \eqref{inffactor}.

Shifting the contour to the far right, we see that $W_P(x,\vec{\nu},\vec{p})$ is entire in $\vec{\nu}$. The symmetry with respect to $(\nu_j, p_j) \mapsto (-\nu_j, -p_j)$ is obvious from \eqref{inffactor}. For $r\in\frac{1}{2}\ZZ$ we have the equality
\[\frac{\Gamma(z+r)}{\Gamma(1-z+r)}
= \frac{\Gamma(z-r)}{\Gamma(1-z-r)}\cdot\frac{\sin(\pi(z-r))}{\sin(\pi(z+r))}
= (-1)^{2r}\frac{\Gamma(z-r)}{\Gamma(1-z-r)}\]
of meromorphic functions in $z\in\CC$. This shows that (cf.\ \eqref{inffactor})
\begin{align*}
\frac{\Gamma(s,\vec{\nu},\vec{p})}{\Gamma(1-s,\vec{\nu},\vec{p})}
&=\prod_{\epsilon_1,\epsilon_2\in\{\pm 1\}}
\frac{\Gamma_{\CC}\left(s+\frac{1}{2}(\epsilon_1\nu_1+\epsilon_2\nu_2) + \frac{1}{2}|\epsilon_1p_1+\epsilon_2p_2|\right)}
{\Gamma_{\CC}\left(1-s-\frac{1}{2}(\epsilon_1\nu_1+\epsilon_2\nu_2)+\frac{1}{2}|\epsilon_1p_1+\epsilon_2p_2|\right)}\\
&=\prod_{\epsilon_1,\epsilon_2\in\{\pm 1\}}
\frac{\Gamma_{\CC}\left(s+\frac{1}{2}(\epsilon_1\nu_1+\epsilon_2\nu_2) + \frac{1}{2}(\epsilon_1p_1+\epsilon_2p_2)\right)}
{\Gamma_{\CC}\left(1-s-\frac{1}{2}(\epsilon_1\nu_1+\epsilon_2\nu_2)+\frac{1}{2}(\epsilon_1p_1+\epsilon_2p_2)\right)}
\end{align*}
is symmetric with respect to $(\nu_j, p_j) \mapsto (p_j, \nu_j)$, completing the proof of \ref{vor-a}.

Combining the first line of the previous display with \cite[Lemma~3.2]{Harcos2002}, we infer for $\Re(s)>\frac{1}{2}|\Re\nu_1|+\frac{1}{2}|\Re\nu_2|$ that
\begin{align*}
&\left|\frac{\Gamma(s,\vec{\nu},\vec{p})}{\Gamma(1-s,\vec{\nu},\vec{p})}\right|
=\prod_{\epsilon_1,\epsilon_2\in\{\pm 1\}}
\left|\frac{\Gamma_{\CC}\left(s+\frac{1}{2}(\epsilon_1\nu_1+\epsilon_2\nu_2) + \frac{1}{2}|\epsilon_1p_1+\epsilon_2p_2|\right)}
{\Gamma_{\CC}\left(1-\ov{s}-\frac{1}{2}(\epsilon_1\ov{\nu_1}+\epsilon_2\ov{\nu_2})+\frac{1}{2}|\epsilon_1p_1+\epsilon_2p_2|\right)}\right|\\
&\quad\ll_{\Re s,\Re\nu_1,\Re\nu_2}\prod_{\epsilon_1,\epsilon_2\in\{\pm 1\}}
\left|s+\tfrac12(\epsilon_1\nu_1+\epsilon_2\nu_2) + 
\tfrac12|\epsilon_1p_1+\epsilon_2p_2|\right|^{\Re(2s+\epsilon_1\nu_1+\epsilon_2\nu_2)-1}\\
&\quad\ll_{\Re s,\Re\nu_1,\Re\nu_2}\prod_{\epsilon_1,\epsilon_2\in\{\pm 1\}}
\left(1+|\epsilon_1p_1+\epsilon_2p_2|\right)^{\Re(2s+\epsilon_1\nu_1+\epsilon_2\nu_2)-1}
\left(|s|+|\nu_1|+|\nu_2|\right)^{\Re(2s+\epsilon_1\nu_1+\epsilon_2\nu_2)}\\
&\quad=\left(1+|p_1+p_2|\right)^{4\Re s-2}\left(1+|p_1-p_2|\right)^{4\Re s-2}\left(|s|+|\nu_1|+|\nu_2|\right)^{8\Re s}.
\end{align*}
Turning back to \eqref{defW}, the singularity of the integrand at $s = 1/2$ is removable, so we can shift the contour to $\Re s = A/2$. The bound \eqref{W1} follows upon noting that
that
\begin{itemize}
\setlength\itemsep{3pt}
\item $\widehat{w}(s) \ll_{C, \Re s} (1 + |s|)^{-C}$ for all $C > 0$ and $s\in\CC$;
\item $\zeta_{\QQ(i)}(2s) \ll (1 + |s|)^2$ for $\Re s > 0$ and $|2s-1|>1$.
\end{itemize}

Finally, to show \ref{vor-c}, we start from the following identity, a consequence of \eqref{RSdef}:
\[\frac{1}{16}\sum_{m,n\in\ZZ[i]\setminus\{0\}}w\left(\frac{|m|^4|n|^2}{P^2}\right)\lambda_n(V_1)\lambda_n(V_2)
= \frac{1}{2\pi i}\int_{(2)} L(s, V_1 \times V_2) \widehat{w}(s) P^{2s} \,\dd s.\]
We shift the contour to $\Re s = -1$; the contribution of the possible poles (on the line $\Re s = 1$) is recorded on the right-hand side of \eqref{vor-formula}. In the remaining integral we apply the functional equation \eqref{func-eq} and change variables $s \mapsto 1-s$ getting
$$\frac{1}{2\pi i}\int_{(2)} L(s, V_1 \times V_2) P^{2-4s} \frac{16^s\Gamma(s,\vec{\nu},\vec{p})}{16^{1-s}\Gamma(1-s,\vec{\nu},\vec{p})}\widehat{w}(1-s) P^{2s}\,\dd s.$$
Moving this term   to the other side, we obtain the desired formula \eqref{vor-formula}, first for $(\nu_1, p_1) \neq \pm (\nu_2, p_2)$, but then by analytic continuation everywhere.
This completes the proof of \ref{vor-c}.
\end{proof}

\section{Pre-trace formula and amplification}
In this section, we first implement a pre-trace setup, using integral kernels that are (by necessity) not bi-$K$-invariant, first in \S\ref{sec31} as the full pre-trace formula based on the theory of Eisenstein series and then as a streamlined pre-trace inequality in \S\ref{sec31b}. In \S\S\ref{tfa-subsec}--\ref{reduction}, we couple the pre-trace setup with either amplification by Hecke operators or self-amplification via diagonal detection of Voronoi type in \S\ref{sec28} to derive estimates on pointwise values of automorphic forms in terms of estimates on generalized spherical trace functions and Diophantine counts.

\subsection{Amplified pre-trace formula}\label{sec31}
In this subsection, we prove an amplified pre-trace formula based on the theory of Eisenstein series and the spectral decomposition of $L^2(\Gamma\backslash G)$ (see \S\ref{Eisenstein-subsec}). This is a familiar identity between spectral and geometric data, and its full force will be needed in the proof of Theorem~\ref{thm3}\ref{thm3-b}; in fact, as an even more general version, we shall use a double pre-trace formula (see \S\ref{sec:double-pre-trace-formula}) in two variables.

Let $A$ be a bounded operator on $L^2(\Gamma\backslash G)$ preserving the subspace $C_0^{\infty}(\Gamma\backslash G)$ of smooth functions with all rapidly decreasing derivatives. Assume that for the basis forms $\phi_{\ell,q}^{V}$, indexed as in \eqref{compactform} by $V$ occurring in $L^2(\Gamma\backslash G)$ (cuspidal or Eisenstein) and $\ell,q\in\ZZ$ satisfying $\ell\geq\max(|p_V|,|q|)$, there are constants $c_{\ell,q}^{V}(A)\in\CC$ such that
\begin{equation}\label{newconstants}
\langle A\psi,\phi_{\ell,q}^{V}\rangle=c_{\ell,q}^{V}(A)\langle \psi,\phi_{\ell,q}^{V}\rangle,\qquad
\psi\in C^{\infty}_0(\Gamma\backslash G).
\end{equation}
Then \eqref{eq:spectral-decomposition-plancherel} yields, for every $\psi\in C^{\infty}_0(\Gamma\backslash G)$,
\begin{equation}\label{Apsi-psi}
\langle A\psi,\psi\rangle=\frac{\langle A\psi,1\rangle\langle 1,\psi\rangle}{\vol(\Gamma\backslash G)}+\sum_{\ell\geq 0}\int_{[\ell]}\sum_{|q|\leq\ell}c_{\ell,q}^{V}(A)|\langle\psi,\phi_{\ell,q}^{V}\rangle|^2\,\dd V.
\end{equation}

For $f\in C_0(G)$ a rapidly decaying continuous function on $G$, and $\psi\in L^2(\Gamma\backslash G)$, we may consider the function $R(f)\psi\in L^2(\Gamma\backslash G)$ defined by
\begin{align*}
(R(f)\psi)(g)&:=\int_Gf(h)\psi(gh)\,\dd h=\int_Gf(g^{-1}h)\psi(h)\,\dd h\\
&=\int_{\Gamma\backslash G}k_f(g,h)\psi(h)\,\dd h,\qquad k_f(g,h):=\sum_{\gamma\in\Gamma}f(g^{-1}\gamma h).
\end{align*}
Thus $R(f)$ is a bounded integral operator on $L^2(\Gamma\backslash G)$ with kernel $k_f$. It is clear that $R(f)$ preserves $C_0^{\infty}(\Gamma\backslash G)$, and its adjoint equals $R(f)^\ast=R(f^\ast)$ with
\[f^\ast(g):=\ov{f(g^{-1})},\qquad g\in G.\]
Further, for a finitely supported sequence of complex coefficients $x=(x_n)_{n\in\ZZ[i]\setminus\{0\}}$, let $R_{\fin}(x)$ be the operator on $L^2(\Gamma\backslash G)$ given by
\begin{equation}\label{eq:def-amplifier}
R_{\fin}(x):=\sum_{n\in\ZZ[i]\setminus\{0\}}x_nT_n.
\end{equation}
The adjoint of this operator equals $R_{\fin}(x)^\ast=R_{\fin}(\ov{x})$.

Let us now fix an integer $\ell\geq 1$. Let $f\in\mcH(\tau_{\ell})_{\infty}$ be such that $f=f^\ast$, and let $x=(x_n)$ be as above such that $x=\ov{x}$, the self-adjointness conditions serving only to lighten the notation below. Further, let $V$ be a non-identity (cuspidal or Eisenstein) automorphic representation of arbitrary type $(\nu_V,p_V)$ occurring in $L^2(\Gamma\backslash G)$, and let $\ell',q\in\ZZ$ be such that $\ell'\geq\max(|p_V|,|q|)$. For $V$ cuspidal, \eqref{eq:opfromtr} and \eqref{eq:trpif} show that
\begin{equation}\label{joint-eigenfunctions}
\begin{aligned}
R(f)\phi_{\ell',q}^{V}&=\delta_{\ell'=\ell}\frac{\widehat{f}(V)}{2\ell+1}\phi_{\ell',q}^{V}, & & & & & \widehat{f}(V)&:=\widehat{f}(\nu_V,p_V);\\
R_{\fin}(x)\phi_{\ell',q}^{V}&=\widehat{x}(V)\phi_{\ell',q}^{V}, & & & & & \widehat{x}(V)&:=\sum_{n\in\ZZ[i]\setminus\{0\}}x_n\lambda_n(V).
\end{aligned}
\end{equation}
For $V$ Eisenstein, these equations are still valid with the obvious extension of $R(f)$ and $R_{\fin}(x)$ to functions in $C^\infty(\Gamma\backslash G)$ of moderate growth, as follows from \eqref{Eisenstein-Hecke-eigen} and the discussion between \eqref{eq:spectral-decomposition} and \eqref{eq:spectral-decomposition-EGM}. Therefore, following the usual argument that $R(f)$ and $R_{\fin}(x)$ are self-adjoint, we obtain that $A:=R(f)R_{\fin}(x)$ satisfies \eqref{newconstants} with
\[ c_{\ell',q}^{V}(A)=\delta_{\ell'=\ell}\frac{\widehat{f}(V)\widehat{x}(V)}{2\ell+1}.\]
Hence \eqref{Apsi-psi} holds with these coefficients and $\ell$-summation replaced by $\ell'$-summation. We note that the coefficients decay rapidly in $\nu$ by Theorem~\ref{thm:pws}. Moreover, $A(1)=R(f)(1)$ vanishes by $f=f\star\ov{\chi_{\ell}}$ and the orthogonality of characters (recalling that $\ell\geq 1$).

Applying \eqref{Apsi-psi} and recalling our observation below \eqref{Hecke-def} about $n\gamma^{-1}$ as $\gamma\in\Gamma\setminus\Gamma_n$, we obtain for every $\psi\in C_0^{\infty}(\Gamma\backslash G)$ that
\begin{align*}
\int_{[\ell]}\sum_{|q|\leq\ell}c_{\ell,q}^{V}(A)|\langle\psi,\phi_{\ell,q}^{V}\rangle|^2\,\dd V
&=\iint_{(\Gamma\backslash G)^2}k_f(g,h) \sum_{n\in\ZZ[i]\setminus\{0\}}x_n T_n\psi(h)\ov{\psi(g)}\,\dd g\,\dd h\\
&=\iint_{(\Gamma\backslash G)^2}\sum_{n\in\ZZ[i]\setminus\{0\}}\frac{x_n}{|n|}\sum_{\gamma\in\Gamma_n}f(g^{-1}\tilde{\gamma}h)\ov{\psi(g)}\psi(h)\,\dd g\,\dd h,
\end{align*}
where $\tilde\gamma$ abbreviates $\gamma/\sqrt{\det\gamma}$. Letting $\psi$ range through smooth, nonnegative, $L^{1}$-nor\-mal\-ized functions supported in increasingly small open neighborhoods of a fixed point $\Gamma g\in\Gamma\backslash G$, and taking limits using the rapid decay of $c_{\ell,q}^{V}(A)$, we obtain the desired \emph{amplified pre-trace formula}
\begin{equation}\label{APTF}
\int_{[\ell]}\frac{\widehat{f}(V)\widehat{x}(V)}{2\ell+1}\sum_{|q|\leq\ell}|\phi_{\ell,q}^{V}(g)|^2\,\dd V=\sum_{n\in\ZZ[i]\setminus\{0\}}\frac{x_n}{|n|}\sum_{\gamma\in\Gamma_n}f(g^{-1}\tilde{\gamma}g).
\end{equation}

The pre-trace formula \eqref{APTF} isolates forms $\phi_{\ell,q}^{V}$ with a specific value of $\ell$ (thus, forms in the 
chosen constituent $V^{\ell}$ in the decomposition \eqref{decomp} for various $V$'s), a starting point for a proof of Theorem~\ref{thm1}. To further isolate eigenforms in the specific constituent $V^{\ell,q}$ (for a fixed $|q|\leqslant\ell$), starting from our earlier $f\in\mcH(\tau_{\ell})_{\infty}$ satisfying $f=f^\ast$, we define a smooth function $f_q\in C_0(G)$ by
\begin{equation}\label{fqg-average-2}
f_q(g):=\frac1{2\pi}\int_0^{2\pi}f\big(g\diag(e^{i\varrho},e^{-i\varrho})\big)\,e^{2qi\varrho}\,\dd \varrho =\frac1{2\pi}\int_0^{2\pi}f\big(\diag(e^{i\varrho},e^{-i\varrho})g\big)\,e^{2qi\varrho}\,\dd \varrho.
\end{equation}
We note that $f_q=f_q^\ast$, but $f_q$ need not lie in $\mcH(\tau_{\ell})_{\infty}$. By the orthogonality of characters on $\RR/\ZZ$, we have
\begin{equation}\label{projectq}
R(f_q)=R(f)\Pi_q=\Pi_q R(f),
\end{equation}
where $\Pi_q$ is the projection onto the closed subspace consisting of $\psi\in L^2(\Gamma \backslash G)$ such that $\psi(g\diag(e^{i\varrho},e^{-i\varrho}))=e^{2qi\varrho}\psi(g)$. In particular, $R(f_q)$ is a bounded, self-adjoint operator, which preserves $C_0^{\infty}(\Gamma\backslash G)$. Moreover, by \eqref{joint-eigenfunctions} and the surrounding discussion,
\[R(f_q)\phi_{\ell',q'}^{V}= \delta_{(\ell',q')=(\ell,q)}\frac{\widehat{f}(V)}{2\ell+1}\phi_{\ell',q'}^{V}\]
holds for $V$ cuspidal, and also for $V$ Eisenstein with the obvious extension of $R(f_q)$ to functions in $C^\infty(\Gamma\backslash G)$ of moderate growth. Thus, applying as above \eqref{Apsi-psi} with $A=R(f_q)R_{\fin}(x)$, we obtain the following amplified pre-trace formula for individual forms:
\begin{equation}\label{APTF-single-form}
\int_{[\ell]}\frac{\widehat{f}(V)\widehat{x}(V)}{2\ell+1}|\phi_{\ell,q}^{V}(g)|^2\,\dd V=\sum_{n\in\ZZ[i]\setminus\{0\}}\frac{x_n}{|n|}\sum_{\gamma\in\Gamma_n}f_q(g^{-1}\tilde{\gamma}g).
\end{equation}

We proved \eqref{APTF} and \eqref{APTF-single-form} for every $f\in\mcH(\tau_{\ell})_{\infty}$ and finitely supported $x=(x_n)$ under the assumption that $f=f^\ast$ and $x=\ov{x}$. In fact \eqref{APTF} and \eqref{APTF-single-form} hold without this assumption, because both sides are $\CC$-linear in $f$ and $x$. Alternatively, one can modify the above proof to work without the self-adjointness assumption, starting with the analogue of \eqref{joint-eigenfunctions} for $R(f^\ast)\phi_{\ell',q}^{V}$ and $R_\fin(\ov{x})\phi_{\ell',q}^{V}$.

\subsection{Positivity and amplified pre-trace inequality}\label{sec31b}
In many situations, the coefficients on the left-hand (spectral) side of \eqref{APTF} and \eqref{APTF-single-form} are nonnegative, and the pre-trace formula is simply used as an inequality, by dropping all but the terms of interest. This is the case for the proofs of Theorems~\ref{thm1},~\ref{thm2}~and~\ref{thm3}\ref{thm3-a}. In this subsection, we derive such amplified pre-trace inequalities in a streamlined way with substantially less heavy machinery, drawing inspiration from \cite[\S 3]{BlomerHarcosMagaMilicevic2020}. For example, here we do not even need to mention Eisenstein series.

Let $A$ be a positive operator operator on $L^2(\Gamma\backslash G)$, and let $\mfB$ be a finite orthonormal system of eigenfunctions $\phi$ of $A$ with (not necessarily distinct) eigenvalues $(c_{\phi}(A))_{\phi\in\mfB}$. Then, $A$ preserves the orthodecomposition
$L^2(\Gamma\backslash G)=\mathrm{Span}(\mfB)\oplus\mathrm{Span}(\mfB)^{\perp}$,
and for any $\psi\in L^2(\Gamma\backslash G)$ the corresponding decomposition $\psi=\psi_1+\psi_2$ with
\[\psi_1:=\sum_{\phi\in\mfB}\langle\psi,\phi\rangle\phi\qquad \text{and} \qquad \psi_2:=\psi-\psi_1\]
gives
\begin{equation}\label{positivity}
\langle A\psi,\psi\rangle=\langle A\psi_1,\psi_1\rangle+\langle A\psi_2,\psi_2\rangle
\geq\langle A\psi_1,\psi_1\rangle=\sum_{\phi\in\mfB}c_{\phi}(A)|\langle\psi,\phi\rangle|^2.
\end{equation}

We will apply this positivity argument to the operators $A=R(f) R_{\fin}(x)$ and $A=R(f_q) R_{\fin}(x)$, where $f\in\mcH(\tau_{\ell})_{\infty}$ and $x=(x_n)$ are as in the previous subsection. Positivity is achieved by making the operators $R(f)$ and $R_{\fin}(x)$ individually positive, because Hecke operators commute with integral operators, and $\Pi_q$ in \eqref{projectq} is a positive operator commuting with $R(f)$. For the positivity of $R(f)$, it suffices that
\begin{equation}\label{positivity-assumption}
f=u\star u\qquad\text{for some $u\in\mcH(\tau_{\ell})_{\infty}$ satisfying $u=u^\ast$}.
\end{equation}
For the positivity of $R_{\fin}(x)$, it suffices that
\begin{equation}
\label{gen-amplifier}
\begin{gathered}
R_{\fin}(x)=\Bigl(\sum_{l\in P}y_lT_l\Bigr) \star \Bigl(\sum_{m\in P}\ov{y_m}T_m\Bigr) +
\Bigl(\sum_{l\in P}z_lT_{l^2}\Bigr) \star \Bigl(\sum_{m\in P}\ov{z_m}T_{m^2}\Bigr),\\[4pt]
x_n:=\sum_{\substack{l,m\in P\\ (d)|(l,m)\\ lm/d^2=n}}y_l\ov{y_m}+
\sum_{\substack{l,m\in P\\ (d)|(l^2,m^2)\\ l^2m^2/d^2=n}}z_l\ov{z_m},
\end{gathered}
\end{equation}
where $(y_l)_{l\in P}$ and $(z_l)_{l\in P}$ are arbitrary complex coefficients supported on a finite set
$P\subset\ZZ[i]\setminus\{0\}$. Here we used that each Hecke operator $T_n$ is self-adjoint.

Now, let $V$ be a cuspidal automorphic representation that occurs in $L^2(\Gamma\backslash G)$ and contains $\tau_{\ell}$-type vectors. Let $\mfB=\{ \phi_q : |q| \leq \ell\}$ be an orthonormal basis of $V^\ell$, with $\phi_q\in V^{\ell,q}$. As in the previous subsection, we evaluate the left-hand side of \eqref{positivity} geometrically, and then apply a limit  in $\psi$ to both sides. This way we obtain the following \emph{amplified pre-trace inequalities} in place of \eqref{APTF} and \eqref{APTF-single-form}:
\begin{align}\label{APTI}
\frac{\widehat{f}(V)\widehat{x}(V)}{2\ell+1}\sum_{\phi\in\mfB}|\phi(g)|^2&\leq\sum_{n\in\ZZ[i]\setminus\{0\}}\frac{x_n}{|n|}\sum_{\gamma\in\Gamma_n}f(g^{-1}\tilde{\gamma}g),\\
\label{APTI-single-form}
\frac{\widehat{f}(V)\widehat{x}(V)}{2\ell+1}|\phi_q(g)|^2&\leq\sum_{n\in\ZZ[i]\setminus\{0\}}\frac{x_n}{|n|}\sum_{\gamma\in\Gamma_n}f_q(g^{-1}\tilde{\gamma}g).
\end{align}

\subsection{Test functions and amplifier}\label{tfa-subsec}
The main idea of the amplified pre-trace inequality \eqref{APTI} is that it can provide a good upper bound for $\sum_{\phi\in\mfB}|\phi(g)|^2$ as long as the test function $f\in\mcH(\tau_{\ell})_{\infty}$ and the amplifier $x=(x_n)$ in \S\ref{sec31b} are chosen so that $\widehat{f}(V)$ and $\widehat{x}(V)$ are sizeable while the right-hand side is not too large. In this subsection, we make these choices.

As in Theorems~\ref{thm1},~\ref{thm2}~and~\ref{thm3}, let $\ell\geq 1$ be an integer, $I\subset i\RR$ and $\Omega\subset G$ be compact sets. Let $V\subset L^2(\Gamma\backslash G)$ be a cuspidal automorphic representation with minimal $K$-type $\tau_{\ell}$ and spectral parameter $\nu_V\in I$. Let us introduce the spectral weights
\begin{equation}\label{eq:def-gaussian-spectral-weight}
h(\nu,p):=\begin{cases} e^{(p^2-\ell^2+\nu^2)/2},\qquad &\text{$\nu\in\CC$, \quad $p\in \frac1{2}\ZZ$, \quad $|p|\leq\ell$,}\\ 0,&\text{$\nu\in\CC$, \quad $p\in \frac1{2}\ZZ$, \quad $|p|>\ell$.}\end{cases}
\end{equation}
According to Theorems~\ref{thm:tr-plancherel}~and~\ref{thm:pws}, the inverse $\tau_\ell$-spherical transform $f:=\widecheck{h}$ given by \eqref{eq:inverse-tauell-transform} belongs to $\mcH(\tau_{\ell})_{\infty}$, and it satisfies $\widehat{f}=h$. Moreover, if we set $u:=\widecheck{v}$ with
\[v(\nu,p):=\begin{cases} (2\ell+1)^{1/2}e^{(p^2-\ell^2+\nu^2)/4},\qquad &\text{$\nu\in\CC$, \quad $p\in \frac1{2}\ZZ$, \quad $|p|\leq\ell$,}\\ 0,&\text{$\nu\in\CC$, \quad $p\in \frac1{2}\ZZ$, \quad $|p|>\ell$,}\end{cases}\]
then $u\in\mcH(\tau_{\ell})_{\infty}$, $u=u^\ast$ by \eqref{eq:spherical-function-symmetry} and \eqref{eq:inverse-tauell-transform}, and $\widehat{f}=\widehat{u}^2/(2\ell+1)=\widehat{u\star u}$. This shows that \eqref{positivity-assumption} is satisfied. Hence $R(f)$ is the kind of positive operator considered in \S\ref{sec31b}, and
by \eqref{joint-eigenfunctions} we have
\begin{equation}\label{lower-bd-arch}
\widehat{f}(V)=h(\nu_V,\ell)\gg_I 1.
\end{equation}
With the notation \eqref{eq:def-tilde-f-s-p}, we have
\[
\widetilde{f}(s,p)=\sqrt{2\pi}(p^2+1-s^2)e^{(p^2-\ell^2-s^2)/2},
\]
whence by \eqref{eq:inverse-spherical-transform-estimate}, \eqref{eq:inverse-tauell-transform}, and the trivial bound
$\bigl|\varphi_{\nu,p}^{\ell}(g^{-1})\bigr|\leq 2\ell+1$, we have
\begin{equation}\label{better-be-bounded}
f(g) \ll \ell^2e^{-\log^2\|g\|}.
\end{equation}
We shall also use the following supplement, a consequence of \eqref{eq:spherical-function-symmetry}
and \eqref{eq:inverse-tauell-transform}:
\begin{equation}\label{in-the-bulk}
f(g)\ll \ell\sup_{\nu\in i\RR}\,\bigl|\varphi_{\nu,\ell}^{\ell}(g)\bigr|+\ell^{-50}.
\end{equation}

We now choose our amplifier, which we do as in \cite[\S5]{BlomerHarcosMilicevic2016}. Let $L\geq 7$ be a parameter, to be chosen at the very end of the proof of Theorems~\ref{thm1},~\ref{thm2}~and~\ref{thm3}, and set
\begin{gather*}
P(L):=\left\{\text{$l\in\ZZ[i]$ prime : $0<\arg(l)<\tfrac{\pi}4$ and $L\leq |l|^2\leq 2L$}\right\};\\
y_l:=\sgn(\lambda_{l}(V)),\qquad z_l:=\sgn(\lambda_{l^2}(V)),\qquad l\in P(L).
\end{gather*}
It follows from the result of Breusch~\cite[Teil II]{Breusch1932} (or from the prime number theorem for arithmetic progressions, for sufficiently large $L$) that $P(L)\neq\emptyset$, while
in \eqref{eq:def-amplifier} and \eqref{gen-amplifier} we have
\begin{equation}
\label{ampl-expand}
x_n=\begin{cases}
\sum\nolimits_{l\in P(L)}(y_l^2+z_l{}^2)\ll L/\log L,&n=1;\\
(1+\delta_{l_1\neq l_2})y_{l_1}y_{l_2}+\delta_{l_1=l_2}z_{l_1}z_{l_2}\ll 1,&\text{$n=l_1l_2$ for some $l_1,l_2\in P(L)$};\\
(1+\delta_{l_1\neq l_2})z_{l_1}z_{l_2}\ll 1,&
\text{$n=l_1^2l_2^2$ for some $l_1,l_2\in P(L)$};\\
0,&\text{otherwise}.
\end{cases}
\end{equation}
This formula is the analogue of \cite[(9.16)]{BlomerHarcosMagaMilicevic2020}, except that we forgot to insert the factors $1+\delta_{l_1\neq l_2}$ there.
In particular, by the inequality $|\lambda_{l}(V)|+|\lambda_{l^2}(V)|>1/2$ that follows from \eqref{hecke-mult}, we have
\begin{equation}
\label{lower-bd-non-arch}
\widehat{x}(V)=\Bigl(\sum_{l\in P(L)}|\lambda_{l}(V)|\Bigr)^2+
\Bigl(\sum_{l\in P(L)}|\lambda_{l^2}(V)|\Bigr)^2\gg \frac{L^2}{\log^2L}.
\end{equation}

Let $\mfB$ be an orthonormal basis of $V^\ell$. Entering the lower bounds \eqref{lower-bd-arch} and \eqref{lower-bd-non-arch} into the amplified pre-trace inequality \eqref{APTI}, we obtain
\begin{equation}\label{pre-APTI-done}
\frac{L^{2-\eps}}{\ell}\sum_{\phi\in\mfB}|\phi(g)|^2\ll_{\eps,I}\sum_{n\in\ZZ[i]\setminus\{0\}}\frac{|x_n|}{|n|}\sum_{\gamma\in\Gamma_n}|f(g^{-1}\tilde\gamma g)|.
\end{equation}
Let us assume that $g\in\Omega$. A straightforward counting combined with the divisor bound shows that
\begin{equation}\label{straightforward}
\#\left\{\gamma\in\Gamma_n:\|g^{-1}\tilde\gamma g\|\leq R\right\}\ll_{\eps,\Omega} R^{4+\eps}|n|^{2+\eps},
\end{equation}
so that, splitting into dyadic ranges for $\|g^{-1}\tilde\gamma g\|$ and using \eqref{better-be-bounded}, we obtain
\[ \sum_{\substack{\gamma\in\Gamma_n\\\log\|g^{-1}\tilde\gamma g\|>8\sqrt{\log\ell}}}
|f(g^{-1}\tilde\gamma g)|
\ll_{\eps,\Omega} \ell^{-50}|n|^{2+\eps}. \]
Thus from \eqref{in-the-bulk} and \eqref{pre-APTI-done} we conclude that
\begin{equation}\label{pre-APTI-done-2}
\begin{aligned}
&\sum_{\phi\in\mfB}|\phi(g)|^2\ll_{\eps,I,\Omega}L^{-2+\eps}\ell^2\sum_{\substack{n\in\ZZ[i]\setminus\{0\} \\ \gamma\in\Gamma_n\\\log\|g^{-1}\tilde\gamma g\|\leq 8\sqrt{\log\ell}}}\frac{|x_n|}{|n|}
\sup_{\nu\in i\RR}|\varphi_{\nu,\ell}^{\ell}(g^{-1}\tilde\gamma g)|+L^{2+\eps}\ell^{-48}.
\end{aligned}
\end{equation}
The bound \eqref{pre-APTI-done-2} explicitly reduces the non-spherical sup-norm problem of estimating $\sum_{\phi\in\mfB}|\phi(g)|^2$ via the amplification method to two ingredients:
\begin{itemize}
\item estimates on $\varphi_{\nu,\ell}^{\ell}(g^{-1}\tilde\gamma g)$ for $g^{-1}\tilde\gamma g\in G$ of moderate size;
\item counting $\gamma\in\Gamma_n$ according to the size of $\varphi_{\nu,\ell}^{\ell}(g^{-1}\tilde\gamma g)$.
\end{itemize}

We now also derive a version of \eqref{pre-APTI-done-2} adapted to estimating a single form $|\phi_q(g)|^2$ for some $|q|\leq\ell$. With the specific $f\in\mcH(\tau_{\ell})_{\infty}$ provided by \eqref{eq:inverse-tauell-transform} and \eqref{eq:def-gaussian-spectral-weight}, we obtain by averaging as in \eqref{fqg-average-2} the test function
\[ f_q(g):=\frac1{(2\ell+1)\pi^2}\sum_{|p|\leq\ell}\int_{0}^{\infty}e^{(p^2-\ell^2-t^2)/2}\,\varphi_{it,p}^{\ell,q}(g^{-1})\,(t^2+p^2)\,\dd t, \]
where
\[\varphi_{\nu,p}^{\ell,q}(g):=\frac1{2\pi}\int_0^{2\pi}
\varphi_{\nu,p}^{\ell}\bigl(g\diag(e^{i\varrho},e^{-i\varrho})\bigr)\,e^{-2qi\varrho}\,\dd \varrho.\]
In particular, this definition generalizes \eqref{spherical-averaged}, and
by \eqref{eq:spherical-function-symmetry} we have the symmetry
\begin{equation}\label{eq:averaged-spherical-function-symmetry}
\varphi_{\nu,p}^{\ell,-q}(g)=\ov{\varphi_{-\ov{\nu},p}^{\ell,q}(g)}=\varphi_{\nu,p}^{\ell,q}(g^{-1}).
\end{equation}
The analogues of \eqref{better-be-bounded}--\eqref{in-the-bulk} clearly hold for the $\RR/\ZZ$-average $f_q$, hence by \eqref{APTI-single-form} the following analogue of \eqref{pre-APTI-done-2} holds as well:
\begin{equation}
\label{pre-APTI-done-2-single-form}
\begin{aligned}
&|\phi_q(g)|^2\ll_{\eps,I,\Omega} L^{-2+\eps}\ell^2\sum_{\substack{n\in\ZZ[i]\setminus\{0\} \\ \gamma\in\Gamma_n\\\log\|g^{-1}\tilde\gamma g\|\leq 8\sqrt{\log\ell}}}\frac{|x_n|}{|n|}
\sup_{\nu\in i\RR}|\varphi_{\nu,\ell}^{\ell,q}(g^{-1}\tilde\gamma g)|+L^{2+\eps}\ell^{-48}.
\end{aligned}
\end{equation}

\subsection{A double pre-trace formula and a fourth moment}\label{sec:double-pre-trace-formula}
In this subsection, we use a different argument, outlined in \S\ref{KS-intro}, to estimate values $|\phi_q(g)|$ in terms of Diophantine counts of \emph{pairs} of Hecke correspondences and estimates on generalized spherical functions; see \eqref{almostdone} and \eqref{conclude-KS} below. 
The argument, reminiscent of self-amplification, relies on using diagonal detection of Voronoi type of \S\ref{sec28} in a double pre-trace formula (see \eqref{two-var} below) to get a handle on the fourth spectral moment of $|\phi_q(g)|$.

Let us fix two integers $\ell,q\in\ZZ$ with $\ell\geq\max(1,|q|)$. Let $n\in\ZZ[i]\setminus\{0\}$ and $g\in G$. By \eqref{APTF-single-form} and the remarks below it, for any $f\in\mcH(\tau_{\ell})_{\infty}$ we have
\[\int_{[\ell]}\widehat{f}(V)\lambda_n(V)|\phi_{\ell,q}^{V}(g)|^2\,\dd V=\frac{2\ell+1}{|n|}\sum_{\gamma\in\Gamma_n}f_q(g^{-1}\tilde{\gamma}g).\]
It is straightforward to adapt, first the two-variable versions of \eqref{eq:opfromtr}, \eqref{eq:trpif}, and \eqref{eq:spectral-decomposition-plancherel}, and then the proof of the above pre-trace formula to yield the following two-variable version.
Let $n_1,n_2\in\ZZ[i]\setminus\{0\}$ and $g_1,g_2\in G$. Then for any $f\in\mcH(\tau_{\ell},\tau_{\ell})_{\infty}$ (recalling the notation introduced after Theorem~\ref{thm:pws2}) we have
\begin{equation}\label{two-var}
\begin{aligned}
&\int_{[\ell]}\int_{[\ell]} \widehat{f}(V_1, V_2)\lambda_{n_1}(V_1)\lambda_{n_2}(V_2) |\phi^{V_1}_{\ell,q}(g_1)|^2|\phi^{V_2}_{\ell,q}(g_2)|^2 \,\dd V_1 \,\dd V_2\\
&\qquad=\frac{(2\ell+1)^2}{|n_1n_2|}\sum_{\gamma_1 \in \Gamma_{n_1}}\sum_{\gamma_2 \in \Gamma_{n_2}}
f_q(g_1^{-1}\tilde{\gamma}_1 g_1, g_2^{-1}\tilde{\gamma}_2 g_2),
\end{aligned}
\end{equation}
where $\widehat{f}(V_1, V_2)$ is given by \eqref{doubletransform} when $V_j$ is of type $(\nu_j,p_j)\in i\RR\times\ZZ$, and
\[f_{q}(g_1, g_2):=\frac1{(2\pi)^2}\int_0^{2\pi}\int_0^{2\pi}  f\big(g_1\diag(e^{i\varrho_1},e^{-i\varrho_1}), g_2\diag(e^{i\varrho_2},e^{-i\varrho_2})\big) \,e^{2q i(\varrho_1 + \varrho_2)}\,\dd \varrho_1 \,\dd \varrho_2.\]
In \eqref{two-var}, we can restrict to pairs $(V_1,V_2)$ satisfying $\lambda_i(V_j)=1$ by introducing an averaging over $\{n_1,in_1\}\times\{n_2,in_2\}$:
\begin{equation}\label{two-var-PGL}
\begin{aligned}
&\int_{[\ell]'}\int_{[\ell]'} \widehat{f}(V_1, V_2)\lambda_{n_1}(V_1)\lambda_{n_2}(V_2) |\phi^{V_1}_{\ell,q}(g_1)|^2|\phi^{V_2}_{\ell,q}(g_2)|^2 \,\dd V_1 \,\dd V_2\\
&\qquad=\frac{(2\ell+1)^2}{4|n_1n_2|}\sum_{\gamma_1\in\Gamma_{n_1}\cup\Gamma_{in_1}}\sum_{\gamma_2\in\Gamma_{n_2}\cup\Gamma_{in_2}}
f_q(g_1^{-1}\tilde{\gamma}_1 g_1, g_2^{-1}\tilde{\gamma}_2 g_2).
\end{aligned}
\end{equation}
The prime symbol in $[\ell]'$ indicates that we sum-integrate over automorphic representations with a lift to $\PGL_2(\ZZ[i])\backslash\PGL_2(\CC)$, so that the results of \S\ref{RSSection} and \S\ref{sec28} are applicable.

Now we consider, for any $n\in\ZZ[i]\setminus\{0\}$, the spectral weights
\[H(V_1,V_2;n):=h(\nu_1,p_1)h(\nu_2,p_2)W_\ell\left(\frac{|n|}{\ell},\vec{\nu},\vec{p}\right),\]
where $h$ is as in \eqref{eq:def-gaussian-spectral-weight} and $W_\ell$ is as in Lemma~\ref{lemma-vor}. Combining 
the Hilbert space isomorphism
\[\mcH(\tau_{\ell}) \hat{\otimes} \mcH(\tau_{\ell})
\longleftrightarrow L^2(\Gtemp(\tau_{\ell})\times \Gtemp(\tau_{\ell}))\]
induced by Theorem~\ref{thm:tr-plancherel} with Theorem~\ref{thm:pws2} and parts \ref{vor-a}--\ref{vor-b} of Lemma~\ref{lemma-vor}, we see that the function $(g_1,g_2)\mapsto\widecheck{H}(g_1,g_2;n)$ given by \eqref{inversedoubletransform} belongs to $\mcH(\tau_{\ell},\tau_{\ell})_{\infty}$, and its double $\tau_\ell$-spherical transform equals $H(V_1,V_2;n)$. Therefore, applying \eqref{two-var-PGL} with $f=\widecheck{H}(\cdot,\cdot;n)$, $n_1=n_2=n$, and $g_1=g_2=g$, and then summing up over $n$, we arrive at
\begin{equation}\label{eval}
\begin{aligned}
&\sum_{n\in\ZZ[i]\setminus\{0\}}\int_{[\ell]'}\int_{[\ell]'}
H(V_1,V_2;n)\lambda_n(V_1)\lambda_n(V_2)|\phi^{V_1}_{\ell,q}(g)|^2|\phi^{V_2}_{\ell,q}(g)|^2\,\dd V_1\,\dd V_2\\
&\qquad=\sum_{n\in\ZZ[i]\setminus\{0\}}\frac{(2\ell+1)^2}{4|n|^2}\sum_{\gamma_1,\gamma_2\in\Gamma_n\cup\Gamma_{in}}
\widecheck{H}_q(g^{-1}\tilde{\gamma}_1 g, g^{-1}\tilde{\gamma}_2 g;n).
\end{aligned}
\end{equation}
By Lemma~\ref{lemma-vor}\ref{vor-c}, the left-hand side of \eqref{eval} equals
\begin{equation}\label{eval-cusp}
\frac{\pi}{4}\ell^2\sum_{\substack{\text{$V$ cuspidal} \\T_i(V)=1,\ |p_V|\leq \ell}}
h(\nu_V, p_V)^2 \,L(1,\mathrm{ad^2} V)\,|\phi^{V}_{\ell,q}(g)|^4\ +\ \text{Eis},
\end{equation}
where the term $\text{Eis}$ is the contribution of Eisenstein representations:
\[\begin{aligned}
\text{Eis}=\ell^2\smash[b]{\sum_{\epsilon,\eta\in\{\pm 1\}}\sum_{\substack{p\in 4\ZZ\\|p|\leq\ell}}}\int_{(0)}\int_{(0)}
&\ell^{\eta(\nu_1-\epsilon \nu_2)}\,h(\nu_1,\epsilon p) h(\nu_2,p)\,\mcL_{\eta}((\nu_1,\epsilon p),(\nu_2,p))\\
&\times|\phi^{E(\nu_1,\epsilon p)}_{\ell,q}(g)|^2|\phi^{E(\nu_2,p)}_{\ell,q}(g)|^2
\,\frac{\dd\nu_1}{\pi i}\,\frac{\dd\nu_2}{\pi i}.
\end{aligned}\]
We make a change of variable $(\nu_1,\nu_2,p)\mapsto(\eta\nu_1,\eta\epsilon\nu_2,\eta\epsilon p)$. By invariance, we can replace the resulting pairs $(\eta\nu_1,\eta p)$ and $(\eta\epsilon\nu_2,\eta\epsilon p)$ by $(\nu_1,p)$ and $(\nu_2,p)$, respectively. In this way we see that
\begin{align*}
\text{Eis}=4\ell^2\smash[b]{\sum_{\substack{p\in 4\ZZ\\|p|\leq\ell}}}\int_{(0)}\int_{(0)}
&\ell^{\nu_1-\nu_2}\,h(\nu_1,p) h(\nu_2,p)\,\mcL_{\eta}((\nu_1,p),(\nu_2,p))\\
&\times|\phi^{E(\nu_1,p)}_{\ell,q}(g)|^2|\phi^{E(\nu_2,p)}_{\ell,q}(g)|^2
\,\frac{\dd\nu_1}{\pi i}\,\frac{\dd\nu_2}{\pi i}.
\end{align*}
By Lemma~\ref{eis-pos}, we conclude that $\text{Eis}\geq 0$. In particular, the right-hand side of \eqref{eval} is real, and it provides an upper bound for the contribution of each cuspidal $V$ in \eqref{eval-cusp}:
\[h(\nu_V, p_V)^2 \,L(1,\mathrm{ad^2} V)\,|\phi^{V}_{\ell,q}(g)|^4
\ll\sum_{n\in\ZZ[i]\setminus\{0\}}\frac{1}{|n|^2}\sum_{\gamma_1,\gamma_2\in\Gamma_n\cup\Gamma_{in}}
\widecheck{H}_q(g^{-1}\tilde{\gamma}_1 g, g^{-1}\tilde{\gamma}_2 g;n).\]
Here we can restrict the $n$-sum to $|n|\leq\ell^{1+\eps}$ at the cost of an error of $\OO_\eps(\ell^{-50})$. Indeed, the contribution of $|n|>\ell^{1+\eps}$ on the two sides of \eqref{eval} are equal, and this contribution is $\OO_\eps(\ell^{-50})$ thanks to the bound $H(V_1, V_2;n)\ll_A(|n|/\ell)^{-A}$ for any $A > 0$ that follows from Lemma~\ref{lemma-vor}\ref{vor-b} and the exponential decay in \eqref{eq:def-gaussian-spectral-weight}.

We now further explicate this bound within the context of Theorems~\ref{thm1}--\ref{thm3} (in particular, in preparation for use in Theorem~\ref{thm3}\ref{thm3-b}). Let $I\subset i\RR$ and $\Omega\subset G$ be compact subsets. We fix a cuspidal automorphic representation
$V\subset L^2(\Gamma\backslash G)$ with $\nu_V\in I$, $p_V=\ell$, $\lambda_i(V)=1$, and we pick a cusp form $\phi_q\in V^{\ell,q}$ with ${\|\phi_q\|}_2=1$. We shall also assume that $g\in\Omega$. By \eqref{RS-res} and our findings above,
\begin{equation}\label{almostdone}
|\phi_q(g)|^4\ll_{\eps,I}\ell^\eps\sum_{\substack{n\in\ZZ[i]\setminus\{0\}\\|n|\leq\ell^{1+\eps}}}\frac{1}{|n|^2}\sum_{\gamma_1,\gamma_2\in\Gamma_n\cup\Gamma_{in}}\widecheck{H}_q(g^{-1}\tilde{\gamma}_1 g, g^{-1}\tilde{\gamma}_2 g;n)+\ell^{-50}.
\end{equation}
We will analyze the right-hand side of \eqref{almostdone} to localize $\gamma_1$, $\gamma_2$ which contribute non-negligibly, and to bound these contributions in terms of generalized spherical functions $\varphi_{\nu,\ell}^{\ell,q}$.

We estimate $\widecheck{H}$ (hence also $\widecheck{H}_q$) in terms of Cartan coordinates using the two-dimensional analogue of \eqref{eq:inverse-spherical-transform-estimate}:
\[\begin{split}
\big|\widecheck{H}(k_1a_{h_1}k_2,k_3a_{h_2}k_4;n)\big|\leq
& \sum_{|p_1|,|p_2|\leq\ell}\ \ \iint\limits_{\substack{s_1>h_1\\s_2>h_2}}\,
\Bigg|\ \iint\limits_{\substack{t_1\in\RR\\t_2\in\RR}}
W_\ell\left(\frac{|n|}{\ell},(it_1, it_2),(p_1,p_2)\right)\\
\times e^{-\ell^2+(p_1^2+p_2^2)/2}\,& e^{-(t_1^2 + t_2^2)/2}\,e^{-it_1s_1-it_2s_2}\,(t_1^2+p_1^2)(t_2^2+p_2^2)
\,\dd t_1\,\dd t_2\,\Bigg|\,\dd s_1\,\dd s_2.
\end{split}\]
This estimate holds for $k_j\in K$ and $h_j>1$. Assuming without loss of generality that $s_1\geq s_2>0$ and shifting the $t_1$-contour, we conclude from Lemma~\ref{lemma-vor}\ref{vor-b} that for any $\eps, B > 0$ the inner double integral is
\begin{align*}
&\ll_{\eps,B}\left(1 + \frac{(1 +|p_1+p_2|)(1 + |p_1- p_2|)}{\ell}\right)^{2B-2+2\eps} e^{-\ell^2+(p_1^2+p_2^2)/2}\ell^4 e^{-B s_1} \left(\frac{|n|}{\ell}\right)^{-B-\eps}\\
&\ll_{\eps,B}\frac{\ell^{4+\eps}e^{-B\max(s_1,s_2)}}{\bigl(1+\ell-|p_1|\bigr)^2\bigl(1+\ell-|p_2|\bigr)^2}\left(\frac{|n|}{\ell}\right)^{-B}. 
\end{align*}
It follows that
\[\widecheck{H}(k_1a_{h_1}k_2,k_3a_{h_2}k_4;n) 
\ll_{\eps,B} \ell^{4+\eps} e^{-B\max(h_1,h_2)} \bigl(\ell/|n|\bigr)^B \]
for any $\eps, B > 0$ and $h_1, h_2 > 1$. This estimate remains true for general $h_1, h_2 \geq 0$, as can be seen by using \eqref{eq:inverse-tauell-transform} and the trivial bound $|\varphi_{it,p}^{\ell}|\leqslant 2\ell+1$
instead of \eqref{eq:inverse-spherical-transform-estimate} for the respective variable if one or both of $h_1, h_2$ are
at most $1$. The same bound applies for $\widecheck{H}_q$, that is,
\[\widecheck{H}_q(g_1,g_2;n)
\ll_{\eps,B}\ell^{4+\eps}\bigg(\frac{\ell/|n|}{\|g_1\|^2+\|g_2\|^2}\bigg)^B\]
for any $\eps, B > 0$ and $g_1,g_2\in G$. So we can refine \eqref{almostdone} to
\[|\phi_q(g)|^4\ll_{\eps,I,\Omega}\ell^\eps\sum_{\substack{n\in\ZZ[i]\setminus\{0\}\\|n|\leq\ell^{1+\eps}}}
\frac1{|n|^2}\sum_{\substack{\gamma_1,\gamma_2\in\Gamma_n\cup\Gamma_{in}\\\| g^{-1}\tilde{\gamma}_j g\|\leq\ell^{\eps}\sqrt{\ell/|n|}}}\widecheck{H}_q(g^{-1}\tilde{\gamma}_1 g, g^{-1}\tilde{\gamma}_2 g;n)+\ell^{-50}.\]

In the last sum, we estimate the terms more directly by \eqref{inversedoubletransform}, \eqref{eq:averaged-spherical-function-symmetry}, and Lemma~\ref{lemma-vor}\ref{vor-b}:
\[\widecheck{H}_q(g^{-1}\tilde{\gamma}_1 g,g^{-1}\tilde{\gamma}_2 g;n)
\ll_\eps\ell^{2+\eps} F(\gamma_1)F(\gamma_2)+\ell^{-80},\]
where we temporarily abbreviate (suppressing $g$ and $q$ from the notation)
\[F(\gamma):=\sup_{\nu\in i\RR}|\varphi_{\nu,\ell}^{\ell,q}(g^{-1}\tilde{\gamma}g)|,\qquad\gamma\in\GL_2(\CC).\]
Recalling also \eqref{straightforward}, we obtain an inequality of bilinear type:
\[|\phi_q(g)|^4\ll_{\eps,I,\Omega}\ell^{2+\eps}\sum_{\substack{n\in\ZZ[i]\setminus\{0\}\\|n|\leq\ell^{1+\eps}}}
\frac1{|n|^2}\sum_{\substack{\gamma_1,\gamma_2\in\Gamma_n\cup\Gamma_{in}\\\|g^{-1}\tilde{\gamma}_j g\|\leq\ell^{\eps}\sqrt{\ell/|n|}}}F(\gamma_1)F(\gamma_2)+\ell^{-50}.\]
With the shorthand notation
\[S(n):=\sum_{\substack{\gamma\in\Gamma_n\\\| g^{-1}\tilde{\gamma} g\|\leq\ell^{\eps}\sqrt{\ell/|n|}}}F(\gamma), \]
we observe that the innermost sum in the previous display equals $(S(n)+S(in))^2$, hence it does not exceed $2(S(n)^2+S(in)^2)$. In the end, we conclude
\begin{equation}\label{conclude-KS}
|\phi_q(g)|^4\ll_{\eps,I,\Omega}\ell^{2+\eps}\sum_{\substack{n\in\ZZ[i]\setminus\{0\}\\|n|\leq\ell^{1+\eps}}}
\frac1{|n|^2}\sum_{\substack{\gamma_1,\gamma_2\in\Gamma_n\\\|g^{-1}\tilde{\gamma}_j g\|\leq\ell^{\eps}\sqrt{\ell/|n|}}}
\prod_{j=1}^2\sup_{\nu\in i\RR}|\varphi_{\nu,\ell}^{\ell,q}(g^{-1}\tilde{\gamma}_j g)|+\ell^{-50},
\end{equation}
which serves as an analogue of \eqref{pre-APTI-done-2-single-form}.

\subsection{Reduction to Diophantine counting}\label{reduction}
In this subsection, we input into the preliminary estimates \eqref{pre-APTI-done-2}, \eqref{pre-APTI-done-2-single-form} and \eqref{conclude-KS} the results of Theorems~\ref{thm4},~\ref{thm6}~and~\ref{thm5}, which provide the desired estimates on spherical trace functions. We shall assume (as we can) that $\ell$ is sufficiently large in terms of $\eps$.

We begin by explicating the estimate \eqref{pre-APTI-done-2} using \eqref{ampl-expand} and Theorem~\ref{thm4}. For $\mcL\geq 1$ and $\vec{\delta}=(\delta_1,\delta_2)\in\RR^{2}_{>0}$, let
\[D(L,\mcL):=\left\{n\in\ZZ[i]: \mcL\leq |n|^2\leq 16\mcL,\,\,
\begin{matrix}\text{$n=1$ or $n=l_1l_2$ or $n=l_1^2l_2^2$}\\\text{for some $l_1,l_2\in P(L)$}\end{matrix}\right\},\]
\[M(g,L,\mcL,\vec{\delta}):=\sum_{n\in D(L,\mcL)}\#\bigg\{\gamma\in\Gamma_n:
g^{-1}\tilde{\gamma}g=k\begin{pmatrix} z&u\\&z^{-1}\end{pmatrix}k^{-1}\,\,
\begin{matrix}\text{for some }k\in K,\,\,|z|\geq 1,\\
\min|z\pm 1|\leq\delta_1,\,\,|u|\leq\delta_2\end{matrix}\bigg\}.\]
Note that every element of $G$ is of the form $k\big(\begin{smallmatrix}z&u\\&z^{-1}\end{smallmatrix}\big)k^{-1}$ for some $k\in K$, $|z|\geq 1$, and $u\in\CC$. Indeed, such a decomposition is immediate with $k\in G$, $z\in\CC^{\times}$, and $u=0$ unless $z=\pm 1$, after which the claim follows by replacing $k$ by $k\big(\begin{smallmatrix} &-1\\1&\end{smallmatrix}\big)$ if needed and using the Iwasawa decomposition of $k$.

Thus to each $\gamma$ occurring in \eqref{pre-APTI-done-2} we may associate a dyadic vector $\vec{\delta}=(\delta_1,\delta_2)$ (that is, $\log_{2}\delta_{j}\in\ZZ$) such that $1/\sqrt{\ell}\leq\delta_j\leq\ell^\eps$ and $\delta_{j}$ are minimal such that $\gamma$ is counted in the corresponding $M(g,L,\mcL,\vec{\delta})$. Therefore, applying \eqref{ampl-expand} and the estimates of Theorem~\ref{thm4} in \eqref{pre-APTI-done-2} leads to the following result.

\begin{lemma}\label{APTI-done-lemma}
Let $\ell\geq 1$ be an integer, $I\subset i\RR$ and $\Omega\subset G$ be compact sets. Let $V\subset L^2(\Gamma\backslash G)$ be a cuspidal automorphic representation with minimal $K$-type $\tau_{\ell}$ and spectral parameter $\nu_V\in I$. Let $\mfB$ be an orthonormal basis of $V^\ell$, and let $g\in\Omega$. Then for any $L\geq 7$ and $\eps>0$ we have
\begin{align*}
\sum_{\phi\in\mfB}|\phi(g)|^2
&\ll_{\eps,I,\Omega}\ell^{3+\eps}L^{\eps}\sum_{\substack{\vec{\delta}\textnormal{ dyadic}\\1/\sqrt{\ell}\leq\delta_j\leq \ell^{\eps}}}\min\left(\frac1{\ell\delta_1^2},\frac1{\sqrt{\ell}\delta_2}\right)\\
&\qquad\times\bigg(\frac{M(g,L,1,\vec{\delta})}{L}+\frac{M(g,L,L^2,\vec{\delta})}{L^3}+\frac{M(g,L,L^4,\vec{\delta})}{L^4}\bigg)+L^{2+\eps}\ell^{-48}.
\end{align*}
\end{lemma}

Lemma~\ref{APTI-done-lemma} is free of any choices of the test function, amplifier, and spherical trace function. It reduces the estimation of $\sum_{\phi\in\mfB}|\phi(g)|^2$ to the Diophantine counting problem of estimating $M(g,L,\mcL,\vec{\delta})$ uniformly in $L$, $\mcL$, and $\vec{\delta}$.

Now, we similarly explicate the estimate \eqref{pre-APTI-done-2-single-form} using \eqref{ampl-expand} and Theorems~\ref{thm6}--\ref{thm5}\ref{thm5-a}. Recall the sets $\mcD\subset G$ and $\mcS\subset K\subset\mcN\subset G$ introduced before Theorem~\ref{thm5}. With $D(L,\mcL)$ as above, we define for $|q|\leq\ell$, $\mcL\geq 1$, $\delta>0$, and $\vec{\delta}=(\delta_1,\delta_2)\in\RR_{>0}^{2}$, the matrix counts
\begin{align*}
M^{\ast}_0(g,L,\mcL,\delta)&:=\sum_{n\in D(L,\mcL)}\#\left\{\gamma\in\Gamma_n:\dist(g^{-1}\tilde\gamma g,\mcS)\leq\delta,\,\,\frac{D(g^{-1}\tilde\gamma g)}{\|g^{-1}\tilde\gamma g\|^2}\ll\frac{\log\ell}{\sqrt{\ell}}\right\},\\
M^\ast(g,L,\mcL,\vec{\delta})&:=\sum_{n\in D(L,\mcL)}\#\left\{\gamma\in\Gamma_n:\dist\left(g^{-1}\tilde\gamma g,K\right)\leq\delta_1,\,\,\dist(g^{-1}\tilde\gamma g,\mcD)\leq\delta_2\right\},
\end{align*}
with a sufficiently large implied constant in the definition of $M^{\ast}_0(g,L,\mcL,\delta)$.

For $q=0$, we estimate the size of $\varphi_{\nu,\ell}^{\ell,q}(g^{-1}\tilde\gamma g)$ in \eqref{pre-APTI-done-2-single-form} using Theorem~\ref{thm5}\ref{thm5-a}. Since there are at most $\OO_{\eps,\Omega}(\ell^\eps|n|^{2+\eps})$ elements $\gamma\in\Gamma_n$ contributing to the right-hand side of \eqref{pre-APTI-done-2-single-form}, the total contribution of those elements which fail to satisfy $D(g^{-1}\tilde\gamma g)\ll\|g^{-1}\tilde\gamma g\|^2(\log\ell)/\sqrt{\ell}$ with a sufficiently large implied constant may be absorbed into the existing $\OO_{\eps,I,\Omega}(L^{2+\eps}\ell^{-48})$ error term. We may thus restrict to $\gamma\in\Gamma_n$ satisfying these conditions. We associate to each remaining $\gamma$ in \eqref{pre-APTI-done-2-single-form} the smallest dyadic $1/\sqrt{\ell}\leq\delta\leq\ell^{\eps}$ such that $\gamma$ is counted in the corresponding $M_0^{\ast}(g,L,\mcL,\delta)$. For a general $|q|\leq\ell$, we associate to each $\gamma$ in \eqref{pre-APTI-done-2-single-form} the
lexicographically smallest dyadic vector $\vec{\delta}=(\delta_1,\delta_2)$ such that $\delta_j\leq\ell^\eps$ and $\delta_1^2\delta_2\geq 1/\sqrt{\ell}$ and $\gamma$ is counted in the corresponding $M_q(g,L,\mcL,\vec{\delta})$. Applying \eqref{ampl-expand} and the estimates of Theorems~\ref{thm6}--\ref{thm5}\ref{thm5-a} in \eqref{pre-APTI-done-2-single-form} leads to the following result.

\begin{lemma}\label{APTI-done-lemma-single-form}
Let $\ell\geq 1$ be an integer, $I\subset i\RR$ and $\Omega\subset G$ be compact sets. Let $V\subset L^2(\Gamma\backslash G)$ be a cuspidal automorphic representation with minimal $K$-type $\tau_{\ell}$ and spectral parameter $\nu_V\in I$. Let $\phi_q\in V^{\ell,q}$ such that ${\|\phi_q\|}_2=1$ and let $g\in\Omega$. Then for any $L\geq 7$ and $\eps>0$ we have
\begin{align*}
|\phi_0(g)|^2
&\ll_{\eps,I,\Omega}\ell^{2+\eps}L^{\eps}\sum_{\substack{\delta\textnormal{ dyadic}\\1/\sqrt{\ell}\leq\delta\leq\ell^{\eps}}}\frac1{\sqrt{\ell}\delta}\\
&\qquad\times\bigg(\frac{M_0^{\ast}(g,L,1,\delta)}{L}+\frac{M_0^{\ast}(g,L,L^2,\delta)}{L^3}+\frac{M_0^{\ast}(g,L,L^4,\delta)}{L^4}\bigg)+L^{2+\eps}\ell^{-48}.
\end{align*}
Moreover, for $|q|\leq\ell$ we have
\begin{align*}
|\phi_q(g)|^2
&\ll_{\eps,I,\Omega}\ell^{2+\eps}L^{\eps}\sum_{\substack{\vec{\delta}\textnormal{ dyadic},\,\,
\delta_j\leq\ell^{\eps}\\\delta_1^2\delta_2\geq 1/\sqrt{\ell}}}
 \frac{1}{\sqrt{\ell} \delta_1^2 \delta_2}\\
&\qquad\times\bigg(\frac{M^\ast(g,L,1,\vec{\delta})}{L}+\frac{M^\ast(g,L,L^2,\vec{\delta})}{L^3}+\frac{M^\ast(g,L,L^4,\vec{\delta})}{L^4}\bigg)+L^{2+\eps}\ell^{-48}.
\end{align*}
\end{lemma}

Similarly, we explicate \eqref{conclude-KS} using Theorem~\ref{thm5}\ref{thm5-b}. Here we introduce the double matrix count
\begin{align*}
&Q(g,L,H_1, H_2):=\\
&\quad\sum_{L \leq |n|  \leq 2L} \#\Bigg\{(\gamma_1, \gamma_2) \in\Gamma_n^2 :\ 
\|g^{-1}\tilde\gamma_jg\|\leq \sqrt{\frac{H_j}{L}},
\,\dist(g^{-1}\tilde\gamma_jg,\mcD)\ll\sqrt{\frac{H_j\log\ell}{L\ell}}\Bigg\},
\end{align*}
with a sufficiently large implied constant in the distance condition.

\begin{lemma}\label{q=ell-case}
Let $\ell\geq 1$ be an integer,  $I\subset i\RR$ and $\Omega\subset G$ be compact sets. Let $V\subset L^2(\Gamma\backslash G)$ be a cuspidal automorphic representation with minimal $K$-type $\tau_{\ell}$ and spectral parameter $\nu_V\in I$. Suppose that $V$ lifts to an automorphic representation for $\PGL_2(\ZZ[i])\backslash\PGL_2(\CC)$.
Let $\phi_{\pm \ell}\in V^{\ell,\pm\ell}$ such that ${\|\phi_{\pm \ell}\|}_2=1$ and let $g\in\Omega$. Then for any $\eps > 0$ we have
\[|\phi_{\pm \ell}(g)|^4\ll_{ \eps,I,\Omega}\ell^{2+ \eps} \max_{1 \leq L, H_1, H_2 \leq \ell^{1+\eps}}  \frac{Q(g,L,H_1, H_2) }{H_1H_2}+\ell^{-50}.\]
\end{lemma}

\section{Proof of Theorem~\ref{thm4}}
In this section, we prove Theorem~\ref{thm4}. It is clear from the definition \eqref{eq:def-spherical-function}
that we can restrict to $k=1$ without loss of generality, and the first bound holds in the stronger form $|\varphi_{\nu,\ell}^{\ell}(g)|\leq 2\ell+1$. In particular, Theorem~\ref{thm4} is trivial for $\ell=1$, hence we shall assume (for notational simplicity) that $\ell\geq 2$. In addition, the exponential factor in \eqref{spherical-def} has absolute value less than $\|g\|^2$ thanks to
\eqref{eq:kappaH} and the identity
\[|ad-bc|^2+|a\bar b+c\bar d|^2=(|a|^2+|c|^2)(|b|^2+|d|^2),\]
hence it suffices to prove that
\begin{equation}\label{suffices}
\int_K  |\psi_{\ell}(\kappa(k^{-1} g k))|\, \dd k \ll_\eps \ell^\eps
\min\left(\frac{\|g\|^4}{|z^2 - 1|^2\ell}, \frac{\|g\|}{|u|\sqrt{\ell}}\right).
\end{equation}
Finally, we shall use the obvious fact that
\begin{equation}\label{eq:bounded_z_z(-1)_u}
|u|,|z|,|z^{-1}|\leq\|g\|.
\end{equation}

Writing $k=k[\phi,\theta,\psi]$ in Euler angles as in \eqref{decomp-K}, and setting
\[
x:= (z^2 - 1)\cos\theta + i e ^{-2i\phi} uz \sin\theta,
\]
one computes
\[
k[\phi,   \theta, \psi]^{-1} g k[\phi,  \theta, \psi] = \left(\begin{matrix}  (1 + x \cos\theta  )/z & \ast\\ \  -ie^{2i\psi}x \sin \theta  /z& \ast\end{matrix}\right).
\]
Our goal is to estimate then
\begin{equation}\label{eq:thm4_integral_to_bound}
\int_0^{\pi} \int_0^{\pi/2} \int_{-\pi}^{\pi} \left|\psi_{\ell}\left(\kappa\left(\begin{pmatrix}  (1 + x \cos\theta  )/z & \ast\\ \  -ie^{2i\psi}x \sin \theta  /z& \ast\end{pmatrix}\right)\right)\right| \sin 2\theta\,\dd \psi\,\dd \theta\,\dd \phi.
\end{equation}
We introduce the notation $\lambda:=\sqrt{\log\ell}$.

\subsection{Small values of the integrand}
First we identify a region where $|\psi_\ell|$ in the integral \eqref{eq:thm4_integral_to_bound} is small. Assume that
\begin{equation}\label{tiny}
\min\bigl(\tan\theta,|x|\sin\theta\bigr)>\frac{4\lambda}{\sqrt{\ell}}.
\end{equation}
Then in
\[
\kappa\left(\begin{pmatrix}  (1 + x \cos\theta  )/z & \ast\\ \  -ie^{2i\psi}x \sin \theta  /z& \ast\end{pmatrix}\right)=\begin{pmatrix} \frac{1 + x \cos\theta}{\sqrt{|1 + x \cos\theta|^2+|x\sin\theta|^2}}& \ast \\ \ast & \ast \end{pmatrix}\in K
\]
the upper left entry has absolute square less than $1-\lambda^2/\ell$, hence
\[\left|\psi_{\ell}\left(\kappa\left(\begin{pmatrix}  (1 + x \cos\theta  )/z & \ast\\ \  -ie^{2i\psi}x \sin \theta  /z& \ast\end{pmatrix}\right)\right)\right| < \left(1-\frac{\log \ell}{\ell}\right)^{\ell}<\frac{1}{\ell}.\]
In view of \eqref{eq:bounded_z_z(-1)_u}, this is admissible for \eqref{suffices}. In the next subsection, we consider the case when \eqref{tiny} fails.

\subsection{Large values of the integrand}
Assume first that $\tan\theta\leq 4\lambda/\sqrt{\ell}$. Then $\theta\leq 4\lambda/\sqrt{\ell}$, hence the corresponding contribution to \eqref{eq:thm4_integral_to_bound} is $\ll\lambda^2/\ell$. This is admissible for \eqref{suffices} in the light of \eqref{eq:bounded_z_z(-1)_u}.

Now assume that $|x|\sin\theta\leq 4\lambda/\sqrt{\ell}$, and decompose the relevant integration domain for $\theta$ as follows. For any $m,n\in\ZZ_{\geq 0}$ and $\phi\in[0,\pi]$, let
\[I(m,n,\phi):=\left\{\theta\in\left(0,\frac{\pi}{2}\right)\,:\,
|x|\sin\theta\leq\frac{4\lambda}{\sqrt{\ell}},\
\frac{1}{2}<2^m\sin\theta\leq 1,\ \frac{1}{2}<2^n\cos\theta\leq 1\right\}.\]
If $\theta\notin I(m,n,\phi)$ holds for every $0\leq m,n\leq 2\log\ell$, then $\sin 2\theta=2\sin\theta\cos\theta\leq 1/\ell$, which is admissible for \eqref{suffices}. Therefore, by \eqref{eq:bounded_z_z(-1)_u} and \eqref{eq:thm4_integral_to_bound}, it suffices to prove the bound
\begin{equation}\label{eq:integral_large_psi}
\int_0^{\pi} \int_{I(m,n,\phi)} \sin 2\theta\,\dd\theta\,\dd\phi \ll \min\left(\frac{\lambda^2}{\ell|z^2-1|^{2}},\frac{\lambda}{\ell^{1/2}|uz|}\right)
\end{equation}
for every $0\leq m,n\leq 2\log\ell$. We shall assume that $\min(m,n)=0$, for otherwise $I(m,n,\phi)=\emptyset$.
We record also that the Lebesgue measure of $I(m,n,\phi)$ is $\OO(2^{-m-n})$, because if $n=0$, then $\sin\theta\asymp \theta$, while if $m=0$, then $\cos\theta\asymp \pi/2-\theta$. Hence, for any $\phi\in[0,\pi]$, we have
\[\int_{I(m,n,\phi)} \sin 2\theta\,\dd \theta = \int_{I(m,n,\phi)} 2\sin\theta\cos\theta\,\dd \theta \ll 2^{-2m-2n}.\]

First consider the case when in $x=(z^2 - 1)\cos\theta + i e ^{-2i\phi} uz \sin\theta$, whose absolute value does not exceed $2^{m+3}\lambda/\sqrt{\ell}$, neither of the two summands is large:
\[
|z^2 - 1|2^{-n} \leq 2^{m+6} \frac{\lambda}{\sqrt{\ell}},\qquad |uz|2^{-m}\leq 2^{m+6} \frac{\lambda}{\sqrt{\ell}}.
\]
Recalling $\min(m,n)=0$, the previous two displays imply  for any $\phi\in[0,\pi]$ that
\[\int_{I(m,n,\phi)}\sin 2\theta\,\dd\theta \ll \min\left(\frac{\lambda^2}{\ell|z^2-1|^{2}},\frac{\lambda}{\ell^{1/2}|uz|}\right).\]
So in this case \eqref{eq:integral_large_psi} is clear.

Now consider the case when in $x=(z^2 - 1)\cos\theta + i e ^{-2i\phi} uz \sin\theta$,
whose absolute value does not exceed $2^{m+3}\lambda/\sqrt{\ell}$, the two summands are individually large:
\begin{equation}\label{eq:triangle_large_sides}
|z^2 - 1|2^{-n}> 2^{m+4} \frac{\lambda}{\sqrt{\ell}},\qquad |uz|2^{-m}> 2^{m+4} \frac{\lambda}{\sqrt{\ell}}, \qquad |z^2-1|2^{-n}\asymp |uz|2^{-m}.
\end{equation}
We   claim that this localizes $\phi$. Indeed, setting
\[
2\phi_0=\arg(iuz)-\arg(z^2-1),
\]
we see that
\[
|z^2-1|\cos \theta + e^{2i(\phi_0-\phi)}|uz|\sin \theta \ll 2^m\frac{\lambda}{\sqrt{\ell}},
\]
and comparing the imaginary parts, we have that
\[
\sin(2\phi-2\phi_0)\ll\frac{2^{2m}\lambda}{|uz|\sqrt{\ell}},\quad\text{and so}\quad
\phi\equiv\phi_0+\OO\left(\frac{2^{2m}\lambda}{|uz|\sqrt{\ell}}\right)\!\!\!\pmod{\pi/2}.
\]
Also, $\theta$ is localized, since
\[
|z^2-1|\cos \theta - |uz|\sin \theta \ll 2^m\frac{\lambda}{\sqrt{\ell}},
\]
and the first term here is monotone decreasing, the second one is monotone increasing in $\theta$. We see that $\theta$ is localized to an interval of length $
\OO(2^m\lambda/|uz|\sqrt{\ell})$ for $\sin\theta\leq \cos\theta$ (in which case $n=0$), and to an interval of length $\OO(\lambda/|z^2-1|\sqrt{\ell})$ for $\cos\theta\leq\sin\theta$ (in which case $m=0$).

We estimate the left-hand side of \eqref{eq:integral_large_psi} by exploiting the above localizations and all three parts of \eqref{eq:triangle_large_sides}. If $\sin\theta\leq\cos\theta$, then $n=0$ and $\sin 2\theta\leq 2^{1-m}$, so altogether we obtain a contribution to \eqref{eq:integral_large_psi} of size
\[\ll 2^{-m} \cdot \frac{2^m \lambda}{|uz|\sqrt{\ell}}\cdot \frac{2^{2m}\lambda}{|uz|\sqrt{\ell}}
\ll \min\left(\frac{\lambda^2}{|z^2-1|^2\ell},\frac{\lambda}{|uz|\sqrt{\ell}}\right).\]
Similarly, if $\cos\theta\leq\sin\theta$, then $m=0$ and $\sin 2\theta\leq 2^{1-n}$, so altogether we obtain a contribution to \eqref{eq:integral_large_psi} of size
\[\ll2^{-n} \cdot\frac{\lambda}{|z^2-1|\sqrt{\ell}}\cdot\frac{\lambda}{|uz|\sqrt{\ell}}
\ll \min\left(\frac{\lambda^2}{|z^2-1|^2\ell},\frac{\lambda}{|uz|\sqrt{\ell}}\right).\]

The proof of Theorem~\ref{thm4} is complete.

\section{Proof of Theorems~\ref{thm6} and \ref{thm5}}
In this section, we prove Theorems~\ref{thm6}~and~\ref{thm5}. We recall that the key player is the function
\begin{equation}\label{varphi-def-proof}
\varphi_{\nu,\ell}^{\ell,q}(g)  := \frac{1}{2\pi} \int_{0}^{2\pi} \varphi_{\nu,\ell}^{\ell}\big(gk[0, 0, \varrho]\big) \,e^{-2qi\varrho} \,\dd\varrho,
\end{equation}
where
\[\varphi_{\nu,\ell}^{\ell}(g):=(2\ell+1)
\int_K  \psi_{\ell}(\kappa(k^{-1} g k)) \,e^{(\nu-1)\rho(H(gk))}\,\dd k.\]
The function $\psi_\ell:K\to\CC$ was defined in \eqref{chi-ell}, but for calculational purposes we extend it now to $\GL_2(\CC)$:
\begin{equation}\label{psidef}\psi_{\ell}\left( \begin{pmatrix} \alpha  &\beta \\ \gamma & \delta \end{pmatrix} \right) :=
\bar{\alpha}^{2\ell}, \qquad \left( \begin{matrix} \alpha  &\beta \\ \gamma & \delta \end{matrix} \right) \in\GL_2(\CC).
\end{equation}

\subsection{Preliminary computations}\label{sec:preliminary-computations}
We write $g$ in Cartan form
\begin{equation}\label{g}
g=k[u_1,v_1,w_1]\begin{pmatrix}r & \\ & r^{-1}\end{pmatrix} k[u_2,v_2,w_2],
\end{equation}
where $r\geq 1$, and we allow $u_j,v_j,w_j\in\RR$ to be arbitrary for convenience.
Spelling out the definitions, and using that the height in the Iwasawa decomposition is left $K$-invariant, we see that $\varphi_{\nu,\ell}^{\ell,q}(g)$ equals
\[\begin{split}
\frac{d_{\ell}}{4\pi^3}\int_{\substack{0\leq u\leq \pi \\ 0\leq v\leq \pi/2 \\ 0\leq w\leq 2\pi \\ 0\leq \varrho\leq 2\pi}} & \psi_{\ell}\left( k[-w,-v,-u]k[u_1,v_1,w_1] \kappa\left(\begin{pmatrix}r&\\&r^{-1}\end{pmatrix} k[u_2,v_2,w_2]k[0,0,\varrho] k[u,v,w] \right)\right)\\
& \cdot e^{-2iq\varrho} \, e^{(\nu-1)\rho\left(H\left(\left(\begin{smallmatrix}r&\\&r^{-1}\end{smallmatrix}\right) k[u_2,v_2,w_2]k[0,0,\varrho]k[u,v,w]\right)\right)} \sin 2v \,\dd u\,\dd v\,\dd w\,\dd \varrho.
\end{split}\]
With a change of variables $k[u_2,v_2,w_2]k[0,0,\rho] k[u,v,w] \mapsto k[u,v,w]$ and dropping the normalized $w$-integration (which is legitimate since the conjugation by $k[0,0,w]$ does not alter the $\psi_{\ell}$-value, and the height in the Iwasawa decomposition is also unaffected by right-multiplication by $k[0,0,w]$), we arrive at
\[\begin{split}
\frac{d_{\ell}}{2\pi^2}\int_{\substack{0\leq u\leq \pi \\ 0\leq v\leq \pi/2 \\ 0\leq \varrho\leq 2\pi}} & \psi_{\ell}\left( k[0,-v,-u] k[u_2,v_2,w_2]k[0,0,\varrho] k[u_1,v_1,w_1] \kappa\left(\begin{pmatrix}r&\\&r^{-1}\end{pmatrix}k[u,v,0]\right)\right)\\
& \cdot e^{-2iq\varrho} \, e^{(\nu-1)\rho\left(H\left(\left(\begin{smallmatrix}r&\\&r^{-1}\end{smallmatrix}\right) k[u,v,0]\right)\right)} \sin 2v   \,\dd u\,\dd v\,\dd\varrho.
\end{split}\]
The sum of absolute squares in the first column of $\diag(r, r^{-1})k[u,v,0]$ equals
\[h(r,v):=r^2\cos^2v+r^{-2}\sin^2v,\]
hence recalling the definitions \eqref{eq:kappaH} and \eqref{psidef}, we can rewrite the integral as
\[\begin{split}
\frac{d_{\ell}}{2\pi^2}\int_{\substack{0\leq u\leq \pi \\ 0\leq v\leq \pi/2 \\ 0\leq \varrho\leq 2\pi}} & \psi_{\ell}\left( k[0,-v,-u] k[u_2,v_2,w_2]k[0,0,\varrho] k[u_1,v_1,w_1]
\begin{pmatrix}r&\\&r^{-1}\end{pmatrix}k[u,v,0]\right) \\
& \cdot e^{-2iq\varrho} \, h(r,v)^{\nu-1-\ell} \sin 2v \,\dd u\,\dd v\,\dd\varrho.
\end{split}\]
Replacing $\varrho$ by $\varrho-u_1-w_2$, the integral further simplifies to
\[\begin{split}
\frac{d_{\ell}e^{2iq(u_1+w_2)}}{2\pi^2}\int_{\substack{0\leq u\leq \pi \\ 0\leq v\leq \pi/2 \\ 0\leq \varrho\leq 2\pi}} & \psi_{\ell}\left( \left(\begin{matrix} e^{-i\varrho} I + e^{i\varrho}J & \ast \\ \ast & \ast \end{matrix}\right)\right) e^{-2iq\varrho} \, h(r,v)^{\nu-1-\ell} \sin 2v  \,\dd u\,\dd v\,\dd \varrho,
\end{split}\]
where
\begin{align*}
I:=&\left(r^{-1}e^{-2iu-iw_1}\sin v\cos v_1+re^{iw_1}\cos v\sin v_1\right)
\left(e^{2iu-iu_2}\sin v\cos v_2-e^{iu_2}\cos v\sin v_2\right),\\
J:=&\left(-r^{-1}e^{-2iu-iw_1}\sin v\sin v_1+re^{iw_1}\cos v\cos v_1\right)
\left(e^{2iu-iu_2}\sin v\sin v_2+e^{iu_2}\cos v\cos v_2\right).\end{align*}
Evaluating the $\varrho$-integral, we obtain
\begin{equation}\label{post-expansion}
\varphi_{\nu,\ell}^{\ell,q}(g) = \frac{d_{\ell}e^{2iq(u_1+w_2)}}{\pi}\binom{2\ell}{\ell + q}
\int_{\substack{0\leq u\leq \pi \\ 0\leq v\leq \pi/2 } } \frac{\sin 2v}{h(r,v)^{\ell+1-\nu}}
\,\bar{I}^{\ell + q} \bar{J}^{\ell - q}\,\dd u\,\dd v.
\end{equation}
Taking the complex conjugate of the right-hand side, and introducing the new variables $t:=r^{-1}\tan v$ and $\phi:=2u$, we get
\begin{align*}
\bigl|\varphi_{\nu,\ell}^{\ell,q}(g)\bigr| = \frac{d_{\ell}}{\pi}&\binom{2\ell}{\ell+q}
\bigg|\int_0^{\infty}\int_0^{2\pi} \frac{t}{(1 + (t/r)^2)^{\ell+1+\nu}(1+(tr)^2)^{\ell+1-\nu}}\\
&\left(e^{-i\phi-2iw_1}(t/r)\cos v_1+\sin v_1\right)^{\ell+q}
\left(e^{i\phi-2iu_2}(tr)\cos v_2-\sin v_2\right)^{\ell+q}\\
&\left(e^{-i\phi-2iw_1}(t/r)\sin v_1-\cos v_1\right)^{\ell-q}
\left(e^{i\phi-2iu_2}(tr)\sin v_2+\cos v_2\right)^{\ell-q}\,\dd\phi\,\dd t\bigg|.
\end{align*}
Now comes the last key step: in the inner $\phi$-integral, we can remove the $r$'s. This is so because $e^{-i\phi}$ must be chosen equally many times as $e^{i\phi}$, and the $r$'s will cancel out in all terms surviving the integration. Another way to see the same thing is to shift the contour as in $\phi\mapsto\phi+i\log r$ where the boundary terms cancel out by $2\pi$-periodicity. Either way, using also the opportunity to replace $\phi\mapsto\phi+u_2-w_1$,
and writing $\Delta:=u_2+w_1$, we finally obtain
\begin{align*}
\bigl|\varphi_{\nu,\ell}^{\ell,q}(g)\bigr| \leq \frac{d_{\ell}}{\pi} \binom{2\ell}{\ell+q}
\int_{0}^{\infty}&\frac{t}{((1 + (t/r)^2)(1+(tr)^2))^{\ell+1}}\\
\times\int_0^{2\pi}
&\bigl|e^{i\phi+i\Delta}t\cos v_1+\sin v_1\bigr|^{\ell+q}
\bigl|e^{i\phi -i\Delta}t\cos v_2-\sin v_2\bigr|^{\ell+q}\\
&\bigl|e^{i\phi+i\Delta}t\sin v_1-\cos v_1\bigr|^{\ell-q}
\bigl|e^{i\phi -i\Delta}t\sin v_2+\cos v_2\bigr|^{\ell-q}\,\dd\phi\,\dd t.
\end{align*}

We estimate the inner integrand using the following lemma, which is purely about inequalities. We state it formally so as to clearly separate issues. (In the case $q=\pm\ell$, all expressions raised to exponent 0 should simply be omitted.) As in the  previous section, we introduce the notation $\lambda:=\sqrt{\log \ell}$.

\begin{lemma}\label{young}
Let $\ell,q\in\ZZ$ be such that $\ell\geq\max(1,|q|)$. Let $X>0$ and $\Lambda>0$.
\begin{enumerate}[label=(\alph*)]
\item\label{part-a}
If $A,B\geq 0$ satisfy $A^2+B^2=X^2$, then
\begin{equation}
\label{ineq-two-terms}
\left(\frac{2\ell}{\ell+q}\right)^{(\ell+q)/2}\left(\frac{2\ell}{\ell-q}\right)^{(\ell-q)/2}A^{\ell+q}B^{\ell-q}\leq X^{2\ell}.
\end{equation}
Moreover, the left-hand side is $\OO_\Lambda(X^{2\ell}\ell^{-\Lambda})$ unless
\begin{equation}\label{cases-of-eq}
\begin{split}
A^2&=\frac{\ell+q}{2\ell}X^2+\OO_{\Lambda}\left(X^2\frac{\lambda^2+\lambda\sqrt{\ell-|q|}}{\ell}\right),\\
B^2&=\frac{\ell-q}{2\ell}X^2+\OO_{\Lambda}\left(X^2\frac{\lambda^2+\lambda\sqrt{\ell-|q|}}{\ell}\right).
\end{split}
\end{equation}
\item\label{part-b}
If $A,B,C,D\geq 0$ satisfy $A^2+B^2=C^2+D^2=X^2$, then
\[ \binom{2\ell}{\ell+q}A^{\ell+q}B^{\ell-q}C^{\ell+q}D^{\ell-q} \ll\frac{X^{4\ell}}{ 1+\sqrt{\ell-|q|}}. \]
Moreover, the left-hand side is $\OO_\Lambda(X^{4\ell}\ell^{-\Lambda})$ unless \eqref{cases-of-eq} and the analogous estimates for $C$, $D$ are satisfied.
\end{enumerate}
\end{lemma}

\begin{proof}
Let us first assume $|q| < \ell$. We use Young's inequality
\[xy \leq \frac{x^a}{a}  + \frac{y^b}{b} , \qquad \frac{1}{a} + \frac{1}{b}=1,\]
to conclude with
\begin{equation}\label{xy}
x := \left(\sqrt{\frac{2\ell}{\ell+q}}\frac{A}{X}\right)^{\frac{\ell+q}{\ell}}, \quad
y := \left(\sqrt{\frac{2\ell}{\ell-q}}\frac{B}{X}\right)^{\frac{\ell-q}{\ell}}, \quad
a := \frac{2\ell}{\ell + q}, \quad
b := \frac{2\ell}{\ell - q}
\end{equation}
that
\[
\left(\sqrt{\frac{2\ell}{\ell+q}}\frac{A}{X}\right)^{\frac{\ell+q}{\ell}}
\left(\sqrt{\frac{2\ell}{\ell-q}}\frac{B}{X}\right)^{\frac{\ell-q}{\ell}}
\leq \frac{A^2+B^2}{X^2}=1.
\]
This is equivalent to \eqref{ineq-two-terms}. We also conclude (still using the notation \eqref{xy}) that the left-hand side of \eqref{ineq-two-terms} is $\OO_\Lambda(X^{2\ell}\ell^{-\Lambda})$ unless
\begin{equation}\label{ineq}
xy>1/2,\qquad xy=1+\OO_\Lambda(\delta),\qquad\delta:=\lambda^2/\ell.
\end{equation}

Let us explore the consequences of \eqref{ineq}. First, by $x^a/a+y^b/b=1$ we have
\[1/3<x,y<3/2.\]
 Without loss of generality, $q\geq 0$ (i.e. $a\leq b$), and then $x^a<a\leq 2$. Moreover,
\[b\log x<(b/a)\log a<(b/a)(a-1)=1,\]
hence also $-b\log y<1+\OO_\Lambda(b\delta)$. In particular, $y^b\gg_\Lambda 1$ whenever $b\delta<1$. Now let us consider the function
\[F(t):=\frac{x^a}{a}+\frac{t^b}{b}-xt.\]
Note that $F(y)=1-xy$, and $F(y_0)=F'(y_0)=0$ for $y_0:=x^{a-1}$. Hence, using Lagrange's form for the remainder term in Taylor's theorem, we see that
\[\delta\gg_\Lambda F(y)\geq\frac{(b-1)}{2}\min(y_0^{b-2},y^{b-2})\,(y-y_0)^2.\]
Here $y_0^{b-2}=x^{2-a}\gg 1$. Now let us assume that $y^b>1$ or $b\delta<1$. Then $y^b\gg_\Lambda 1$, whence $y-y_0\ll_\Lambda\sqrt{\delta/b}$ by the previous display. From here and \eqref{ineq} we get the following two approximations for $bxy$:
\begin{align*}
bxy&=bxy_0+\OO_\Lambda(\sqrt{b\delta})=bx^a+\OO_\Lambda(\sqrt{b\delta}),\\
bxy&=b+\OO_\Lambda(b\delta)=(b-1)x^a+y^b+\OO_\Lambda(b\delta).
\end{align*}
Comparing the right-hand sides, we conclude that
\begin{equation}\label{eq:xayb}
x^a-y^b\ll_\Lambda b\delta+\sqrt{b\delta}.
\end{equation}
In the remaining case when $y^b\leq 1$ and $b\delta\geq 1$, the inequality \eqref{eq:xayb} holds automatically in the stronger form $|x^a-y^b|<2\leq 2b\delta$.

We proved that \eqref{ineq} implies \eqref{eq:xayb} in all ranges. For our specific set-up \eqref{xy}, the inequality \eqref{eq:xayb} says that
\[aA^2-bB^2\ll_\Lambda X^2(b\delta+\sqrt{b\delta}),\]
and this is equivalent to \eqref{cases-of-eq} in the light of $A^2+B^2=X^2$.
This shows \ref{part-a} under the assumption $0\leq q<\ell$, but it is easily seen to continue to hold also for $q = \ell$ in which case \eqref{ineq} simply reads $A^2 = X^2 + \OO_\Lambda(X^2\delta)$. The argument for $- \ell \leq q < 0$ is identical.

Turning to \ref{part-b}, we conclude from \ref{part-a} that
\[ \left(\frac{2\ell}{\ell+q}\right)^{\ell+q}\left(\frac{2\ell}{\ell-q}\right)^{\ell-q}A^{\ell+q}B^{\ell-q}C^{\ell+q}D^{\ell-q}\leq X^{4\ell}. \]
On the other hand, using Stirling's formula $n!\sim(n/e)^n\sqrt{2\pi n}$, we have for $|q|<\ell$ that
\[ \binom{2\ell}{\ell+q}\asymp\frac{(2\ell)^{2\ell}}{(\ell+q)^{\ell+q}(\ell-q)^{\ell-q}}\sqrt{\frac{2\ell}{(\ell+q)( \ell-q)}}, \]
and so combining the two most recent displays we have the announced bound
 \[ \binom{2\ell}{\ell+q}A^{\ell+q}B^{\ell-q}C^{\ell+q}D^{\ell-q} \ll\frac{X^{4\ell}}{ 1+\sqrt{\ell-|q|}}. \]
We added artificially the $1+$ term in the denominator, so that the inequality also holds for the previously excluded case $|q| = \ell$ in view of $AC, BD \leq X^2$ (which follows directly  from $A^2 + C^2 = B^2 +D^2 = X^2$). The claim that the left-hand side is negligible unless \eqref{cases-of-eq} holds for $(A,B)$ and $(C,D)$ is immediate from \ref{part-a}.
\end{proof}

We now return to the double integral in the upper bound for $\varphi_{\nu,\ell}^{\ell,q}(g)$.
We estimate
the inner integral by writing the integrand as  $A^{\ell + q}  B^{\ell - q} C^{\ell + q} D^{\ell - q}$ in the obvious way and applying Lemma~\ref{young}, where
\[A^2+B^2=C^2+D^2=X^2=1+t^2,\]
and
\[ A^2=\frac{1+t^2}{2}+\frac{t^2-1}{2}\cos 2v_1+t\sin 2v_1\cos(\phi+\Delta), \]
with analogous expressions for $B^2$, $C^2$, and $D^2$. Since
\[\frac{(1 + (t/r)^2)(1+(tr)^2)}{(1 + t^2)^2} = 1 + \left(\frac{r - r^{-1}}{t + t^{-1}}\right)^2,\]
we conclude that the contribution of the inner integral is $\OO_\Lambda(\ell^{-\Lambda})$ unless
\begin{equation}\label{r}
\min(t,t^{-1})\ll_\Lambda\frac{\lambda}{(r-1)\sqrt{\ell}}.
\end{equation}
For $r=1$ we treat the right-hand side as infinity. We may then summarize our findings as follows.

\begin{lemma}\label{lem1} Let $\Lambda\in\NN$. Let $\ell,q\in\ZZ$ be such that $\ell\geq\max(1,|q|)$, and let $\nu\in i\RR$. Assume that $g\in\SL_2(\CC)$ is given by \eqref{g}. Let us abbreviate $\Delta:= u_2 + w_1$ and $\lambda := \sqrt{\log \ell}$. Let $\mcM = \mcM(v_1,v_2,\Delta,r,\Lambda)$ be the set of $(\phi, t) \in [0, 2\pi] \times [0, \infty)$ satisfying \eqref{r} as well as
\begin{equation}\label{eq}
\begin{split}
2t\sin 2v_1\cos(\phi+\Delta)&=(1-t^2)\cos 2v_1+\frac{q}{\ell}(1+t^2)+\OO_{\Lambda}\left((1+t^2)\frac{\lambda^2+ \lambda\sqrt{\ell-|q|}}{\ell} \right),\\
2t\sin 2v_2\cos(\phi-\Delta)&=(t^2-1)\cos 2v_2-\frac{q}{\ell}(1+t^2)+\OO_{\Lambda}\left((1+t^2)\frac{\lambda^2 + \lambda\sqrt{\ell-|q|}}{\ell} \right),\\
\end{split}
\end{equation}
with a sufficiently large (but fixed) implied constant depending on $\Lambda$. Then
\begin{equation}\label{keyupperbound}
\varphi_{\nu,\ell}^{\ell,q}(g)  \ll_\Lambda  \frac{\ell}{1 + \sqrt{\ell - |q|}}
\int_{\mcM} \frac{t}{(1+t^2)^2} \, \dd \phi\, \dd t + \ell^{-\Lambda}.
\end{equation}
\end{lemma}

\subsection{Simplifying assumptions}\label{simplifying-section}
For the proof of Theorems~\ref{thm6}~and~\ref{thm5}, we can and we shall assume that $|\Delta|\leq\pi/4$. Indeed, using the last relation in \eqref{angleequiv} multiple times, we can choose the coordinates in \eqref{g} so that this bound is satisfied. Moreover, we can replace $g$ by
\[g^{-1}=k\left[\frac{\pi}{2}-w_2,v_2-\frac{\pi}{2},u_2+\frac{\pi}{2}\right]
\begin{pmatrix}r & \\ & r^{-1}\end{pmatrix}
k\left[w_1-\frac{\pi}{2},v_1-\frac{\pi}{2},\frac{\pi}{2}-u_1\right]\]
if needed, because the quantities $\Delta$, $\|g\|$, $D(g)$ do not change under this replacement, 
$\bigl|\varphi_{\nu,\ell}^{\ell,q}(g)\bigr|=\bigl|\varphi_{\nu,\ell}^{\ell,q}(g^{-1})\bigr|$ holds by \eqref{eq:averaged-spherical-function-symmetry}, and
\[\dist(g,\mcH)=\dist(g^{-1},\mcH),\qquad\mcH\in\{K,\mcD,\mcS\}\]
holds by \eqref{distinvariance}.

We shall derive (most of) the bounds in Theorems~\ref{thm6}~and~\ref{thm5} from \eqref{keyupperbound}. In Lemma~\ref{lem1}, the pair $(\Delta,r)$ does not change under the above discussed replacement $g\mapsto g^{-1}$, while
the corresponding integration domains $\mcM$ are related by
\[(\phi,t)\in\mcM\left(v_2-\frac{\pi}{2},v_1-\frac{\pi}{2},\Delta,r,\Lambda\right)\quad\Longleftrightarrow\quad
(\phi,t^{-1})\in\mcM(v_1,v_2,\Delta,r,\Lambda).\]
Moreover, the integrand in \eqref{keyupperbound} is invariant under $t\mapsto t^{-1}$, hence we can assume that the contribution of $t\leq 1$ is not smaller than the contribution of $t>1$. So from now on we restrict $\mcM$ in \eqref{keyupperbound} to the corresponding subset of $[0, 2\pi] \times [0, 1]$. On this subset we have, by \eqref{r},
\begin{equation}\label{r1}
t\in[0,1]\qquad\text{and}\qquad t \ll_\Lambda \frac{\lambda}{(r-1)\sqrt{\ell}}.
\end{equation}

\subsection{Proof of Theorem~\ref{thm6}}
The bound \eqref{thm6bound} is trivial for $\ell\ll_\Lambda 1$, hence we shall assume that $\ell$ is sufficiently large in terms of $\Lambda$. With the notation
\[\alpha := \dist(g, K) \asymp r-1\qquad\text{and}\qquad\beta := \dist(g, \mcD),\]
it follows from \eqref{keyupperbound} and the previous subsection that it suffices to show
\begin{equation}\label{beginThm6}
\frac{\ell}{1 + \sqrt{\ell - |q|}}\int_{\mcM}  t \, \dd \phi\, \dd t \ll_{\eps,\Lambda}
\ell^{\eps}\min\left(1,\frac{\| g \|}{\sqrt{\ell}\alpha^2\beta}\right),
\end{equation}
where $\mcM$ is now restricted by \eqref{r1}. In fact our arguments below will show that $\ell^\eps$ can be replaced by $(\log\ell)^3$.

We start with the first bound of \eqref{beginThm6}. With the notation
\[\sigma:=\lambda^2+\lambda\sqrt{\ell-|q|},\qquad
\mu:=\frac{q}{\ell} - \cos 2v_1,\qquad \rho:=\sin 2 v_1,\]
the first equation in \eqref{eq} becomes
\begin{equation}\label{tquadratic}
\mu t^2 - 2t\rho\cos(\phi+\Delta) + \frac{2q}{\ell}-\mu+\OO_\Lambda\left(\frac{\sigma}{\ell}\right)=0.
\end{equation}
Without loss of generality, $\mu\neq 0$, and then we can view \eqref{tquadratic} as a quadratic equation for $t$. Multiplying by $\mu$ and completing the square, we obtain the alternative form
\begin{equation}\label{compl}
\bigl(\mu t - \rho\cos(\phi + \Delta)\bigr)^2+\bigl(\rho\sin(\phi + \Delta)\bigr)^2
=1-\frac{q^2}{\ell^2}+\OO_\Lambda\left(\frac{|\mu|\sigma}{\ell}\right).
\end{equation}
In particular, the discriminant of \eqref{tquadratic} equals $4D(\phi)+\OO_\Lambda(|\mu|\sigma/\ell)$, where
\begin{equation}\label{disc-q}
D(\phi) := 1-\frac{q^2}{\ell^2}-\bigl(\rho\sin(\phi + \Delta)\bigr)^2
\end{equation}
We assume first that $|q|/\ell\leq 5/6$, and decompose $\mcM$ into two parts $\mcM^\pm$ according as
$|\rho\sin(\phi+\Delta)|$ exceeds $1/2$ or not. On $\mcM^+$, the equation \eqref{tquadratic} localizes $\phi$ within $\ll_\Lambda\sigma/(\ell t)$ for each given $t\in[0,1]$. On $\mcM^-$, we have $D(\phi)\geq 1/18$, hence the equation \eqref{compl} localizes $t$ within $\ll_\Lambda\sigma/\ell$ for each given $\phi\in[0,2\pi]$. This shows that
\[\int_\mcM t\,\dd\phi\,\dd t\ll_\Lambda\int_0^1 t\,\frac{\sigma}{\ell t}\,\dd t+\int_0^{2\pi}\frac{\sigma}{\ell}\,\dd\phi
\ll\frac{\sigma}{\ell},\]
hence the first bound of \eqref{beginThm6} follows in stronger form. From now on we assume that $|q|/\ell>5/6$. We decompose $\mcM$ into two parts $\mcM^\pm$ according as $D(\phi)$ is positive or not, and we make two initial observations. First, $\mcM^+$ is clearly empty when $|q|=\ell$. Second, $|\mu|>1/6$ holds for large $\ell$, because \eqref{tquadratic} coupled with $t\in[0,1]$ yields
\[\frac{2|q|}{\ell}-|\mu|-\OO_\Lambda\left(\frac{\sigma}{\ell}\right)\leq 2t|\rho|
\leq 2\sqrt{1 - \left( \frac{q}{\ell} - \mu\right)^2}.\]
In order to estimate the contribution of $\mcM^+$ in \eqref{beginThm6}, we decompose $\mcM^+$
into pieces
\[\mcM^{+}({\mtD},\eta):=\left\{(\phi,t)\in\mcM^+:
\text{$D(\phi) \asymp {\mtD}$ and $|\cos(\phi + \Delta)| \asymp \eta$}
\right\}.\]
If $\eta \leq \ell^{-10}$, we can estimate trivially, so there are only $\OO(\log \ell)$ relevant values for $\eta$. If $\rho \geq \ell^{-10}$, then by the same argument there are only $\OO(\log \ell)$ relevant values for ${\mtD}$. If $\rho < \ell^{-10}$, then ${\mtD} > \ell^{-1}$ by $|q|<\ell$, hence again
there are only $\OO(\log \ell)$ relevant values for ${\mtD}$. So in all cases it suffices to restrict to $\OO((\log \ell)^2)$ pairs $({\mtD},\eta)$. Our current assumptions localize $\sin(\phi+\Delta)$ within
$\ll\sqrt{{\mtD}}/|\rho|$, and hence $\phi$ within $\ll\min(1,\sqrt{\mtD}/|\rho\eta|)$, independently of $t$. On the other hand, given $\phi$, the equation \eqref{compl} localizes $t$ within $\ll_\Lambda\min((\sigma/\ell){\mtD}^{-1/2},\sqrt{\sigma/\ell})$.
Such $t$ are of size $\ll_\Lambda |\rho\eta| + \sqrt{\mtD} + \sqrt{\sigma/\ell}$, so that
\[\int_{\mcM^+({\mtD},\eta)} t\,\dd\phi\,\dd t \ll_\Lambda \left( |\rho\eta|+ \sqrt{\mtD} + \sqrt{\frac{\sigma}{\ell}}  \right)\min\left( \frac{\sigma}{\ell \sqrt{\mtD} }, \sqrt{\frac{\sigma}{\ell}}\right)  \min\left(1, \frac{\sqrt{\mtD} }{|\rho \eta|}\right)
\ll\frac{\sigma}{\ell}.\]
This contribution is admissible for the first bound of \eqref{beginThm6}.
It remains to estimate the contribution of $\mcM^-$ in \eqref{beginThm6}. On this set we have
\[0\leq -D(\phi)\ll_\Lambda\frac{\sigma}{\ell}\]
by \eqref{compl}. The argument is similar as for $\mcM^+$, in fact simpler as we only need $\OO(\log \ell)$ pieces $\mcM^{-}(\eta)$ defined by $|\cos(\phi+\Delta)|\asymp\eta$. Initially we localize $\phi$ within $\ll\min(1,|\rho\eta|^{-1}\sqrt{\sigma/\ell})$, independently of $t$. The equation \eqref{compl} localizes $t$ within $\ll_\Lambda\sqrt{\sigma/\ell}$, and such $t$ are of size $\ll_\Lambda |\rho\eta| + \sqrt{\sigma/\ell}$. We obtain altogether
\[\int_{\mcM^-(\eta)} t\,\dd\phi\,\dd t \ll_\Lambda \left(|\rho \eta| + \sqrt{\frac{\sigma}{\ell}}\right)
\sqrt{\frac{\sigma}{\ell}}\min\left(1, \frac{1}{|\rho\eta|}\sqrt{\frac{\sigma}{\ell}}\right) \ll \frac{\sigma}{\ell},\]
which is again admissible for the first bound of \eqref{beginThm6}.

We now turn to the second bound of \eqref{beginThm6}. We shall assume (as we can) that $\mcM\neq\emptyset$ and $\sqrt{\ell}\alpha^2\beta>\| g \|$. We pick an arbitrary point $(\phi,t)\in\mcM$. Combining \eqref{eq} and \eqref{r1}, we get
\[\cos 2v_j = - \text{sgn}(q) + \OO(t)\sin 2v_j + \OO_\Lambda\left(t^2 +\frac{\lambda^2+ \ell - |q|}{\ell}\right),\]
where for $q=0$ we can replace $\sgn(q)$ by $1$. After squaring and solving for $\sin 2v_j$, then feeding back the result into the previous display, we get
\[\sin 2v_j=\OO_\Lambda\left(t +\frac{\lambda+ \sqrt{\ell - |q|}}{\sqrt{\ell}}\right),\qquad
\cos 2v_j = - \text{sgn}(q) + \OO_\Lambda\left(t^2 +\frac{\lambda^2+ \ell - |q|}{\ell}\right).\]
Recalling also \eqref{g}, and using \eqref{r1} again, we infer that
\[\beta \ll_\Lambda\|g\|\left(\frac{\lambda}{\alpha\sqrt{\ell}} +\frac{\lambda+ \sqrt{\ell - |q|}}{\sqrt{\ell}}\right).\]
Hence we always have
\[1 \ll_\Lambda \frac{\|g\|\lambda}{\alpha\beta\sqrt{\ell}}
\qquad \text{or} \qquad 1 \ll_\Lambda \|g\|\lambda\frac{1+\sqrt{\ell - |q|}}{\beta\sqrt{\ell}}.\]
In either case, for any $c > 0$, the previous display combined with \eqref{r1} yields that
\begin{align*}
\frac{\ell}{1 + \sqrt{\ell - |q|}}\int_{\mcM}  t \, \dd \phi\, \dd t
&\ll_{\Lambda,c}\frac{\lambda^2}{(1+\sqrt{\ell - |q|})\alpha^2}
\left(\left(\frac{\|g\|\lambda}{\alpha\beta\sqrt{\ell}}\right)^c+ \|g \|\lambda\frac{1+\sqrt{\ell - |q|}}{\beta\sqrt{\ell}}\right)\\[4pt]
& \ll_{\eps,\Lambda,c} \ell^\eps\left(\frac{\| g \|^c}{\ell^{c/2} \alpha^{2 + c} \beta^c}+\frac{\| g \|}{\sqrt{\ell} \alpha^2 \beta}\right).
\end{align*}
Choosing $c = 2$, and recalling our initial assumption $\sqrt{\ell}\alpha^2\beta>\| g \|$,
we obtain the second bound of \eqref{beginThm6} in stronger form.

The proof of Theorem~\ref{thm6} is complete.

\subsection{Proof of Theorem~\ref{thm5}\ref{thm5-a}}\label{thm5a-proof-sec}
The averaged spherical trace function $\varphi_{\nu,\ell}^{\ell, q}(g)$ exhibits starkly different behavior depending on the value of $-\ell\leq q\leq\ell$. Some of these features are already visible along $K=\SU_2(\CC)$. From \eqref{varphi-def-proof} and \eqref{post-expansion} we can see that, in the notation of \eqref{decomp-K} and \eqref{matrix-coeff},
\[\varphi_{\nu,\ell}^{\ell,q}(k[u,v,w])=\Phi_{q,q}^{\ell}(k[u,v,w])=
e^{2\pi i q(u+w)}(\cos v)^{2q}P_{\ell-q}^{(0,2q)}(\cos 2v).\]
The absolute value of the right-hand side exhibits a primary peak at $v\in\pi\ZZ$ of size 1. For $q=\pm\ell$, this is followed by a sharp drop to $\OO_N(\ell^{-N})$ after a range of length about $\ell^{-1/2}$. For a generic $q$, the drop becomes soft through a highly oscillatory range of magnitude $\ell^{-1/2}$ (faster and more oscillatory for smaller $q$) and a secondary, Airy-type peak of size about $\ell^{-1/3}$ before the delayed sharp drop. For $q=0$, the secondary peak grows to a full peak of size 1 at $v\in\frac12\pi+\pi\ZZ$ (corresponding to skew-diagonal matrices in $K$) and the sharp drop disappears. These varying features, which are illustrated in Figure~\ref{Jacobi-figure}, become vastly more complicated off $K$, where the hard work in Theorems~\ref{thm6} and \ref{thm5} lies. Nevertheless, their traces are visible in the hard localization to $\mcD$ (but none to $K$!) for $q=\pm\ell$ and the hard localization to $\mcN$ with soft localization to $\mcS\subset K\subset\mcN$ for $q=0$.

\begin{figure}
\centering
\includegraphics[width=0.24\textwidth]{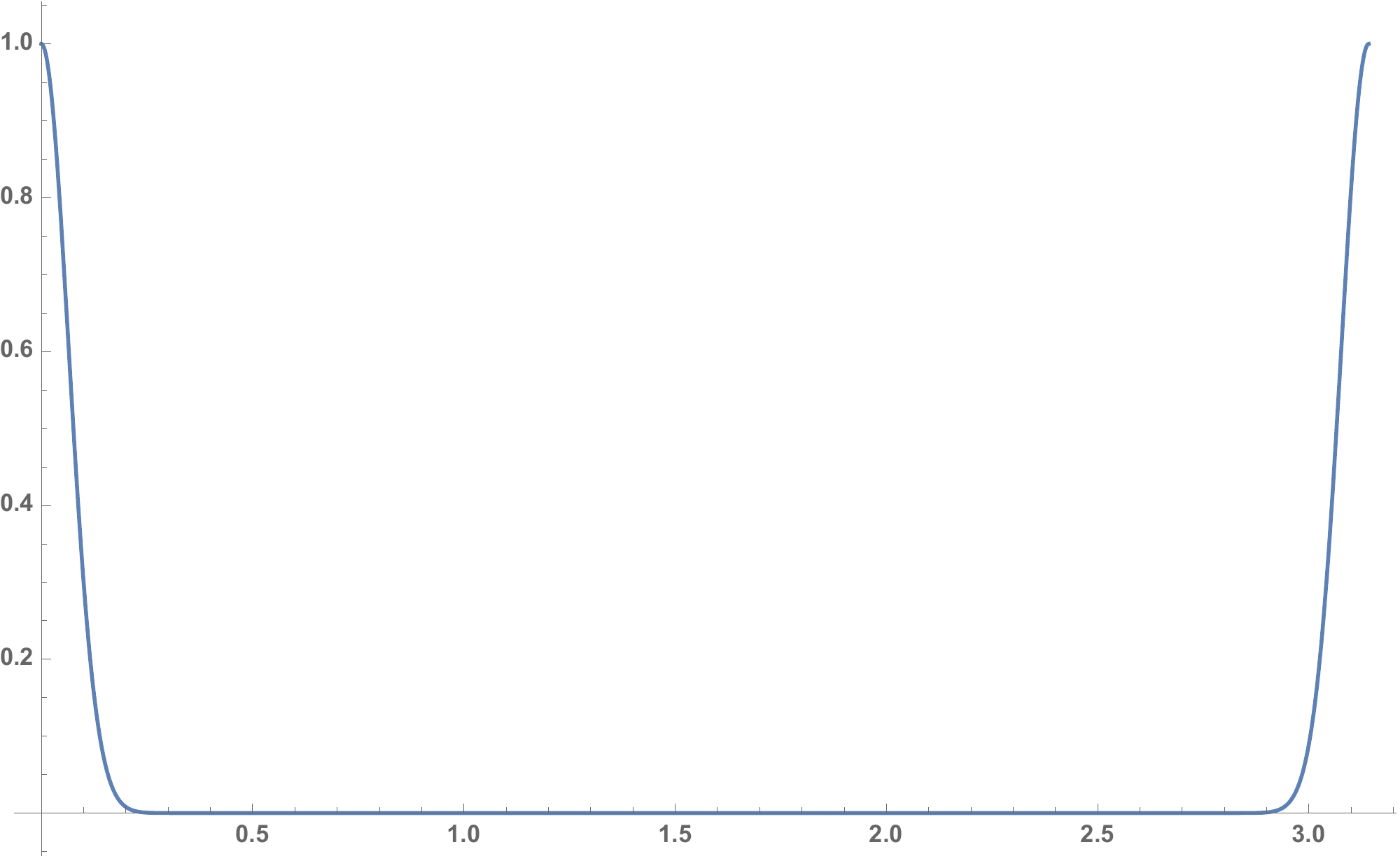}
\includegraphics[width=0.24\textwidth]{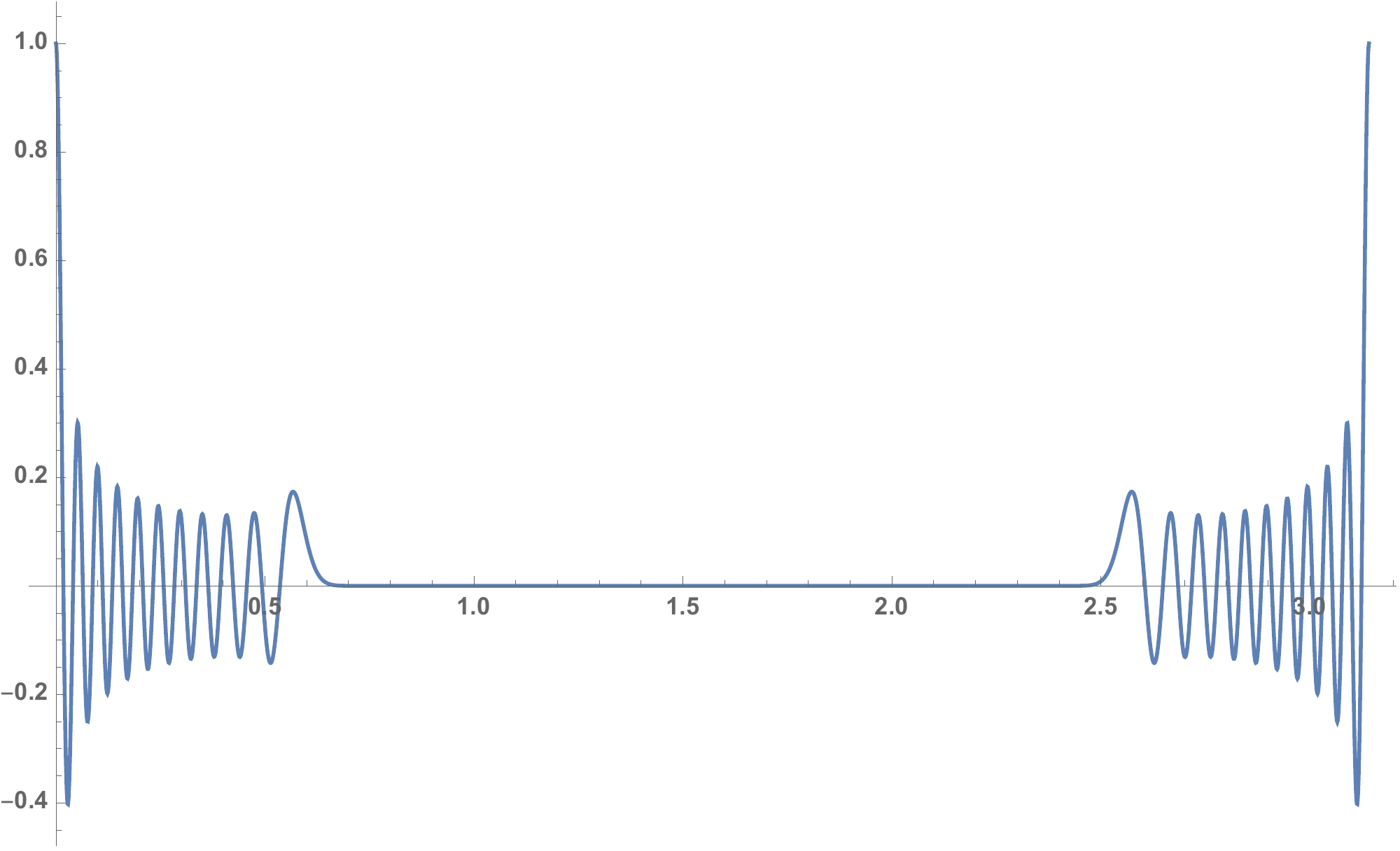}
\includegraphics[width=0.24\textwidth]{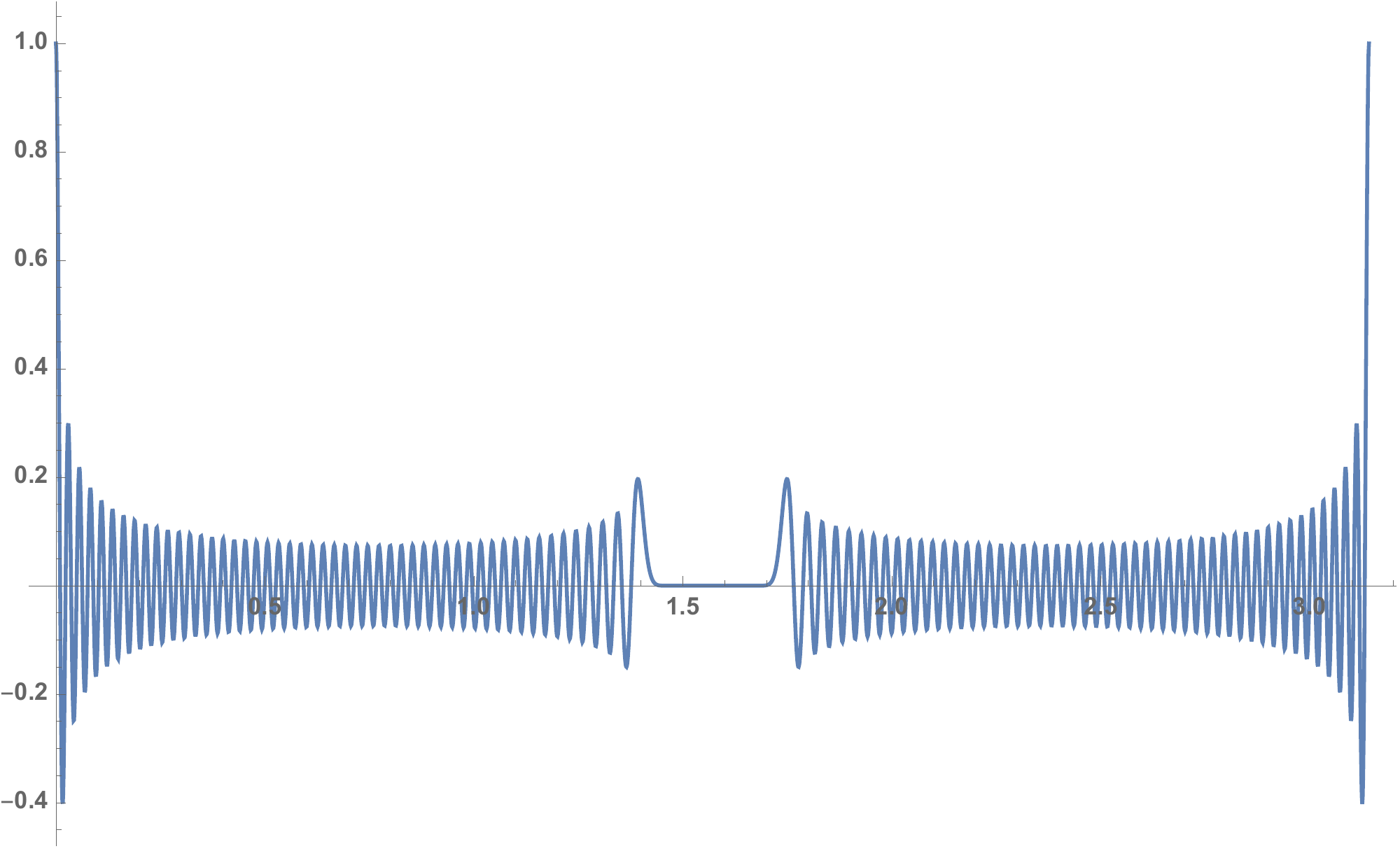}
\includegraphics[width=0.24\textwidth]{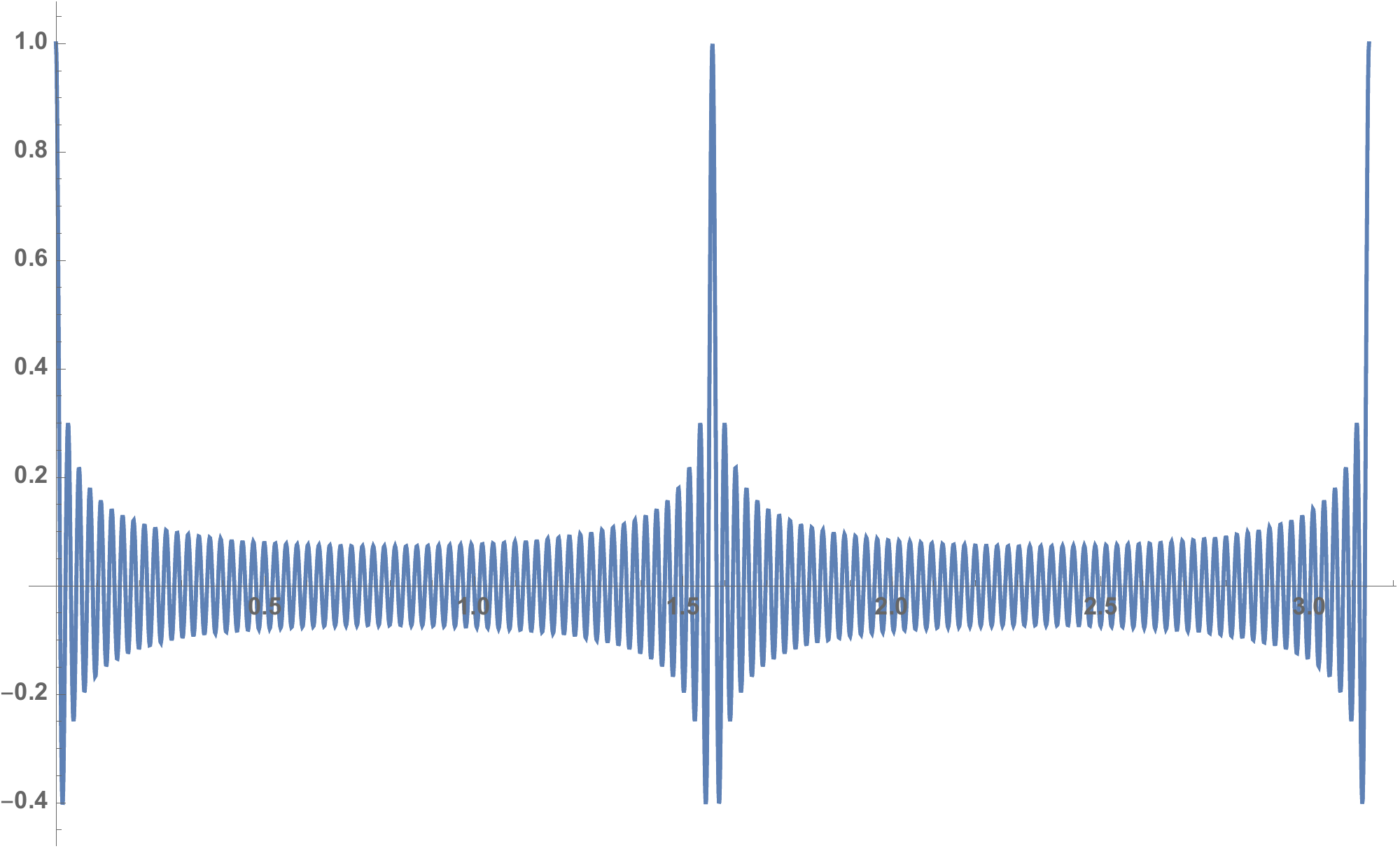}
\caption{Plots of $(\cos v)^{2q}P_{\ell-q}^{(0,2q)}(\cos 2v)$ for $0\leq v\leq\pi$, $\ell=120$, $q=120$, $q=100$, $q=20$, and $q=0$.}
\label{Jacobi-figure}
\end{figure}

In this subsection, we consider in more detail the case $q = 0$. Then \eqref{keyupperbound} simplifies to
\begin{equation}\label{keyupperbound:q=0}
\varphi_{\nu,\ell}^{\ell,0}(g) \ll_\Lambda\sqrt{\ell}
\int_{\mcM} \frac{t}{(1+t^2)^2} \, \dd \phi\, \dd t + \ell^{-\Lambda},
\end{equation}
where by \eqref{eq} and the last paragraph of \S \ref{sec:preliminary-computations}, the set $\mcM$ can be described by the constraints given in \eqref{r1} and
\begin{equation}\label{eq0}
\begin{split}
2t\sin 2v_1\cos(\phi+\Delta)&=(1-t^2)\cos 2v_1 +\OO_\Lambda(\lambda/\sqrt{\ell}),\\
2t\sin 2v_2\cos(\phi-\Delta)&=(t^2-1)\cos 2v_2 +\OO_\Lambda(\lambda/\sqrt{\ell}).\\
\end{split}
\end{equation}
We shall use the notations
\begin{align*}
P(\phi) &:= \max(|\sin 2v_1 \cos(\phi + \Delta)|, |\sin 2v_2 \cos(\phi-\Delta)|),\\
R & := \max(|\cos 2v_1|, |\cos2v_2|),\\
N & := \max(|\sin(2v_1+2v_2) \cos\Delta|, |\sin(2v_1-2v_2)\sin\Delta|).
\end{align*}
Recall also the earlier notations \eqref{adbc} and \eqref{Dgdef}. As
\[|a|^2 - |d|^2 = \frac{r^2-r^{-2}}{2}(\cos 2v_1 + \cos 2v_2), \qquad
|b|^2 - |c|^2 = \frac{r^2-r^{-2}}{2}(\cos 2v_1 - \cos 2v_2),\]
we can identify $\mcN$ as the set of matrices with $r=1$ or $\cos 2v_1 = \cos 2v_2 = 0$. More precisely,
by \eqref{eq0} and \eqref{r1} we have
\[D(g)\ll r(r-1)R\ll_\Lambda r(r-1)\left(t +\frac{\lambda}{\sqrt{\ell}}\right)\ll_\Lambda\frac{r^2\lambda}{\sqrt{\ell}},\]
so that unless $D(g)\ll_\Lambda\|g\|^2\lambda \ell^{-1/2}$, we have $\mcM=\emptyset$, yielding $\varphi^{\ell, 0}_{\nu, \ell}(g)\ll_\Lambda\ell^{-\Lambda}$. Hence we are left with proving \eqref{thm5boundq=0}.

In \eqref{keyupperbound:q=0}, the contribution of the $t$-integral over the interval $[0,\ell^{-\Lambda/2-1/4}]$ is negligible, and we split the rest of $\mcM$ in dyadic ranges $\mcM(\delta)$ according to $\ell^{-\Lambda/2-1/4}<t\asymp \delta\leq 1$. The number of such ranges is $\OO_{\Lambda}(\log \ell)$. Assume $(\phi,t)\in \mcM(\delta)$. The discriminants of the two quadratic equations \eqref{eq0} are $4D_j(\phi)+\OO_{\Lambda}(\lambda/\sqrt{\ell})$, where
\[D_1(\phi):=1 - \sin^2 (2v_1) \sin^2(\phi+ \Delta),\qquad D_2(\phi):=1 - \sin^2 (2v_2) \sin^2(\phi - \Delta).\]
A simple calculation gives that
\begin{equation}\label{disc-lower-bound}
D_1(\phi)+D_2(\phi)\geq P(\phi)^2+R^2.
\end{equation}
If $|\sin 2v_1\sin(\phi+\Delta)| > 1/2$, then for any fixed $t$, \eqref{eq0} localizes $\phi$ to a set of measure $\OO_{\Lambda}(\lambda/\sqrt{\ell})$. Otherwise, for any fixed $\phi$, \eqref{eq0} localizes $t$ to a set of measure $\OO_{\Lambda}(\lambda/\sqrt{\ell})$. We conclude that
\begin{equation}\label{I1}
\meas(\mcM(\delta))\ll_\Lambda\lambda/\sqrt{\ell}.
\end{equation}

Now we prove the alternative bound
\begin{equation}\label{I2b}
\meas(\mcM(\delta))\ll_\Lambda\frac{\lambda^4}{N\ell}.
\end{equation}
We shall assume that $N\ell>1$, for otherwise \eqref{I2b} follows from \eqref{I1}. Under this assumption, we have $\max(|\sin 2v_1|,|\sin 2v_2|)\gg\ell^{-1}$, which implies that
\[
\meas(\{(\phi,t)\in \mcM(\delta):P(\phi)\leq \ell^{-3}\}) \ll \ell^{-1}.
\]
Indeed, if $\phi$ changes by at least $\ell^{-1}$ and at most $\pi/4$, then $\cos(\phi\pm\Delta)$ both change by $\Omega(\ell^{-2})$, hence $P(\phi)$ changes by $\Omega(\ell^{-3})$. This implies that $P(\phi)\leq \ell^{-3}$ localizes $\phi$ to a set of measure $\OO(\ell^{-1})$.
Therefore, the contribution of $\{(\phi,t)\in \mcM(\delta):P(\phi)\leq \ell^{-3}\}$ to the left-hand side of \eqref{I2b} is
$\OO_{\Lambda}(\delta/\sqrt{\ell})$, which is admissible by $N\leq 1$. We decompose the rest of $\mcM(\delta)$ into dyadic ranges $\mcM(\delta,\mtP)$ according to $\ell^{-3}\leq P(\phi)\asymp \mtP\leq 1$. The number of such ranges is $\OO(\log \ell)$, hence in order to verify \eqref{I2b}, it suffices to prove
\[\meas(\mcM(\delta,\mtP))\ll_\Lambda\frac{\lambda^2}{N\ell}.\]
The proof of this estimate immediately reduces to the following two localizations:
\begin{equation}\label{phi-loc}
\meas(\{\phi\in[0,2\pi]:(\phi,t)\in\mcM(\delta,\mtP)\text{ for some $t\asymp \delta$}\}) \ll_{\Lambda} \frac{(\mtP+R)\lambda}{N\sqrt{\ell}},
\end{equation}
and for any $\phi\in[0,2\pi]$,
\begin{equation}\label{t-loc}
\meas(\{t\in[0,1]:(\phi,t)\in\mcM(\delta,\mtP)\}) \ll_{\Lambda} \frac{\lambda}{(\mtP+R)\sqrt{\ell}}.
\end{equation}
Now we prove these localizations.

Starting out from \eqref{eq0}, we execute two eliminations: one to eliminate the main terms of the right-hand sides, and the other one to eliminate the left-hand sides. Introducing
\[
F(\phi):= \cos \phi \cos \Delta \sin(2v_1+2v_2) - \sin \phi \sin \Delta \sin(2v_1-2v_2),
\]
these give
\[
tF(\phi) \ll_{\Lambda} R \lambda/\sqrt{\ell}\qquad\text{and}\qquad (1-t^2) F(\phi) \ll_{\Lambda} \mtP \lambda /\sqrt{\ell}.
\]
In particular, we obtain both for $t>1/2$ and $t\leq 1/2$ that
\begin{equation}\label{bound-on-Fphi}
F(\phi) \ll_{\Lambda} (\mtP+R) \lambda /\sqrt{\ell}.
\end{equation}
Letting
\[
N':= \sqrt{\sin^2(2v_1+2v_2)\cos^2(\Delta) + \sin^2(2v_1-2v_2)\sin^2(\Delta)} \asymp N,
\]
and choosing $\psi\in [0,2\pi)$ such that
\[
\cos\psi = \frac{\sin(2v_1+2v_2)\cos\Delta}{N'},\qquad
\sin\psi = - \frac{\sin(2v_1-2v_2)\sin \Delta}{N'},
\]
\eqref{bound-on-Fphi} gives rise to
\[
\cos(\phi-\psi) \ll_{\Lambda} \frac{(\mtP+R) \lambda }{N\sqrt{\ell}}.
\]
This localizes $\phi$ to a set of measure $\OO_{\Lambda}((\mtP+R)\lambda/N\sqrt{\ell})$. Indeed, if the right-hand side is very small in terms of the implied constant, then $\phi-\psi$ is bounded away from $\pi\ZZ$, hence the derivative $\cos'(\phi-\psi)$ is bounded away from zero, while otherwise the claimed localization is trivial. This gives \eqref{phi-loc}. Fixing $\phi\in[0,2\pi]$, and solving under \eqref{eq0} the quadratic equation in $t$ of the larger discriminant, we see by \eqref{disc-lower-bound} that $t$ is localized to a set of measure $\OO_{\Lambda}(\lambda/(\mtP+R)\sqrt{\ell})$. This gives \eqref{t-loc}. Altogether, the proof of \eqref{I2b} is complete.

Combining \eqref{I1} and \eqref{I2b}, we obtain
\[\meas(\mcM(\delta))\ll_{\eps,\Lambda}\ell^{\eps-1}\mu,\qquad\mu:=\min(\sqrt{\ell},N^{-1}).\]
We claim that 
\begin{equation}\label{toprovedelta-b}
\dist(g,\mcS)\ll_{\Lambda}\lambda\delta^{-1}\mu^{-1},\qquad \text{if $\mcM(\delta)\neq\emptyset$.}
\end{equation}
This implies the inequality
\[\sqrt{\ell}\int_{\mcM(\delta)} t\, \dd \phi\, \dd t\ll_{\eps,\Lambda}\ell^{\eps-1/2}\delta\mu
\ll_{\eps,\Lambda}\frac{\ell^\eps}{\sqrt{\ell}\dist(g,\mcS)},\]
which, summed over the $\OO(\log\ell)$ dyadic ranges for $\delta$, suffices for the proof of \eqref{thm5boundq=0}. Note that the bound $\varphi_{\nu,\ell}^{\ell,q}(g)\ll_\eps\ell^{\eps}$ is already covered by Theorem~\ref{thm6}.

To complete the proof of Theorem~\ref{thm5}\ref{thm5-a}, it remains to show \eqref{toprovedelta-b}. For this final argument, we can and we shall assume that $-\pi/8\leq v_1,v_2\leq 3\pi/8$, because replacing $(u_1,v_1)$ by $(-u_1,v_1+\pi/2)$, or $(v_2,w_2)$ by $(v_2+\pi/2,-w_2)$, has the effect of multiplying $g$ by $\left(\begin{smallmatrix}&i\\i&\end{smallmatrix}\right)$ from either side without altering $\Delta$ or the statement \eqref{toprovedelta-b}. We fix a pair $(\phi,t)\in\mcM(\delta)$.

Now, $N\leq \mu^{-1}$ implies that
\begin{equation}\label{eq:toprovedelta-c}
v_1+v_2\in\frac{\pi}2\ZZ+\OO\left(\frac1{\mu}\right)\qquad\text{and}\qquad
v_1-v_2\in\frac{\pi}2\ZZ+\OO\left(\frac1{\mu|\Delta|}\right).
\end{equation}
Let us introduce the short-hand notation
\[m[v]:=\begin{pmatrix}\cos v&i\sin v\\i\sin v&\cos v\end{pmatrix},\qquad v\in\RR.\]
Keeping \eqref{decomp-K} and \eqref{g} in mind, we observe initially that
\begin{equation}\label{initialdistance}
m[v_1]\diag\left(re^{i\Delta},r^{-1}e^{-i\Delta}\right)m[v_2]=m[v_1+v_2]+\OO\bigl(r-1+|\Delta|\bigr).
\end{equation}
On the right-hand side, we have $\dist(m[v_1+v_2],\mcS)\ll\mu^{-1}$ by \eqref{eq:toprovedelta-c}, and also
\begin{equation}\label{rbound}
r-1\ll_\Lambda\frac{\lambda}{t\sqrt{\ell}}\ll\frac{\lambda}{\delta\mu}
\end{equation}
by \eqref{r1} and $\mu\leq\sqrt{\ell}$. Hence \eqref{toprovedelta-b} follows from \eqref{initialdistance} as long as $\Delta\ll_\Lambda\lambda\delta^{-1}\mu^{-1}$. In other words, we can and we shall assume that $|\Delta|\gg_\Lambda\lambda\delta^{-1}\mu^{-1}$ holds with a sufficiently large implied constant depending on $\Lambda$. In particular, we shall assume that the error terms in \eqref{eq:toprovedelta-c}, and similar error terms for angles in the rest of this subsection, are less than $\pi/8$ in size. Under this assumption, \eqref{eq:toprovedelta-c} breaks into two cases.

\emph{Case 1:} $v_1,v_2\ll\mu^{-1}|\Delta|^{-1}$ and $v_1+v_2\ll\mu^{-1}$. In this case, we refine \eqref{initialdistance} to
\begin{align*}
&m[v_1]\diag\left(re^{i\Delta},r^{-1}e^{-\Delta}\right)m[v_2]\\
&=m[v_1+v_2]+m[v_1]\diag\left(re^{i\Delta}-1,r^{-1}e^{-i\Delta}-1\right)m[v_2]\\
&=m[v_1+v_2]+\diag\left(re^{i\Delta}-1,r^{-1}e^{-i\Delta}-1\right)+\OO\bigl(r-1+\mu^{-1}\bigr)\\
&=\diag\left(e^{i\Delta},e^{-i\Delta}\right)+\OO\bigl(r-1+\mu^{-1}\bigr).
\end{align*}
The main term $\diag\left(e^{i\Delta},e^{-i\Delta}\right)$ lies in $\mcS$, hence \eqref{toprovedelta-b} follows by \eqref{rbound}.

\emph{Case 2:} $v_1,v_2=\pi/4+\OO(\mu^{-1}|\Delta|^{-1})$ and $v_1+v_2=\pi/2+\OO(\mu^{-1})$. As we shall see, this case does not occur. The assumptions imply that $\sin 2v_1$ and $\sin 2v_2$ exceed $1/2$. We multiply the second equation in \eqref{eq0} by $\sin 2v_1$, and the first equation in \eqref{eq0} by $\sin 2v_2$. Adding and subtracting the resulting two equations, we obtain
\begin{align*}
4t\sin 2v_1\sin 2v_2\cos\phi\cos\Delta&=(t^2-1)\sin(2v_1-2v_2)+\OO_\Lambda(\lambda/\sqrt{\ell}),\\
4t\sin 2v_1\sin 2v_2\sin\phi\sin\Delta&=(t^2-1)\sin(2v_1+2v_2)+\OO_\Lambda(\lambda/\sqrt{\ell}).
\end{align*}
We infer that
\[\delta\ll|t\cos\phi|+|t\sin\phi|
\ll_\Lambda|\sin(2v_1-2v_2)|+\frac{|\sin(2v_1+2v_2)|}{|\Delta|}+\frac{\lambda}{\sqrt{\ell}|\Delta|}
\ll_\Lambda\frac{\lambda}{\mu|\Delta|}.\]
This contradicts our earlier assumption that $\Delta\gg_\Lambda\lambda\delta^{-1}\mu^{-1}$ holds with a sufficiently large implied constant depending on $\Lambda$.

The proof of Theorem~\ref{thm5}\ref{thm5-a} is complete.

\subsection{Proof of Theorem~\ref{thm5}\ref{thm5-b}}
We finally consider the case $q = \pm \ell$. By the symmetries \eqref{distinvariance} and \eqref{eq:averaged-spherical-function-symmetry}, we can restrict to $q = \ell$. We have already shown the bound $\varphi^{\ell,\ell}_{\nu,\ell}(g)\ll_\eps\ell^{\eps}$ in greater generality in Theorem~\ref{thm6}. As a first step, we complement this with a stronger bound for $r \geq 2$. To this end, we return to \eqref{post-expansion}. As $q=\ell$, the binomial coefficient and the $J$-factor disappear. When $\ov{I}^{2\ell}$ is expanded, we see a Laurent polynomial of $e^{2iu}$. When we integrate in $u$ from $0$ to $\pi$, all the terms but the constant one vanish.
We calculate the constant term using the binomial theorem and the original product definition of $I$. This way we see that
\begin{align*}
\varphi_{\nu,\ell}^{ \ell,\ell}(g) = \ &  d_{\ell}  e^{2i\ell(u_1-u_2-w_1+w_2)} r^{2\ell}
\sum_{m=0}^{2\ell}\binom{2\ell}{m}^2 (r^{-2}e^{2iu_2+2iw_1}\cos v_1\cos v_2)^m\\
&(-\sin v_1\sin v_2)^{2\ell-m}
\int_0^{\pi/2} (\sin^2 v)^m (\cos^2 v)^{2\ell-m} \frac{\sin 2v}{h(r,v)^{\ell+1-\nu}}\,\dd v.
\end{align*}
Using the variable $x:=\sin^2 v$, we rewrite this as
\begin{align*}
\varphi_{\nu,\ell}^{ \ell,\ell}(g) = \ &  d_{\ell} e^{2i\ell(u_1-u_2-w_1+w_2)} r^{2\nu-2}
\sum_{m=0}^{2\ell}\binom{2\ell}{m}^2
(r^{-2}e^{2iu_2+2iw_1}\cos v_1\cos v_2)^m\\
&(-\sin v_1\sin v_2)^{2\ell-m}\int_0^1 \frac{x^m (1-x)^{2\ell-m}}{(1-x+r^{-4}x)^{\ell+1-\nu}}\,\dd x.
\end{align*}
With the short-hand notation
\[U:=r^{-1}e^{iu_2+iw_1}\sqrt{x\cos v_1\cos v_2}\qquad\text{and}\qquad V:=i\sqrt{(1-x)\sin v_1\sin v_2},\]
we obtain finally
\begin{align}\label{eq:phi-ell-ell-ell-bound}
\begin{split}
\bigl|\varphi_{\nu,\ell}^{\ell,\ell}(g)\bigr| \leq &\ \frac{2\ell+1}{r^2}\left|
\int_0^1\sum_{m=0}^{2\ell}\binom{2\ell}{m}^2\frac{U^{2m}V^{4\ell-2m}}{(1-x+r^{-4}x)^{\ell+1-\nu}}\,\dd x \right|\\[4pt]
=&\ \frac{2\ell+1}{r^2} \left|\int_0^1\frac{1}{2\pi}\int_0^{2\pi}
\frac{(Ue^{i\phi}+V)^{2\ell}(Ue^{-i\phi}+V)^{2\ell}}{(1-x+r^{-4}x)^{\ell+1-\nu}}\,\dd\phi\,\dd x\right|.
\end{split}
\end{align}
Using that $(Ue^{i\phi}+V)(Ue^{-i\phi}+V)=U^2+V^2+2UV\cos\phi$ is on the line segment connecting $(U+V)^2$ and $(U-V)^2$, we observe that
\[\frac{|(Ue^{i\phi}+V)(Ue^{-i\phi}+V)|^2}{1-x+r^{-4}x}\leq \max_\pm \frac{|U\pm V|^4}{1-x+r^{-4}x},\]
which by the Cauchy--Schwarz inequality can be further upper bounded by
\[\leq\frac{(1-x+r^{-2}x)^2}{1-x+r^{-4}x}=1-\frac{x}{1+\frac{2}{r^2-1}+\frac{1}{\left(r^2-1\right)^2 (1-x)}}.\]
Hence the contribution to the rightmost expression in \eqref{eq:phi-ell-ell-ell-bound} of $x\in[0,1]$ satisfying
\[x>\delta\left(1+\frac{2}{r^2-1}+\frac{1}{\left(r^2-1\right)^2 (1-x)}\right),\qquad\delta:=\frac{\log\ell}{\ell},\]
is admissible for \eqref{thm5boundq=ell}. By $r\geq 2$, the remaining values $x\in[0,1]$ satisfy
\[x<3\delta\qquad\text{or}\qquad x(1-x)<\frac{3\delta}{(r^2-1)^2},\]
hence also $x<3\delta$ or $1-x<8\delta/r^4$. So the remaining contribution is
\[\leq\frac{2\ell+1}{r^2}\int_{[0,3\delta)\cup(1-8\delta/r^4,1]}\frac{\dd x}{1-x+r^{-4}x}\ll\frac{\log\ell}{r^2},\]
which is again admissible for \eqref{thm5boundq=ell}.

By \eqref{keyupperbound}, it remains to show that
\begin{equation}\label{toprovedelta-a}
\dist(g, \mcD) \ll_\Lambda\| g \|\lambda/\sqrt{\ell},\qquad \text{if $\mcM\neq\emptyset$.}
\end{equation}
In the present case $q=\ell$, the condition \eqref{eq} simplifies to
\begin{equation}\label{eqell}
\begin{split}
2t\sin 2v_1\cos(\phi+\Delta)&=(1-t^2)\cos 2v_1 + (1 + t^2)+\OO_\Lambda(\lambda^2/\ell),\\
2t\sin 2v_2\cos(\phi-\Delta)&=(t^2-1)\cos 2v_2 - (1 + t^2)+\OO_\Lambda(\lambda^2/\ell),
\end{split}
\end{equation}
hence for the proof of \eqref{toprovedelta-a} we can and we shall assume that $|v_1+v_2|\leq\pi/2$. Indeed, replacing $v_1$ by $v_1+\pi$ has the effect of replacing $g$ by $-g$ without altering $\Delta$ or the statement \eqref{toprovedelta-a}. We fix a pair $(\phi,t)\in\mcM$.

The two equations in \eqref{eqell} yield readily that
\[(\sin 2v_j)^2 \leq 2+2\cos 2v_j\ll_\Lambda t^2+t|\sin 2v_j|+\lambda^2/\ell.\]
Hence $\sin 2v_j \ll_\Lambda t+\lambda/\sqrt{\ell}$, that is,
\begin{equation}\label{pi-coset}
v_1, v_2 \in \frac{\pi}{2}\ZZ +\OO_\Lambda\left(t+\frac{\lambda}{\sqrt{\ell}}\right).
\end{equation}
Combining \eqref{eqell} with the Cauchy--Schwarz inequality, we also get
\[ (1+t^2)^2+\OO_\Lambda(\lambda^2/\ell)\leq (1-t^2)^2+4t^2\cos^2(\phi\pm\Delta).\]
Equivalently,
\[\sin(\phi\pm\Delta)\ll_\Lambda \frac{\lambda}{t\sqrt{\ell}}.\]
Using also our initial assumption $|\Delta|\leq\pi/4$, we conclude that
\begin{equation}\label{cos-ess-1}
\Delta\ll_\Lambda\frac{\lambda}{t\sqrt{\ell}}\qquad\text{and}
\qquad\phi\in\pi\ZZ+\OO_\Lambda\left(\frac{\lambda}{t\sqrt{\ell}}\right).
\end{equation}
In particular, $\cos(\phi\pm\Delta)=\epsilon+\OO_\Lambda(\lambda^2/t^2\ell)$ for some $\epsilon\in\{\pm 1\}$. Plugging this back to \eqref{eqell}, and using also \eqref{pi-coset} along with
\[t \left(t+\frac{\lambda}{\sqrt{\ell}}\right) \min\left(1, \frac{\lambda^2}{t^2\ell}\right) \ll \frac{\lambda^2}{\ell},\]
we obtain
\begin{equation}\label{eqell-rewrite}
\begin{split}
2t\epsilon\sin 2v_1&=(1-t^2)\cos 2v_1 + (1 + t^2)+\OO_\Lambda(\lambda^2/\ell),\\
2t\epsilon\sin 2v_2&=(t^2-1)\cos 2v_2 - (1 + t^2)+\OO_\Lambda(\lambda^2/\ell).
\end{split}
\end{equation}
Now consider the following three unit vectors in $\RR^2$:
\[\mbv_1:=(\cos 2v_1,\sin 2v_1),\qquad \mbv_2:=(\cos 2v_2,-\sin 2v_2),\qquad
\mbt:=\left(\frac{t^2-1}{t^2+1},\frac{2t\epsilon}{t^2+1}\right).\]
By \eqref{eqell-rewrite}, the scalar products $\mbv_j\mbt$ are $1+\OO_\Lambda(\lambda^2/\ell)$, hence the directed angles $\arg(\mbv_j)-\arg(\mbt)$ lie in $2\pi\ZZ+\OO_\Lambda(\lambda/\sqrt{\ell})$. It follows that
\[\arg(\mbv_1)-\arg(\mbv_2)\in2\pi\ZZ+\OO_\Lambda(\lambda/\sqrt{\ell}),\]
and then the assumption $|v_1+v_2|\leq\pi/2$ forces that
\begin{equation}\label{eps}
v_1+v_2\ll_\Lambda\lambda/\sqrt{\ell}.
\end{equation}

We are now ready to complete the proof of Theorem~\ref{thm5}\ref{thm5-b}. By \eqref{pi-coset} and \eqref{eps}, there exists a multiple $v$ of $\pi/2$ such that
\begin{align*}
m[v_1]&=m[v]+\OO_\Lambda\bigl(t+\lambda/\sqrt{\ell}\bigr),\\
m[v_2]&=m[-v]+\OO_\Lambda\bigl(t+\lambda/\sqrt{\ell}\bigr),\\
m[v_1+v_2]&=\id+\OO_\Lambda\bigl(\lambda/\sqrt{\ell}\bigr).
\end{align*}
Therefore, using also \eqref{r1} and \eqref{cos-ess-1}, we conclude that
\begin{align*}
&m[v_1]\diag\left(re^{i\Delta},r^{-1}e^{-\Delta}\right)m[v_2]\\
&=m[v_1+v_2]+m[v_1]\diag\left(re^{i\Delta}-1,r^{-1}e^{-i\Delta}-1\right)m[v_2]\\
&=m[v_1+v_2]+m[v]\diag\left(re^{i\Delta}-1,r^{-1}e^{-i\Delta}-1\right)m[-v]
+\OO_\Lambda\bigl(r\lambda/\sqrt{\ell}\bigr)\\
&=m[v]\diag\left(re^{i\Delta},r^{-1}e^{-i\Delta}\right)m[-v]
+\OO_\Lambda\bigl(r\lambda/\sqrt{\ell}\bigr).
\end{align*}
The main term $m[v]\diag\left(re^{i\Delta},r^{-1}e^{-i\Delta}\right)m[-v]$ lies in $\mcD$, hence \eqref{toprovedelta-a} follows.

The proof of Theorem~\ref{thm5}\ref{thm5-b} is complete.

\section{Proof of Theorem~\ref{thm1}}\label{thm1-proof-sec}
In this section, we prove Theorem~\ref{thm1}. Lemma~\ref{APTI-done-lemma}, which results from the amplified pre-trace inequality and estimates on the spherical trace function, proves an estimate on $|\Phi(g)|^2$ for $g\in\Omega$ in terms of the Diophantine counts $M(g,L,\mcL,\vec{\delta})$. We begin with the key remaining step of estimating these counts.

We allow all implied constants within this section to depend on $\Omega$, and we drop the subscript from notation. Moreover, we adopt the notation $A\preccurlyeq B$ to mean that $|A|\ll_{\eps}(\ell L)^{\eps} B$, where $\eps>0$ is fixed but may be taken as small as desired at each step, and the implied constant is allowed to depend on $\eps$. 

For each $\mcL\in\{1,L^2,L^4\}$ and $\vec{\delta}=(\delta_1,\delta_2)$ with $0<\delta_1,\delta_2\leq\ell^\eps$, we will estimate the count $M(g,L,\mcL,\vec{\delta})$ of matrices
\begin{align}\label{detgamma}
\begin{aligned}
&\gamma=\begin{pmatrix}a&b\\c&d\end{pmatrix}\in\MM_2(\ZZ[i]),
\qquad&&\begin{array}{l}\det\gamma=n\in D(L,\mcL),\qquad |n|\asymp\mcL^{1/2},\end{array}\\
&g^{-1}\tilde{\gamma}g=k\begin{pmatrix} z&u\\&z^{-1}\end{pmatrix}k^{-1}
&&\begin{array}{l}\text{for some $k\in K$ such that}\\\text{$|z|\geq 1$, $\min|z\pm 1|\leq\delta_1$, $|u|\leq\delta_2$,}\end{array}
\end{aligned}
\end{align}
where as before $\tilde{\gamma} = \gamma/\sqrt{n}$. By the symmetry $\gamma\leftrightarrow -\gamma$, we can and we shall assume that $|z-1|\leq|z+1|$. Then the conditions imply that both $|z-1|$ and $|z^{-1}-1|$ are at most $\delta_1$, hence
\[\left|\frac{a+d}{\sqrt{n}}-2\right|=|\tr\tilde\gamma-2|=|z+z^{-1}-2|=|z-1||z^{-1}-1|\leq\delta_1^2.\]
On the other hand, since $\|g\|\asymp_{\Omega}1$, we also have that
\[ \|\tilde{\gamma}-\id\|=\left\|gk\begin{pmatrix}z-1&u\\&z^{-1}-1\end{pmatrix}k^{-1}g^{-1}\right\|
\ll\delta_1+\delta_2. \]
Summarizing, we need to estimate the number of matrices $\gamma$ as in \eqref{detgamma} such that
\begin{equation}\label{abcd-conditions}
\left|a+d-2\sqrt{n}\right|\leq\delta_1^2\sqrt{|n|},\qquad |a-d|,|b|,|c|\ll(\delta_1+\delta_2)\sqrt{|n|}.
\end{equation}
In particular, we have $|a+d|\preccurlyeq\sqrt{|n|}$ and
\begin{equation}\label{trace-cor}
(a-d)^2+4bc=(a+d)^2-4n\preccurlyeq\delta_1^2|n|.
\end{equation}

As is often the case, parabolic matrices $\gamma$ (those with trace $\pm 2\sqrt{n}$) play a distinctive role in this counting problem, and we split the count accordingly into the parabolic and non-parabolic subcounts as
\[ M(g,L,\mcL,\vec{\delta})=M^{\pp}(g,L,\mcL,\vec{\delta})+M^{\np}(g,L,\mcL,\vec{\delta}). \]
We shall prove the following result using \eqref{detgamma}, \eqref{abcd-conditions}, and \eqref{trace-cor}.

\begin{lemma}\label{counting-for-thm1}
Let $\Omega\subset G$ be a compact subset, $L\geq 1$, and $\mcL\in\{1,L^2,L^4\}$. For $g\in\Omega$ and $\vec{\delta}=(\delta_1,\delta_2)$ with $0<\delta_1,\delta_2\preccurlyeq 1$, we have the following bounds.
\begin{align}
\label{Mbound1}M(g,L,1,\vec{\delta})&\preccurlyeq_{\Omega}1,\\
\label{Mbound2}M^{\pp}(g,L,\mcL,\vec{\delta})&\preccurlyeq_{\Omega} \mcL^{1/2}+\mcL\delta_2^2,\\
\label{Mbound3}M^{\np}(g,L,L^2,\vec{\delta})&\preccurlyeq_{\Omega} L^4\delta_1^4(\delta_1^2+\delta_2^2),\\
\label{Mbound4}M^{\np}(g,L,L^4,\vec{\delta})&\preccurlyeq_{\Omega} L^6\delta_1^4(\delta_1^2+\delta_2^2).
\end{align}
Moreover,
\begin{equation}\label{otherwise-vanishes}
M^{\np}(g,L,\mcL,\vec{\delta})=0\qquad\text{unless}\qquad\delta_1\succcurlyeq\mcL^{-1/4}.
\end{equation}
\end{lemma}

\begin{proof}
The bound \eqref{Mbound1} is immediate from \eqref{abcd-conditions}. We turn to the bound \eqref{Mbound2}, which counts parabolic matrices $\gamma$. In this case, we have $(a-d)^2+4bc=0$ and $z=1$, hence in particular \eqref{abcd-conditions} holds with $0$ in place of $\delta_1$. If $bc\neq 0$, then there are $\ll\mcL^{1/2}$ choices for $a+d=2\sqrt{n}$, and $\ll\mcL^{1/2}\delta_2^2$ choices for $a-d\neq 0$. The difference $a-d$ determines the product $bc$ uniquely, hence by the divisor bound, there are $\preccurlyeq 1$ choices for $(b,c)$. This is admissible for \eqref{Mbound2}. If $bc=0$, then there are $\ll\mcL^{1/2}$ choices for $a=d=\sqrt{n}$, and $\ll 1+\mcL^{1/2}\delta_2^2$ choices for $(b,c)$. This is again admissible for \eqref{Mbound2}.

From now on we count non-parabolic matrices $\gamma$, in which case $(a-d)^2+4bc\neq 0$.
The statement \eqref{otherwise-vanishes} is immediate from \eqref{trace-cor}, so we are left with proving \eqref{Mbound3} and \eqref{Mbound4}, where we may assume $\delta_1\succcurlyeq\mcL^{-1/4}$. If $bc\neq 0$, then there are $\preccurlyeq \mcL^{1/2}(\delta_1^2+\delta_2^2)$ choices for $a-d$, and, for given $a-d$, there are $\preccurlyeq \mcL\delta_1^4$ choices for $(b,c)$ by \eqref{trace-cor} and the divisor bound. If $bc=0$, then there are $\preccurlyeq\mcL^{1/2}\delta_1^2$ choices for $a-d$ by \eqref{trace-cor}, and  $\preccurlyeq \mcL^{1/2}(\delta_1^2+\delta_2^2)$ choices for $(b,c)$. Altogether, there are $\preccurlyeq\mcL^{3/2}\delta_1^4(\delta_1^2+\delta_2^2)$ choices for the triple $(a-d,b,c)$. In the middle range $\mcL=L^2$, we additionally use that there are $\preccurlyeq\mcL^{1/2}$ choices for $a+d$, whence \eqref{Mbound3} follows. In the high range $\mcL=L^4$, $n=l_1^2l_2^2$ is a square, and $(a-d)^2+4bc\neq 0$ factors as $(a+d+2l_1l_2)(a+d-2l_1l_2)$. Hence the triple $(a-d,b,c)$ in fact determines $a+d$ up to $\preccurlyeq 1$ possibilities by the divisor bound, and \eqref{Mbound4} follows.
\end{proof}

Combining Lemmata~\ref{APTI-done-lemma} and \ref{counting-for-thm1}, we obtain that
\[\sum_{\phi\in\mfB}|\phi(g)|^2\preccurlyeq_{I,\Omega}\ell^3
\left(\frac1L+S^{\pp}(L)+S^{\np}(L,L^2)+S^{\np}(L,L^4)\right)+L^{2}\ell^{-48},\]
where
\begin{alignat*}{3}
S^{\pp}(L)&:=\sum_{\substack{\vec{\delta}\text{ dyadic}\\1/\sqrt{\ell}\leq\delta_j\preccurlyeq 1}}\frac1{\sqrt{\ell}\delta_2}\left(\frac{L+L^2\delta_2^2}{L^3}+\frac{L^2+L^4\delta_2^2}{L^4}\right)&&\preccurlyeq\frac1{L^2}+\frac1{\sqrt{\ell}},\\
S^{\np}(L,L^2)&:=\sum_{\substack{\vec{\delta}\text{ dyadic}\\1/\sqrt{\ell}\leq\delta_j\preccurlyeq 1}}
\frac1{\ell\delta_1^2}\cdot\frac{L^4\delta_1^4(\delta_1^2+\delta_2^2)}{L^3}&&\preccurlyeq\frac{L}{\ell},\\
S^{\np}(L,L^4)&:=\sum_{\substack{\vec{\delta}\text{ dyadic}\\1/\sqrt{\ell}\leq\delta_j\preccurlyeq 1}}\frac1{\ell\delta_1^2}\cdot\frac{L^6\delta_1^4(\delta_1^2+\delta_2^2)}{L^4}&&\preccurlyeq\frac{L^2}{\ell}.
\end{alignat*}
Putting everything together, we conclude that
\[\sum_{\phi\in\mfB}|\phi(g)|^2\preccurlyeq_{I,\Omega}\ell^3
\left(\frac1L+\frac1{\sqrt{\ell}}+\frac{L^2}{\ell}\right)+L^2\ell^{-48}\ll\ell^{8/3}, \]
by making the essentially optimal choice $L:=7\ell^{1/3}$ (which satisfies our earlier
condition $L\geq 7$).

The proof of Theorem~\ref{thm1} is complete.

\section{Proof of Theorem~\ref{thm3}}\label{thm2proof-sec}
In this section, we prove Theorem~\ref{thm3}. For $q= 0$, Lemma~\ref{APTI-done-lemma-single-form} provides an estimate on $|\phi_q(g)|^2$ for $g\in\Omega$ in terms of the Diophantine count $M_0^{\ast}(g,L,\mcL,\delta)$, while for $q = \pm \ell$ we need to analyze $Q(g, L, H_1,H_2)$ as follows from Lemma~\ref{q=ell-case}. We begin by estimating these counts. We keep the notational conventions from \S\ref{thm1-proof-sec}.

\subsection{A comparison lemma}
The Diophantine counts in Lemmata~\ref{APTI-done-lemma-single-form} and \ref{q=ell-case} involve the positioning relative to certain special sets of the matrix $g^{-1}\tilde{\gamma}g$, which we now explicate in preparation for a counting argument. Using $g\in\Omega$, we may write explicitly
\[ g=\begin{pmatrix} g_1&g_2\\g_3&g_4\end{pmatrix},\qquad g_j\ll 1. \]
An explicit calculation shows that
\[ g^{-1}\begin{pmatrix} a&b\\c&d\end{pmatrix}g=\begin{pmatrix}\frac{a+d}{2}+L_1&L_2\\L_3&\frac{a+d}{2}-L_1\end{pmatrix}, \]
where
\begin{equation}\label{coordinates}
\begin{alignedat}{7}
L_1&=\hphantom{-}(a-d)\big(\tfrac12+g_2g_3\big)&&+bg_3g_4&&-cg_1g_2, \\
L_2&=\hphantom{-}(a-d)g_2g_4&&+bg_4^2&&-cg_2^2,\\
L_3&=-(a-d)g_1g_3&&-bg_3^2&&+cg_1^2.
\end{alignedat}
\end{equation}
We record the following simple but effective result, which will be used in both parts of Theorem~\ref{thm3}.

\begin{lemma}\label{l2l3-small}
Let $\Omega\subset G$ be a compact subset, and $g\in\Omega$. Let $a,b,c,d\in\CC$ and $\Delta>0$ be such that
$L_2,L_3\ll\Delta$.
\begin{enumerate}[(a)]
\item\label{1213-a}
For at least one $s\in\{a-d,b,c\}$, we have
\[\begin{bmatrix} a-d&b&c\end{bmatrix}^{\top}
=\begin{bmatrix}\lambda_1&\lambda_2&\lambda_3\end{bmatrix}^{\top}s+\OO(\Delta)\]
with $\lambda_1,\lambda_2,\lambda_3\ll 1$ depending only on $g$.
\item\label{1213-b}
For the same choice of $s\in\{a-d,b,c\}$, we have
\[ (a-d)^2+4bc=\mu s^2+\OO(\Delta |s|+\Delta^2), \]
with $\mu=\lambda_1^2+4\lambda_2\lambda_3\gg 1$. If additionally $(a-d)^2+4bc=0$, then $a-d,b,c\ll\Delta$.
\end{enumerate}
\end{lemma}

\begin{proof}
We may write the defining equations for $L_2$ and $L_3$ as $\begin{bmatrix} L_2&L_3\end{bmatrix}^{\top}=M\begin{bmatrix} a-d&b&c\end{bmatrix}^{\top}$ for a $2\times 3$ matrix $M$ whose $2\times 2$ minors we compute to be
\[ \begin{vmatrix}g_4^2&-g_2^2\\-g_3^2&g_1^2\end{vmatrix}=g_1g_4+g_2g_3,\quad \begin{vmatrix} g_2g_4&-g_2^2\\-g_1g_3&g_1^2\end{vmatrix}=g_1g_2,\quad \begin{vmatrix} g_2g_4&g_4^2\\-g_1g_3&-g_3^2\end{vmatrix}=g_3g_4. \]
At least one of these minors exceeds $1/3$ in absolute value, since
\[(g_1g_4+g_2g_3)^2-4g_1g_2g_3g_4=1.\]
Consider the case when $|g_1g_4+g_2g_3|>1/3$. Then we may solve the latter two equations in  \eqref{coordinates} for $b$, $c$, which yields
\[ \begin{bmatrix} b\\c\end{bmatrix}=\begin{bmatrix}-g_1g_2\\g_3g_4\end{bmatrix}\frac{a-d}{g_1g_4+g_2g_3}+\OO(\Delta).\]
This settles the first claim in the lemma with $s=a-d$. The second claim follows from
\[ (a-d)^2+4bc=\frac{(a-d)^2}{(g_1g_4+g_2g_3)^2}+\OO\big(\Delta|a-d|+\Delta^2\big). \]
The other cases (of which it suffices to consider one) are similar. For example, under $|g_1g_2|>1/3$ we have
\[ \begin{bmatrix}a-d\\c\end{bmatrix}=\begin{bmatrix}-g_1g_4-g_2g_3\\-g_3g_4\end{bmatrix}\frac{b}{g_1g_2}+\OO(\Delta),\quad (a-d)^2+4bc=\frac{b^2}{(g_1g_2)^2}+\OO\big(\Delta|b|+\Delta^2\big), \]
from which the lemma follows.
\end{proof}

\subsection{Second moment count for \texorpdfstring{$q=\pm\ell$}{q=l,-l}}
We will now establish an upper bound for the quantity $Q(g,L, H_1, H_2)$ counting pairs of matrices $(\gamma_1,\gamma_2)$ such that
\begin{equation}\label{thm2a-KS-conditions}
\begin{gathered}
\gamma_j=\begin{pmatrix} a_j&b_j\\c_j&d_j\end{pmatrix}\in\MM_2(\ZZ[i]),
\qquad \det\gamma_1=\det\gamma_2=n,\qquad L \leq |n|\leq 2 L,\\
\|g^{-1}\tilde{\gamma}_jg\|\leq\sqrt{\frac{H_j}{L}},\qquad
\dist(g^{-1}\tilde{\gamma}_jg,\mcD)\ll\sqrt{\frac{H_j\log\ell}{L\ell}}.
\end{gathered}
\end{equation}
We denote the quantities in \eqref{coordinates} corresponding to $\gamma_j$ as
$L_{1j}$, $L_{2j}$, $L_{3j}$. From \eqref{coordinates} and \eqref{thm2a-KS-conditions} we deduce that
\begin{equation}\label{2a-KS-immediate}
\|\gamma_j\|\ll\sqrt{H_j},\qquad L_{2j},L_{3j}\preccurlyeq\sqrt{H_j/\ell},
\end{equation}
and
\begin{equation}\label{det-cond}
(a_1+d_1)^2-(a_2+d_2)^2=(a_1-d_1)^2+4b_1c_1-(a_2-d_2)^2-4b_2c_2.
\end{equation}
We shall prove the following result using \eqref{thm2a-KS-conditions}, \eqref{2a-KS-immediate}, and \eqref{det-cond}.

\begin{lemma}\label{lemma-ell-count}
Let $\Omega\subset G$ be a compact subset and $L\geq 1$. For $g\in\Omega$ and
$1\leq H_1, H_2\preccurlyeq\ell$, we have
\begin{equation}\label{Qbound}
Q(g,L,H_1, H_2)\preccurlyeq_\Omega H_1H_2.
\end{equation}
\end{lemma}

\begin{proof} We shall use that the entries $a_j,b_j,c_j,d_j\in\ZZ[i]$ of each participating $\gamma_j$ satisfy the conditions of Lemma~\ref{l2l3-small} with $\Delta_j\preccurlyeq 1$ in the role of $\Delta$. Indeed, this follows from \eqref{2a-KS-immediate} and $H_1, H_2\preccurlyeq\ell$.

Let $s_j\in\{a_j-d_j,b_j,c_j\}$ be as in Lemma~\ref{l2l3-small}\ref{1213-a}. By Lemma~\ref{l2l3-small}\ref{1213-a} and \eqref{2a-KS-immediate}, for a given pair $(s_1,s_2)$, there are $\preccurlyeq 1$ choices for the two triples $(a_j-d_j,b_j,c_j)$, which then determine both sides of \eqref{det-cond}. Using this preliminary observation, we do the counting in two steps.

First we count $(\gamma_1,\gamma_2)$ satisfying \eqref{thm2a-KS-conditions} and $(a_1+d_1)^2\neq(a_2+d_2)^2$. By \eqref{2a-KS-immediate}, there are $\ll H_1H_2$ choices for the pair $(s_1, s_2)$, hence $\preccurlyeq H_1H_2$ choices for the two triples $(a_j-d_j,b_j,c_j)$. Given the triples, by \eqref{2a-KS-immediate}--\eqref{det-cond} and the divisor bound, there are $\preccurlyeq 1$ choices for $(a_1+d_1,a_2+d_2)$. This is admissible for \eqref{Qbound}.

Now we count $(\gamma_1,\gamma_2)$ satisfying \eqref{thm2a-KS-conditions} and $(a_1+d_1)^2=(a_2+d_2)^2$.
In this case, Lemma~\ref{l2l3-small}\ref{1213-b} coupled with \eqref{2a-KS-immediate}--\eqref{det-cond} shows that
$s_1^2-s_2^2\preccurlyeq\sqrt{H_1}+\sqrt{H_2}$. Hence, by the divisor bound (separating the case when $s_1^2=s_2^2$), there are $\preccurlyeq\max(H_1,H_2)$ choices for the pair $(s_1, s_2)$ and same for the two triples $(a_j-d_j,b_j,c_j)$. Independently of the triples, by \eqref{2a-KS-immediate}, there are $\ll\min(H_1,H_2)$ choices for $(a_1+d_1,a_2+d_2)$. This is again admissible for \eqref{Qbound}.
\end{proof}

\subsection{Interlude: a first moment count}\label{thm2a-first-moment-sec}
For the proof of Theorem~\ref{thm2} in \S \ref{sec-proof2} below, we need a variation of the previous Diophantine argument that is most conveniently stated and proved at this point. 
For $\mcL\in\{1,L^2,L^4\}$ and every $0<\delta\preccurlyeq 1$, we will establish an upper bound on the quantity
\begin{equation}\label{thm2a-conditions}
M_\mcD(g,L,\mcL,\eps,\delta):=\sum_{n\in D(L,\mcL)}
\#\left\{\gamma\in\Gamma_n:\|g^{-1}\tilde\gamma g\|\ll\ell^{\eps},\,\,\dist(g^{-1}\tilde\gamma g,\mcD)\leq\delta\right\},
\end{equation}
where the implied constant is absolute. As before, we conclude from the conditions in \eqref{thm2a-conditions} and the explicit description in \eqref{coordinates} that
\begin{equation}\label{coeff-bound-2a_l2l3-delta0}
\|\gamma\|\preccurlyeq\mcL^{1/4}\qquad\text{and}\qquad L_2,L_3\ll\mcL^{1/4}\delta.
\end{equation}
We shall prove the following result using \eqref{coeff-bound-2a_l2l3-delta0} and the identity
\begin{equation}\label{parabolicidentity}
(a-d)^2+4bc=(a+d)^2-4n.
\end{equation}

\begin{lemma}\label{first-moment-count-lemma}
Let $\Omega\subset G$ be a compact subset, $L\geq 1$, and $\eps>0$. For $g\in\Omega$ and $0<\delta\preccurlyeq 1$, we have the following bounds.
\begin{align}
\label{Nbound1}M_{\mcD}(g,L,1,\eps,\delta)&\preccurlyeq_{\Omega} 1,\\
\label{Nbound2}M_{\mcD}(g,L,L^2,\eps,\delta)&\preccurlyeq_{\Omega} L^2+L^4\delta^4,\\
\label{Nbound3}M_{\mcD}(g,L,L^4,\eps,\delta)&\preccurlyeq_{\Omega} L^2+L^6\delta^4.
\end{align}
\end{lemma}

\begin{proof}
The bound \eqref{Nbound1} corresponds to $\mcL=1$, and it is immediate from \eqref{coeff-bound-2a_l2l3-delta0}. Hence we focus on the bounds \eqref{Nbound2}--\eqref{Nbound3} that correspond to $\mcL\in\{L^2,L^4\}$. We shall use that the entries $a,b,c,d\in\ZZ[i]$ of each participating $\gamma$ satisfy the conditions of Lemma~\ref{l2l3-small} with $\Delta=\mcL^{1/4}\delta$, as follows from \eqref{coeff-bound-2a_l2l3-delta0}.

First we count parabolic matrices $\gamma$. In this case, we have $(a-d)^2+4bc=0$, hence also $a-d,b,c\ll\mcL^{1/4}\delta$ by Lemma~\ref{l2l3-small}\ref{1213-b}. If $bc\neq 0$, then there are $\ll\mcL^{1/2}$ choices for $a+d=\pm 2\sqrt{n}$, and $\ll\mcL^{1/2}\delta^2$ choices for $a-d\neq 0$. The difference $a-d$ determines the product $bc$ uniquely, hence by the divisor bound, there are $\preccurlyeq 1$ choices for $(b,c)$. This is admissible for \eqref{Nbound2}--\eqref{Nbound3}. If $bc=0$, then there are $\ll\mcL^{1/2}$ choices for $a=d=\pm\sqrt{n}$, and $\ll 1+\mcL^{1/2}\delta^2$ choices for $(b,c)$. This is again admissible for \eqref{Nbound2}--\eqref{Nbound3}.

Now we count non-parabolic matrices $\gamma$, in which case $(a-d)^2+4bc\neq 0$. Let $s\in\{a-d,b,c\}$ be as in Lemma~\ref{l2l3-small}\ref{1213-a}. There are $\preccurlyeq\mcL^{1/2}$ choices for $s$, and for a given $s$, there are $\ll 1+\mcL\delta^4$ choices for the triple $(a-d,b,c)$ by Lemma~\ref{l2l3-small}\ref{1213-a}. Altogether, there are $\preccurlyeq \mcL^{1/2}+\mcL^{3/2}\delta^4$ choices for the triple $(a-d,b,c)$. In the middle range $\mcL=L^2$, we additionally use that there are $\preccurlyeq\mcL^{1/2}$ choices for $a+d$, whence \eqref{Nbound2} follows. In the high range $\mcL=L^4$, $n=l_1^2l_2^2$ is a square, and $(a-d)^2+4bc\neq 0$ factors as $(a+d+2l_1l_2)(a+d-2l_1l_2)$. Hence the triple $(a-d,b,c)$ in fact determines $a+d$ up to $\preccurlyeq 1$ possibilities by the divisor bound, and \eqref{Nbound3} follows.
\end{proof}

\subsection{Counting setup for \texorpdfstring{$q=0$}{q=0}}
For each $\mcL\in\{1,L^2,L^4\}$ and $0<\delta\preccurlyeq 1$, we will establish an upper bound on the quantity $M_0^{\ast}(g,L,\mcL,\delta)$ consisting of matrices
\begin{equation}\label{thm2b-conditions}
\begin{gathered}
\gamma=\begin{pmatrix}a&b\\c&d\end{pmatrix}\in\MM_2(\ZZ[i]),
\qquad\det\gamma=n\in D(L,\mcL),\qquad |n|\asymp\mcL^{1/2}\\
\dist(g^{-1}\tilde{\gamma}g,\mcS)\leq\delta,\qquad\frac{D(g^{-1}\tilde\gamma g)}{\|g^{-1}\tilde\gamma g\|^2}\ll\frac{\log\ell}{\sqrt{\ell}}.
\end{gathered}
\end{equation}
From the first distance condition in \eqref{thm2b-conditions} we conclude that
\begin{equation}\label{coeff-bound}
a,b,c,d\preccurlyeq\mcL^{1/4}.
\end{equation}
Using the description in \eqref{coordinates}, the distance conditions in \eqref{thm2b-conditions} imply that
\begin{gather}
\label{cond1} \left\{\begin{aligned} &L_2,L_3\ll\delta\sqrt{|n|}\\ &\left|\tfrac{a+d}{2}\pm L_1\right|=(1+\OO(\delta))\sqrt{|n|}\end{aligned}\right. \qquad\text{or}\qquad
\left\{\begin{aligned} &a+d,L_1\ll\delta\sqrt{|n|}\\ &|L_2|,|L_3|=(1+\OO(\delta))\sqrt{|n|};\end{aligned}\right.\\
\label{cond0} \left|\tfrac{a+d}{2}+L_1\right|^2-\left|\tfrac{a+d}{2}-L_1\right|^2
\preccurlyeq\sqrt{\mcL/\ell}\qquad\text{and}\qquad|L_2|^2-|L_3|^2\preccurlyeq\sqrt{\mcL/\ell}.
\end{gather}
As in \S\ref{thm1-proof-sec}, we split the count into the parabolic and non-parabolic subcounts as
\[ M_0^{\ast}(g,L,\mcL,\delta)=M_0^{\ast\pp}(g,L,\mcL,\delta)+M_0^{\ast\np}(g,L,\mcL,\delta). \]
We shall prove the following result using \eqref{thm2b-conditions}--\eqref{cond0} and \eqref{parabolicidentity}.

\begin{lemma}\label{counting-for-thm2}
Let $\Omega\subset G$ be a compact subset, $L\geq 1$, and $\mcL\in\{L^2,L^4\}$. For $g\in\Omega$ and $0<\delta\preccurlyeq 1$, we have the following bounds.
\begin{align}
\label{Rbound1} M_0^{\ast}(g,L,1,\delta)&\preccurlyeq_{\Omega} 1,\\
\label{Rbound2} M_0^{\ast\pp}(g,L,\mcL,\delta)&\preccurlyeq_{\Omega}\mcL^{1/2}+\mcL\delta^2,\\
\label{Rbound3} M_0^{\ast}(g,L,L^2,\delta)&\preccurlyeq_{\Omega} L^{3/2}+L^3\delta^3+\frac{L^2+L^{7/2}\delta^2}{\sqrt{\ell}}+\frac{L^4\delta^2}{\ell},\\[4pt]
\label{Rbound4} M_0^{\ast\np}(g,L,L^4,\delta)&\preccurlyeq_{\Omega} L^3+L^5\delta^2+\frac{L^4+L^6\delta^2}{\sqrt{\ell}}.
\end{align}
\end{lemma}

\begin{proof}
The bound \eqref{Rbound1} is immediate from \eqref{coeff-bound}. For the proof of \eqref{Rbound2}, we observe that, in the parabolic case, \eqref{cond1} implies $L_2,L_3\ll\mcL^{1/4}\delta$. Indeed, this is clear when the first half of \eqref{cond1} holds. Otherwise, the conditions $a+d=\pm 2\sqrt{n}$ and $a+d\ll\delta\sqrt{|n|}$ force $\delta\gg 1$, so the claimed bound is clear again. Applying Lemma~\ref{l2l3-small}\ref{1213-b}, we infer that $a-d,b,c\ll\mcL^{1/4}\delta$ holds in the parabolic case. From here \eqref{Rbound2} follows readily, as in the second paragraph of the proof of Lemma~\ref{first-moment-count-lemma}. Finally, we shall prove \eqref{Rbound3} and \eqref{Rbound4} in the next two subsections.
\end{proof}

\subsection{Volume argument}\label{subsec:volume-argument}
Here, we present a volume argument that we will use repeatedly to estimate the number of lattice points satisfying \eqref{thm2b-conditions}--\eqref{cond0}. The symbol $\vol$ will refer to the Lebesgue measure in $\CC^m\simeq\RR^{2m}$, with $m$ being clear from the context.

The explicit expressions for the linear forms in \eqref{coordinates} may be rewritten as 
\begin{equation}\label{Ag3}
\begin{bmatrix} L_1&L_2&L_3\end{bmatrix}^{\top}=A_0(g)\begin{bmatrix}a-d&b&c\end{bmatrix}^{\top},
\end{equation}
where $A_0:\Omega\to\GL_3(\CC)$ is a continuous function.
It is straightforward to verify that $\det A_0(g)=1/2$ holds identically. 
We shall also use the $4$-dimensional variant
\begin{equation}\label{Ag4}
\begin{bmatrix} a+d&L_1&L_2&L_3\end{bmatrix}^{\top}=\diag(1,A_0(g))\begin{bmatrix} a+d&a-d&b&c\end{bmatrix}^{\top}.
\end{equation}

Now, let $m\geq 1$ be a fixed integer ($m\in\{2,3,4\}$ in our applications), and let $A:\Omega\to\GL_m(\CC)$ be a fixed continuous function. As $\Omega$ is compact, there exists a fixed compact subset $K=K(A,\Omega)\subset\CC^m$ such that each $2m$-dimensional lattice $A(g)\ZZ[i]^m\subset\CC^m$ ($g\in\Omega)$ has a fundamental parallelepiped lying in $K$ and of volume $\asymp 1$. It follows by a standard volume argument that for any compact subset $V\subset\CC^m$ and $g\in\Omega$ we have
\begin{equation}\label{volume-bound}
\#\bigl(V\cap A(g)\ZZ[i]^m\bigr)\ll\vol V^\bullet\qquad\text{where}\qquad V^\bullet:=V+K.
\end{equation}

We also record for repeated reference a simple volume computation. For $r,\Delta>0$, we define the sets
\begin{align*}
W_1(r,\Delta)&:=\bigl\{(z_1,z_2)\in\CC^2:|z_1|,|z_2|\leq r,\,\,\Re(z_1\ov{z_2})\leq\Delta\bigr\},\\
W_2(r,\Delta)&:=\bigl\{(z_1,z_2)\in\CC^2:|z_1|,|z_2|\leq r,\,\,\bigl||z_1|^2-|z_2|^2\bigr|\leq\Delta\bigr\}.
\end{align*}
Cutting these into two parts according to whether $|z_2|\leq |z_1|$ or $|z_2|>|z_1|$, we obtain readily by Fubini's theorem that
\[\vol W_j(r,\Delta)\ll\min(r^4,r^2\Delta).\]
On the other hand, we have
\[W_j(r,\Delta)^\bullet\subset W_j\bigl(r+\OO(1),\Delta+\OO(r+1)\bigr)\]
with implied constants depending only on $A$ and $\Omega$, hence
\begin{equation}\label{volW}
\vol W_j(r,\Delta)^{\bullet}\ll\min\bigl((r+1)^4,(r+1)^2(\Delta+r+1)\bigr)\ll 1+r^2\Delta+r^3.
\end{equation}

\subsection{Middle and high range for \texorpdfstring{$q=0$}{q=0}}
We now estimate the count $M_0^{\ast}(g,L,\mcL,\delta)$ in the ``middle range'' $\mcL=L^2$ and the ``high range'' $\mcL=L^4$. In the high range, we shall focus on the non-parabolic contribution $M_0^{\ast\np}(g,L,L^4,\delta)$, since 
we have already proved \eqref{Rbound2}, and here we shall profit substantially from the fact that $\det\gamma$ is a square.

\subsubsection{Middle range}
In the middle range $\mcL=L^2$, we estimate the number of choices in $M_0^{\ast}(g,L,L^2,\delta)$ as follows.

For the case when the first half of \eqref{cond1} holds, we introduce the set
\begin{align*}
V_1(\delta):=\big\{(z_0,z_1,z_2,z_3)\in\CC^4:\ &z_0,z_1\preccurlyeq\mcL^{1/4},\,\,
\Re(z_0\ov{z_1})\preccurlyeq\sqrt{\mcL/\ell},\\
&z_2,z_3\ll\mcL^{1/4}\delta,\,\,|z_2|^2-|z_3|^2\preccurlyeq \sqrt{\mcL/\ell}\big\},
\end{align*}
suppressing from notation the dependence implicit in $\preccurlyeq$. Then we have by \eqref{volW}
\begin{align}
\nonumber\vol V_1(\delta)^{\bullet}&\preccurlyeq
\vol W_1(\mcL^{1/4},\sqrt{\mcL/\ell})^{\bullet}\cdot 
\vol W_2(\mcL^{1/4}\delta,\sqrt{\mcL/\ell})^{\bullet}\\
\label{V1bound}&\ll (\mcL^{3/4}+\mcL/\sqrt{\ell})(1+\mcL^{3/4}\delta^3+\mcL\delta^2/\sqrt{\ell}).
\end{align}
For the case when the second half of \eqref{cond1} holds, we introduce the set
\begin{align*}
V_2(\delta)=\big\{(z_0,z_1,z_2,z_3)\in\CC^4:\ &z_0,z_1\ll\mcL^{1/4}\delta,\,\,
\Re(z_0\ov{z_1})\preccurlyeq\sqrt{\mcL/\ell},\\
&z_2,z_3\preccurlyeq\mcL^{1/4},\,\,|z_2|^2-|z_3|^2\preccurlyeq\sqrt{\mcL/\ell}\big\},
\end{align*}
suppressing from notation the dependence implicit in $\preccurlyeq$. Then we have by \eqref{volW}
\begin{align}
\nonumber\vol V_2(\delta)^{\bullet}&\preccurlyeq
\vol W_1(\mcL^{1/4}\delta,\sqrt{\mcL/\ell})^{\bullet}\cdot 
\vol W_2(\mcL^{1/4},\sqrt{\mcL/\ell})^{\bullet}\\
\label{V2bound}&\ll (\mcL^{3/4}+\mcL/\sqrt{\ell})(1+\mcL^{3/4}\delta^3+\mcL\delta^2/\sqrt{\ell}).
\end{align}

Using \eqref{thm2b-conditions}--\eqref{cond0}, \eqref{Ag4}--\eqref{volume-bound}, and \eqref{V1bound}--\eqref{V2bound}, we conclude \eqref{Rbound3} in the form
\[ M_0^{\ast}(g,L,L^2,\delta)\preccurlyeq (L^{3/2}+L^2/\sqrt{\ell})(1+L^{3/2}\delta^3+L^2\delta^2/\sqrt{\ell}). \]

\subsubsection{High range}
As in the proof of Lemmata~\ref{counting-for-thm1} and \ref{first-moment-count-lemma}, in the high range $\mcL=L^4$, once the triple $(a-d,b,c)$ is determined for a non-parabolic matrix $\gamma$ (so that \eqref{parabolicidentity} holds), $a+d$ and along with it $\gamma$ is determined up to $\preccurlyeq 1$ choices by the divisor bound, using that $n=l_1^2l_2^2$ is a square. We now estimate the number of choices in $M_0^{\ast\np}(g,L,L^4,\delta)$ as follows.

For the case when the first half of \eqref{cond1} holds, we introduce the set
\[ V_3(\delta):=\big\{(z_1,z_2,z_3)\in\CC^3:
z_1\preccurlyeq\mcL^{1/4},\,\,z_2,z_3\ll\mcL^{1/4}\delta,\,\,|z_2|^2-|z_3|^2\preccurlyeq\sqrt{\mcL/\ell}\big\}, \]
suppressing from notation the dependence implicit in $\preccurlyeq$. Then we have by \eqref{volW}
\begin{equation}\label{V3bound}
\vol V_3(\delta)^{\bullet}\preccurlyeq
\sqrt{\mcL}\cdot\vol W_2(\mcL^{1/4}\delta,\sqrt{\mcL/\ell})^{\bullet}
\ll \sqrt{\mcL}(1+\mcL\delta^2/\sqrt{\ell}+\mcL^{3/4}\delta^3).
\end{equation}
For the case when the second half of \eqref{cond1} holds, we introduce the set
\[ V_4(\delta):=\big\{(z_1,z_2,z_3)\in\CC^3:z_1\ll\mcL^{1/4}\delta,\,\,z_2,z_3\preccurlyeq\mcL^{1/4},\,\,|z_2|^2-|z_3|^2\preccurlyeq\sqrt{\mcL/\ell}\big\}, \]
suppressing from notation the dependence implicit in $\preccurlyeq$. Then we have by \eqref{volW}
\begin{equation}\label{V4bound}
\vol V_4(\delta)^{\bullet}\preccurlyeq
(1+\mcL^{1/4}\delta)^2\cdot\vol W_2(\mcL^{1/4},\sqrt{\mcL/\ell})^{\bullet}
\ll(1+\sqrt{\mcL}\delta^2)(\mcL^{3/4}+\mcL/\sqrt{\ell}).
\end{equation}

Using \eqref{parabolicidentity}, \eqref{thm2b-conditions}--\eqref{cond0}, \eqref{Ag3}, \eqref{volume-bound}, and \eqref{V3bound}--\eqref{V4bound}, we conclude \eqref{Rbound4} in the form
\[ M_0^{\ast\np}(g,L,L^4,\delta)\preccurlyeq
L^2(1+L^4\delta^2/\sqrt{\ell}+L^3\delta^3)+(1+L^2\delta^2)(L^3+L^4/\sqrt{\ell}). \]

The proof of Lemma~\ref{counting-for-thm2} is complete.

\subsection{Proof of Theorem~\ref{thm3}}

In the case $q=0$, we combine Lemmata~\ref{APTI-done-lemma-single-form} and \ref{counting-for-thm2} to see that
\[ |\phi_0(g)|^2\preccurlyeq_{I,\Omega}\ell^2
\left(\frac1L+S^{\ast}_0(L,L^2)+S^{\ast}_0(L,L^4)\right)+L^{2}\ell^{-48}, \]
where
\begin{alignat*}{3}
S_0^{\ast}(L,L^2)&:=\sum_{\substack{\delta\text{ dyadic}\\1/\sqrt{\ell}\leq\delta\preccurlyeq 1}}
\frac1{\sqrt{\ell}\delta L^3}
\left(L^{3/2}+L^3\delta^3+\frac{L^2+L^{7/2}\delta^2}{\sqrt{\ell}}+\frac{L^4\delta^2}{\ell}\right)
&&\preccurlyeq\frac1{L^{3/2}}+\frac1{\sqrt{\ell}}+\frac{L}{\ell^{3/2}},\\
S_0^{\ast}(L,L^4)&:=\sum_{\substack{\delta\text{ dyadic}\\1/\sqrt{\ell}\leq\delta\preccurlyeq 1}}
\frac1{\sqrt{\ell}\delta L^4}
\left(L^3+L^5\delta^2+\frac{L^4+L^6\delta^2}{\sqrt{\ell}}\right)
&&\preccurlyeq\frac1L+\frac{L}{\sqrt{\ell}}+\frac{L^2}{\ell}.
\end{alignat*}
Putting everything together, we conclude that
\[ |\phi_0(g)|^2\preccurlyeq_{I,\Omega}\ell^2\left(\frac1L+\frac{L}{\sqrt{\ell}}+\frac{L^2}{\ell}\right)+L^2\ell^{-48}\ll \ell^{7/4}, \]
by making the essentially optimal choice $L:=7\ell^{1/4}$ (which satisfies our earlier
condition $L\geq 7$).

The case $q = \pm \ell$ is immediate from Lemmata~\ref{q=ell-case} and \ref{lemma-ell-count}, hence the proof of Theorem~\ref{thm3} is complete.

\section{Proof of Theorem~\ref{thm2}}\label{sec-proof2}
In this section, we prove Theorem~\ref{thm2}. Here we take the aim of the softest possible proof based on the localization properties of the averaged spherical trace function (proved in Theorem~\ref{thm6} and then encoded in the form of the amplified pre-trace inequality in Lemma~\ref{APTI-done-lemma-single-form}) and the already available ingredients for the counting problem.

For each $\mcL\in\{1,L^2,L^4\}$ and $\vec{\delta}=(\delta_1,\delta_2)$ with $0<\delta_1,\delta_2\leq\ell^\eps$, the count $M^\ast(g,L,\mcL,\vec{\delta})$ in Lemma~\ref{APTI-done-lemma-single-form} may be estimated in a split fashion as
\[ M^\ast(g,L,\mcL,\vec{\delta})\leq \min\big(M_K(g,L,\mcL,\delta_1),M_\mcD(g,L,\mcL,\eps,\delta_2)\big), \]
where
\[M_K(g,L,\mcL,\delta):=\sum_{n\in D(L,\mcL)}\#\left\{\gamma\in\Gamma_n:\dist\left(g^{-1}\tilde\gamma g,K\right)\leq\delta\right\},\]
and $M_{\mcD}(g,L,\mcL,\eps,\delta)$ is as in \eqref{thm2a-conditions}. The quantity $M_K(g,L,\mcL,\delta)$ is the classical Diophantine count in the spherical sup-norm problem in the eigenvalue aspect, which in the present context
was treated in detail in \cite{BlomerHarcosMilicevic2016}. In the notation of that paper, we have:
\begin{itemize}
\item $u(\tilde\gamma gK,gK)\asymp\dist(g^{-1}\tilde\gamma g,K)^2$ in \cite[(5.3)]{BlomerHarcosMilicevic2016};
\item $N=1$, and $r\asymp_{\Omega}1$ for $g\in\Omega$, in \cite[(6.2)]{BlomerHarcosMilicevic2016}.
\end{itemize}
Thus the count $M_K(g,L,\mcL,\delta_1)$ agrees with $M(gK,L,\mcL,\OO(\delta_1^2))$ in \cite[(5.17)--(5.18)]{BlomerHarcosMilicevic2016}. Importing estimates \cite[(7.1), (7.2), (7.5), (11.1), (11.6)]{BlomerHarcosMilicevic2016}, we conclude that
\[M_K(g,L,1,\delta_1)\preccurlyeq_{\Omega} 1,\quad
M_K(g,L,L^2,\delta_1)\preccurlyeq_{\Omega} L^2+L^4\delta_1,\quad
M_K(g,L,L^4,\delta_1)\preccurlyeq_{\Omega} L^3+L^6\delta_1.\]

The count $M_{\mcD}(g,L,\mcL,\eps,\delta)$ was estimated in Lemma~\ref{first-moment-count-lemma}. Combining everything, we obtain the following lemma.

\begin{lemma}\label{soft-counting-for-general-q}
For $g\in\Omega$, $L>0$, and arbitrary $\eps>0$ and $\vec{\delta}=(\delta_1,\delta_2)$ with $0<\delta_{j}\preccurlyeq 1$, the quantity $M^\ast(g,L,\mcL,\vec{\delta})$ in Lemma~\ref{APTI-done-lemma-single-form} satisfies
\begin{align*}
M^\ast(g,L,1,\vec{\delta})&\preccurlyeq_{\Omega} 1,\\
M^\ast(g,L,L^2,\vec{\delta})&\preccurlyeq_{\Omega}\min\big(L^2+L^4\delta_1,L^2+L^4\delta_2^4\big),\\
M^\ast(g,L,L^4,\vec{\delta})&\preccurlyeq_{\Omega}\min\big(L^3+L^6\delta_1,L^2+L^6\delta_2^4\big).
\end{align*}
\end{lemma}

We are now ready for the proof of Theorem~\ref{thm2}. From Lemma~\ref{soft-counting-for-general-q}, we have  for every pair $\vec{\delta}=(\delta_1,\delta_2)$ with $0<\delta_1,\delta_2\leq\ell^{\eps}$ that
\[\frac{M^\ast(g,L,1,\vec{\delta})}{L}+\frac{M^\ast(g,L,L^2,\vec{\delta})}{L^3}+\frac{M^\ast(g,L,L^4,\vec{\delta})}{L^4}
\preccurlyeq_{\Omega} \left(\frac1L+L^2\min\left(\delta_1,\delta_2^4\right)\right).\]
Inserting this into Lemma~\ref{APTI-done-lemma-single-form}, we find that
\begin{align*}
|\phi_q(g)|^2&\preccurlyeq_{I,\Omega}\ell^{2}
\sum_{\substack{\vec{\delta}\textnormal{ dyadic},\,\,\delta_j\preccurlyeq 1\\\delta_1^2\delta_2\geq 1/\sqrt{\ell}}}
\frac{1}{\sqrt{\ell}\delta_1^2\delta_2}
\left(\frac1L+ L^2 \min\left(\delta_1,\delta_2^4\right)\right)+L^{2}\ell^{-48}\\
&\preccurlyeq \ell^{2} \Biggl(\frac1L+
\sum_{\substack{\vec{\delta}\textnormal{ dyadic},\,\,\delta_j\preccurlyeq 1\\\delta_1^2\delta_2\geq 1/\sqrt{\ell}}}
L^2 \min\left(\frac{1}{\sqrt{\ell}\delta_1\delta_2}, \delta_1, \delta_2^4\right)\Biggr)
\preccurlyeq \ell^{2} \left(\frac1L + \frac{L^2}{\ell^{2/9}}\right),
\end{align*}
where we used $\min(A, B, C) \leq A^{4/9}B^{4/9} C^{1/9}$ in the last step.
The choice $L:=7\ell^{2/27}$ is optimal up to a constant, and it satisfies our earlier
condition $L\geq 7$, hence we obtain Theorem~\ref{thm2} in the form
\[{\|\phi_q|_{\Omega}\|}_\infty\preccurlyeq_{I,\Omega}\ell^{26/27}.\qedhere \]

The proof of Theorem~\ref{thm2} is complete.

\bibliographystyle{amsalpha}
\bibliography{supnormK}

\end{document}